\documentclass{article}

\usepackage[table]{xcolor}
\usepackage{ragged2e, array}
\usepackage[colorlinks,linkcolor=magenta,citecolor=blue, pagebackref=true]{hyperref}
\usepackage[margin=1.2in]{geometry}
\usepackage[utf8]{inputenc} 
\usepackage[T1]{fontenc}    
\usepackage{url}            
\usepackage{booktabs}       
\usepackage{nicefrac}       
\usepackage{microtype}      
\usepackage[ruled,vlined]{algorithm2e}
\usepackage[most]{tcolorbox}
\usepackage{amsmath,amsfonts,amssymb,mathtools}
\usepackage{amsthm}
\usepackage{array}
\usepackage{multirow}
\usepackage[shortlabels]{enumitem}
\usepackage[numbers]{natbib}
\usepackage{tikz}
\usepackage{fp}

\newcommand{\RR}{\mathbb{R}}
\newcommand{\NN}{\mathbb{N}}
\newcommand{\EE}{\mathtt{E}}

\newcommand{\BI}{\mathfrak B}

\newcommand{\1}{\mathbf{1}}

\newcommand{\Lcal}{\mathcal{L}}

\newcommand{\Fcal}{\mathcal{F}}

\newcommand{\Pcal}{\mathcal{P}}

\newcommand{\Kcal}{\mathcal{K}}
\newcommand{\Mcal}{\mathcal{M}}
\newcommand{\Scal}{\mathcal{S}}
\newcommand{\Qcal}{\mathcal{Q}}

\newcommand{\HS}{\mathrm{HS}}

\newcommand{\PM}{\mathrm{PrPl}}
\newcommand{\PMEB}{\mathrm{PrPl\text{-}EB}}
\newcommand{\PMH}{\mathrm{PrPl\text{-}H}}
\newcommand{\PMpm}{\mathrm{PrPl\pm}}
\newcommand{\VAEB}{{\mathrm{PrPl\text{-}EB}(n)}}
\newcommand{\dKelly}{\mathrm{dKelly}}
\newcommand{\gKelly}{\mathrm{gKelly}}
\newcommand{\hgKelly}{\mathrm{hgKelly}}
\newcommand{\SOS}{\mathrm{GRAPA}}
\newcommand{\aSOS}{\mathrm{aGRAPA}}
\newcommand{\WoR}{\mathrm{WoR}}
\newcommand{\HWoR}{\mathrm{H\text{-}\WoR}}
\newcommand{\EBWoR}{\mathrm{EB\text{-}\WoR}}
\newcommand{\EL}{\mathrm{EL}}
\newcommand{\ONS}{\mathrm{O}}
\newcommand{\LBOW}{\mathrm{L}}
\newcommand{\CB}{\mathrm{CB}}
\newcommand{\SRP}{\mathrm{SRP}}
\newcommand{\CBWoR}{\mathrm{CB}\text{-}\WoR}
\newcommand{\G}{\mathrm{G}}
\newcommand{\Var}{\mathrm{Var}}

\DeclareMathOperator*{\argmin}{argmin}
\DeclareMathOperator*{\argmax}{argmax}

\newtheorem{lemma}{Lemma}
\newtheorem{theorem}{Theorem}
\newtheorem{proposition}{Proposition}

\newtheorem{corollary}{Corollary}
\newtheorem{remark}{Remark}
\usepackage{comment}
\usepackage{graphicx} 
\usepackage{caption}
\usepackage{subcaption}
\usepackage{enumitem}

\usepackage{setspace}

\usepackage[english]{babel}

\makeatletter
\renewcommand\tableofcontents{%
    \@starttoc{toc}%
}
\makeatother

\newcommand{\myqed}{\phantom\qedhere\qed}

\newcommand{\myproofname}[1]{Proof of #1}

\title{Estimating means of bounded random variables by betting}
\author{%
Ian Waudby-Smith$^1$ and Aaditya Ramdas$^{12}$\vspace{0.05in}\\
  Departments of Statistics$^1$ and Machine Learning$^2$\\
  Carnegie Mellon University\\
  \texttt{\{ianws, aramdas\}@cmu.edu} \\
}

\begin{document}

\maketitle

\begin{abstract}
  This paper derives confidence intervals (CI) and time-uniform confidence sequences (CS) for the classical problem of estimating an unknown mean from bounded observations. We present a general approach for deriving concentration bounds, that can be seen as a generalization and improvement of the celebrated Chernoff method. At its heart, it is based on a class of composite nonnegative martingales, with strong connections to testing by betting and the method of mixtures. We show how to extend these ideas to sampling without replacement, another heavily studied problem. In all cases, our bounds are adaptive to the unknown variance, and empirically vastly outperform existing approaches based on Hoeffding or empirical Bernstein inequalities and their recent supermartingale generalizations by~\citet{howard_uniform_2019}. In short, we establish a new state-of-the-art for four fundamental problems: CSs and CIs for bounded means, when sampling with and without replacement.
\end{abstract}

\tableofcontents

\section{Introduction}
\label{sec:introduction}

This work presents a new approach to two fundamental problems: (Q1) how do we produce a confidence interval for the mean of a distribution with (known) bounded support using $n$ independent observations? (Q2) given a fixed list of $N$ (nonrandom) numbers with known bounds, how do we produce a confidence interval for their mean by sampling $n \leq N$ of them without replacement in a random order? We work in a nonasymptotic and nonparametric setting, meaning that we do not employ asymptotics or parametric assumptions. Both (Q1) and (Q2) are well studied questions in probability and statistics, but we bring new conceptual tools to bear, resulting in state-of-the-art solutions to both. 

We also consider sequential versions of these problems where observations are made one-by-one; we derive time-uniform confidence sequences, or equivalently, confidence intervals that are valid at arbitrary stopping times. In fact, we first describe our techniques in the sequential regime, because the employed proof techniques naturally lend themselves to this setting. We then instantiate the derived bounds for the more familiar setting of a fixed sample size when a batch of data is observed all at once. Our supermartingale techniques can be thought of as generalizations of classical methods for deriving concentration inequalities, but we prefer to present them in the language of betting, since this is a more accurate reflection of the authors' intuition.

Arguably the most famous concentration inequality for bounded random variables was derived by~\citet{hoeffding_probability_1963}. What is now referred to as ``Hoeffding's inequality'' was in fact improved upon in the same paper where he derived a Bernoulli-type upper bound on the moment generating function of bounded random variables \citep[Equation (3.4)]{hoeffding_probability_1963}. While these bounds are already reasonably tight in a worst-case sense, the resulting confidence intervals do not adapt to non-Bernoulli distributions with lower variance. Inequalities by ~\citet{bennett_probability_1962}, \citet{bernstein_theory_1927} and \citet{bentkus2004hoeffding} improve upon Hoeffding's, but such improvements require knowledge of nontrivial upper bounds on the variance. This led to the development of so-called ``empirical Bernstein inequalities'' by~\citet{audibert2007tuning}and~\citet{maurer_empirical_2009}, which outperform Hoeffding's method for low-variance distributions at large sample sizes by estimating the variance from the data. Our new, and arguably quite simple, approaches to developing bounds significantly outperform these past works (e.g.~Figure~\ref{fig:comparisonAll}).\footnote{\href{https://github.com/WannabeSmith/betting-paper-simulations}{github.com/wannabesmith/betting-paper-simulations} has code to reproduce figures. 
The \texttt{betting} module of the Python package in \href{https://github.com/gostevehoward/confseq}{github.com/gostevehoward/confseq} has the main algorithms, but the package also contains implementations from other papers.} 
We also show that the same conceptual (betting) framework extends to without-replacement sampling, resulting in significantly tighter bounds than classical ones by \citet{serfling1974probability}, improvements by \citet{bardenet2015concentration} and previous state-of-the-art methods due to~\citet{waudby2020confidence}.

For providing intuition, our approach can be described in words as follows:
\textit{If we are allowed to repeatedly bet against the mean being $m$, and if we make a lot of money in the process, then we can safely exclude $m$ from the confidence set.}
The rest of this paper makes the above claim more precise by showing smart, adaptive strategies for (automated) betting, quantifying the phrase ``a lot of money'', and explaining why such an exclusion is mathematically justified. 
At the risk of briefly losing the unacquainted reader, here is a slightly more detailed high-level description:
\begin{quote}
    For each $m \in [0,1]$, we set up a ``fair'' multi-round game of statistician against nature whose payoff rules are such that if the true mean happened to equal $m$, then the statistician can neither gain nor lose wealth in expectation (their wealth in the $m$-th game is a nonnegative martingale), but if the mean is not $m$, then it is possible to bet smartly and make money. Each round involves the statistician making a bet on the next observation, nature revealing the observation and giving the appropriate (positive or negative) payoff to the statistician. The statistician then plays all these games (one for each $m$) in parallel, starting each with one unit of wealth, and possibly using a different, adaptive, betting strategy in each. The $1-\alpha$ confidence set at time $t$ consists of all $m\in[0,1]$ such that the statistician's money in the corresponding game has not crossed $1/\alpha$. The true mean $\mu$ will be in this set with high probability.
\end{quote}
Our choice of language above stems from a game-theoretic approach towards probability, as developed in the books by \citet{shafer_probability_2005,shafer2019game} and a recent paper by \citet{shafer2019language}, but from a purely mathematical viewpoint, our results are extensions of a unified supermartingale approach towards nonparametric concentration and estimation described in \citet{howard_exponential_2018,howard_uniform_2019}; related supermartingale approaches were studied by~\cite{kaufmann2018mixture}, \cite{jun2019parameter}. We elaborate on this viewpoint in Section~\ref{subsec:connections-to-betting}. The most directly related works to our own are by~\cite{hendriks2018test}, whose preprint has initial explorations of methods similar to ours for with-replacement sequential testing and estimation, and~\cite{stark2020sets}, who credits Kaplan for a computationally intractable variant of our approach for sequential testing in the without-replacement case. Apart from several novel results, the present paper extends these past works in \emph{depth, breadth and unity}: our work contains a deeper empirical and theoretical investigation from statistical and computational viewpoints, places our work in a broader context of related work in both settings, and unifies the with- and without-replacement methodology for both testing and estimation in both fixed-time and sequential settings.

We now have the appropriate context for a concrete formalization of our problem, which is slightly more general than introduced above. After that, we describe the game, why the rules of engagement result in valid statistical inference, and derive computationally and statistically efficient betting strategies. 

\begin{figure}[!htbp]
\centering \textbf{Time-uniform confidence sequences}
    \includegraphics[width=0.9\textwidth]{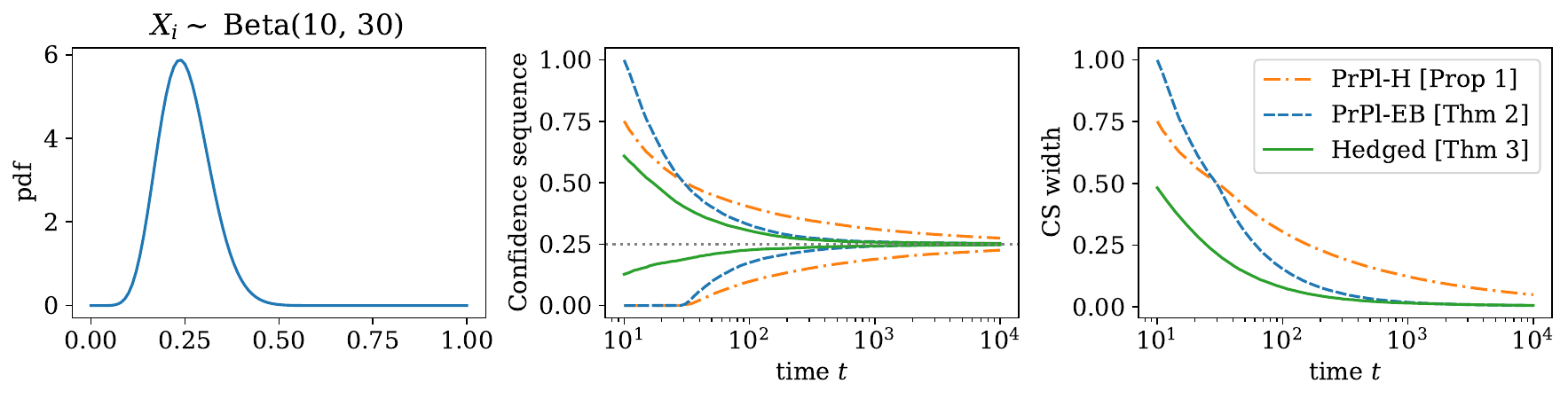}
 \hfill
 \centering \textbf{Fixed-time confidence intervals}
    \centering
    \includegraphics[width=0.9\textwidth]{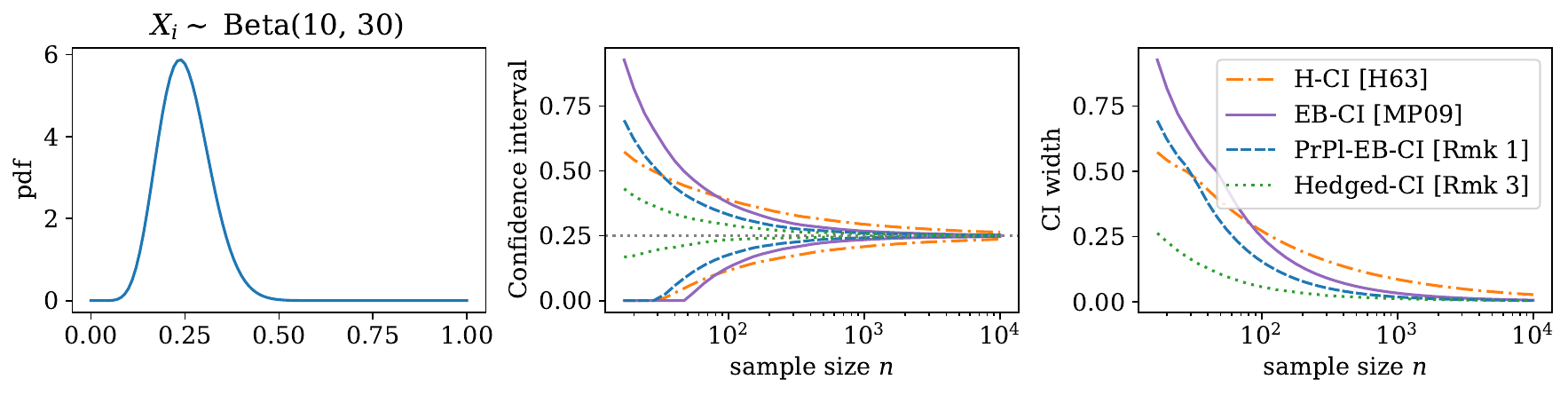}
 \caption{Time-uniform 95\% confidence sequences (upper row) and fixed-time 95\% confidence intervals (lower row) for the mean of independent and idenically distributed (iid) draws from a Beta(10, 30) distribution (unknown to the methods). The betting approaches (Hedged and Hedged-CI) adapt to both the small variance and asymmetry of the data, outperforming the other methods. For a detailed empirical comparison under a larger variety of settings, see Section~\ref{section:simulations}; for additional comparisons under non-iid data, see Section~\ref{section:noniid}.
 }
 \label{fig:comparisonAll}
\end{figure}


\paragraph{Outline.} We summarize the broad approach
 in  Section~\ref{sec:setup}.
As a warmup, we derive a new predictable plug-in method for deriving confidence sequences using exponential supermartingales (Section~\ref{section:warmup}), which already leads to computationally efficient and visually appealing empirical Bernstein confidence intervals and sequences. We then further improve on the aforementioned methods by developing a new martingale approach to deriving time-uniform and fixed-time confidence sets for means of bounded random variables, and connect the developed ideas to betting (Section~\ref{section:capitalProcess}). Section~\ref{section:howToBet} discusses some principles to derive powerful betting strategies to obtain tight confidence sets.
We then show how our techniques also extend to sampling without replacement (Section~\ref{sec:WoR}). 
Revealing simulations are performed along the way to demonstrate the efficacy of the new methods, with a more extensive comparison with past work in Section~\ref{section:simulations}. Section~\ref{section:history-short} summarizes how betting ideas have shaped mathematics, outside of our paper's focus on statistical inference.
We postpone proofs to Section~\ref{sec:app-proofs} and further theoretical insights to Section~\ref{sec:app-theory}.

\section{Concentration inequalities via nonnegative supermartingales}\label{sec:setup}
To set the stage, let $\Qcal^m$ be the set of all distributions on $[0,1]$, where each distribution has mean $m$.
Note that $\Qcal^m$ is a convex set of distributions and it has no common dominating measure, since it consists of both discrete and continuous distributions. 

Consider the setting where we observe a (potentially infinite) sequence of $[0, 1]$-valued random variables with conditional mean $\mu$ for some unknown $\mu \in [0,1]$. We write this as $(X_t)_{t=1}^\infty \sim P$ for some $P \in \Pcal^\mu$, where $\Pcal^\mu$ is the set of all distributions $P$ on $[0, 1]^\infty$ such that $\EE_P(X_t \mid X_1, \dots, X_{t-1}) = \mu$. This includes familiar settings such as independent observations, where $X_i \sim Q_i \in \mathcal Q^\mu$, or i.i.d.\ observations where all $Q_i$'s are identical, but captures more general settings where the conditional distribution of $X_t$ given the past is an element of $\Qcal^\mu$. When one only observes $n$ outcomes, it suffices to imagine throwing away the rest, so that in what follows, we avoid new notation for distributions $P$ over finite length sequences.

We are interested in deriving tight confidence sets for $\mu$, typically intervals, with no further assumptions. Specifically, for a given error tolerance $\alpha \in (0, 1)$, a $(1-\alpha)$ confidence interval (CI) is a random set $C_n \equiv C(X_1,\dots,X_n) \subseteq [0,1]$ such that 
\begin{equation}
\small
\forall n \geq 1, \inf_{P \in \Pcal^\mu} P(\mu \in C_n) \geq 1-\alpha.
\end{equation}
As mentioned earlier, the inequality by~\cite{hoeffding_probability_1963} implies that we can choose
\begin{equation}
\label{eq:hoeffding}
\small
C_n := \left(\overline X_n \pm \sqrt{\frac{\log(2/\alpha)}{2 n}} \right) \cap [0,1].
\end{equation}
Above, we write $(a\pm b)$ to mean $(a-b,a+b)$ for brevity.

This inequality is derived by what is now known as the Chernoff method~\citep{boucheron_concentration_2013}, involving an analytic upper bound on the moment generating function of a bounded random variable.
However, we will proceed differently; we adopt a hypothesis testing perspective, and couple it with a generalization of the Chernoff method. 
As mentioned in the introduction, we first consider the sequential regime where data are observed one after another over time, since nonnegative supermartingales --- the primary mathematical tools used throughout this paper --- naturally arise in this setup. As we will see,  these sequential bounds can be instantiated for a fixed sample size, yielding tight confidence intervals for this more familiar setting. These will be much tighter than the Hoeffding confidence interval \eqref{eq:hoeffding}, which is itself one such fixed-sample-size instantiation~\citep[Figures 4 and 6]{howard_exponential_2018}.

Let us briefly review some terminology.
For succinctness, we use the notation $X_1^t := (X_1, \dots, X_t)$. Define the sigma-field $\Fcal_t := \sigma(X_1^t)$ generated by $X_1^t$ with $\Fcal_0$ being the trivial sigma-field. The \textit{canonical filtration} $\Fcal := (\Fcal_t)_{t=0}^\infty$ refers to the increasing sequence of sigma-fields $\Fcal_0 \subset \Fcal_1 \subset \Fcal_2 \subset \cdots$. A stochastic process $(M_t)_{t=0}^\infty$ is called a \emph{test supermartingale} for $P$ if $(M_t)_{t=0}^\infty$ is a nonnegative process adapted to $\Fcal$, $M_0 = 1$, and 
\begin{equation}
\small
\EE_P(M_t \mid \Fcal_{t-1}) \leq M_{t-1} \text{ for each $t \geq 1$. }
\end{equation}
$(M_t)_{t=0}^\infty$ is called a \emph{test martingale} for $P$ if the above ``$\leq$'' is replaced with ``$=$''. We sometimes shorten $(M_t)_{t=0}^\infty$ to just $(M_t)$ for brevity. 
If the above property holds simultaneously for all $P \in \Pcal$, we call $(M_t)$ a test (super)martingale for $\Pcal$. 
We say that a sequence $(\lambda_t)_{t=1}^\infty$ is \textit{predictable} if $\lambda_t$ is $\Fcal_{t-1}$-measurable for each $t \geq 1$, meaning $\lambda_t$ can only depend on $X_1^{t-1}$. (In)equalities are interpreted in an almost sure sense.

\subsection{Confidence sequences and the method(s) of mixtures} 
Even though the concentration inequalities thus far have been described in a setting where the sample size $n$ is fixed in advance, all of our ideas stem from a sequential approach towards uncertainty quantification. The goal there is not to produce one confidence set $C_n$, but to produce an infinite sequence $(C_t)_{t=1}^\infty$ such that
\begin{equation}
\small
\label{eqn:CS}
\sup_{P \in \Pcal^\mu} P(\exists t \geq 1: \mu \notin C_t) \leq \alpha.
\end{equation}
Such a $(C_t)_{t=1}^\infty$ is called a \textit{confidence sequence} (CS), and preferably  $\lim_{t\to\infty} C_t = \{\mu\}$. It is known~\citep[Lemma 3]{howard_uniform_2019} that \eqref{eqn:CS} is equivalent to requiring that $\sup_{P \in \Pcal^\mu} P(\mu \notin C_\tau) \leq \alpha$ for arbitrary stopping times $\tau$ with respect to $\Fcal$.

As detailed in the next subsection, one general way to construct a CS is to invert a family of sequential tests based on applying Ville's maximal inequality~\citep{ville_etude_1939} to a test (super)martingale. 
In fact, \citet{ramdas2020admissible} proved that this is (in some formal sense) a universal method to construct CSs, meaning that any other approach can in principle be recovered or dominated by the aforementioned one. 

Designing test supermartingales is nontrivial, and the task of making it have ``power one'' against composite alternatives is often accomplished via the \emph{method of mixtures}. This can arguably be traced back (in a nonstochastic context) to  Ville's 1939 thesis and (in a stochastic context) to \citet{wald_sequential_1945}.  Robbins and collabarators~\citep{robbins_iterated_1968,robbins_statistical_1970,darling_confidence_1967} applied the method to derive CSs, and these ideas have been extended to a variety of nonparametric settings by~\citet{howard_exponential_2018,howard_uniform_2019}. The latter paper describes several variants: conjugate mixtures, discrete mixtures, stitching and inverted stitching. 

These works form our vantage point for the rest of the paper, but we extend them in several ways. First, we describe a ``predictable plug-in''  technique that is implicit in the work of Ville. It can be viewed as a nonparametric extension of a passing remark in the parametric setting in the textbook by Wald~[\citeyear[Eq.10:10]{wald_sequential_1945}] and later explored in the parametric case by~\cite{robbins_expected_1974}.

Like Ville's work in the binary setting, the predictable plug-in method connects the game-theoretic approach and the aforementioned mixture methods --- succinctly, the plugged-in value determines the bet, where each bet is implicitly targeting a different alternative (much like the components of a mixture). Following this translation, prior work on using the method mixtures for confidence sequences can be viewed as using the same betting strategy (mixture distribution) for every value of $m$. We find that there is significant statistical benefit to betting differently for each $m$ (but tied together in a specific way, not in an ad hoc manner). One must typically specify the mixture distribution in advance of observing data, but betting can be viewed as building up a data-dependent mixture distribution on the fly (this led us to previously name our approach as the ``predictable mixture'' method). These sequential perspectives are powerful, even if  only interested in fixed-sample CIs.

\subsection{Nonparametric confidence sequences via sequential testing}

As seen above, it is straightforward to derive a confidence interval for $\mu$ by resorting to a nonparametric concentration inequality like Hoeffding's.
In contrast, it is also well known that CIs are inversions of families of hypothesis tests (as we will see below), so one could presumably derive CIs by first specifying tests.
However, the literature on nonparametric concentration inequalities, such as Hoeffding's, has not commonly utilized a hypothesis testing perspective to derive concentration bounds; for example the excellent book on concentration by~\citet*{boucheron_concentration_2013} has no examples of such an approach. This is presumably because the underlying nonparametric, composite hypothesis tests may be quite challenging themselves, and one may not have nonasymptotically valid solutions or closed-form analytic expressions for these tests. This is in contrast to simple parametric nulls, where it is often easy to calculate a $p$-value based on likelihood ratios. In abandoning parametrics, and thus abandoning likelihood ratios, it may be unclear how to define a powerful test or calculate a nonasymptotically valid $p$-value. This is where betting and test (super)martingales come to the rescue.
\citet[Proposition 4]{ramdas2020admissible} prove that not only do likelihood ratios form test martingales, but every (nonparametric, composite) test martingale is also a (nonparametric, composite) likelihood ratio.

\begin{theorem}[4-step procedure for supermartingale confidence sets]
\label{theorem:4step}
On observing $(X_t)_{t=1}^\infty \sim P$ from $P \in \Pcal^\mu$ for some unknown $\mu \in [0, 1]$, do
\begin{tcolorbox}[%
    enhanced, 
    breakable,
    frame hidden,
    overlay broken = {
        \draw[line width=0.2mm, gray, rounded corners]
        (frame.north west) rectangle (frame.south east);},
    ]{}
    \label{recipe}
    \begin{enumerate}[(a)]
        \item Consider the composite null hypothesis $H_0^m: P \in \Pcal^m$ for each $m \in [0,1]$.
        \item For each index $m \in [0, 1]$, construct a nonnegative process $M_t^m \equiv M^m(X_1, \dots, X_t)$ such that the process $(M_t^\mu)_{t=0}^\infty$ indexed by $\mu$ has the following property: for each $P \in \Pcal^\mu$, $(M_t^\mu)_{t=0}^\infty$ is upper-bounded by a test (super)martingale for $P$, possibly a different one for each $P$.
        \item For each $m \in [0, 1]$ consider the sequential test $(\phi_t^m)_{t=1}^\infty$ defined by 
        \[ \phi_t^m := \1(M_t^m \geq 1/\alpha), \]
        where $\phi_t^m = 1$ represents a rejection of $H_0^m$ after $t$ observations.
        \item Define $C_t$ as the set of $m \in [0, 1]$ for which $\phi_t^m$ fails to reject $H_0^m$:
        \[ C_t := \left \{ m \in [0, 1] : \phi_t^m = 0 \right \}. \]

    \end{enumerate}
\end{tcolorbox}
\noindent Then $(C_t)_{t=1}^\infty$ is a $(1-\alpha)$-confidence sequence for $\mu$: 
\( \sup_{P \in \Pcal^\mu} P(\exists t \geq 1 : \mu \notin C_t) \leq \alpha. \)
\end{theorem}

The above result relies centrally on Ville's inequality~\citep{ville_etude_1939}, which states that if $(L_t) \equiv (L_t)_{t=1}^\infty$ is (upper bounded by) a test martingale for $P$, then we have $P(\exists t \geq 1: L_t \geq 1/\alpha) \leq \alpha$. See~\cite[Section 6]{howard_exponential_2018} for a short proof.

\begin{proof}[\myproofname{Theorem~\ref{theorem:4step}}]
By Ville's inequality, $\phi_t^m$ is a level-$\alpha$ sequential hypothesis test, in the sense that for any $P \in \Pcal^\mu$, we have
\( \small P(\exists t \geq 1 : \phi_t^\mu = 1) \leq \alpha. \)
Now, by definition of the sets $(C_t)_{t=1}^\infty$, we have that $\mu \notin C_t$ at some time $t \geq 1$ if and only if there exists a time $t \geq 1$ such that $\phi_t^\mu = 1$, and hence
\begin{equation}
  \small \sup_{P \in \Pcal^\mu} P(\exists t \geq 1 : \mu \notin C_t) = \sup_{P\in \Pcal^\mu} P(\exists t \geq 1 : \phi_t^\mu = 1) \leq \alpha,
\end{equation}
which completes the proof. 
\myqed
\end{proof}
At a high level, this approach is not new. Composite test supermartingales for $\Pcal$ have been used in past works on concentration inequalities and/or confidence sequences (which are related but different), from the initial series of works by Robbins and collaborators in the 1960s and 1970s, to \cite{de_la_pena_pseudo-maximization_2007}, to recent work by \citet[Section 7.2]{jun2019parameter} and \citet{howard_exponential_2018,howard_uniform_2019}. Test martingales have also been explicitly considered in some hypothesis testing problems~\citep{vovk2005algorithmic,shafer_test_2011}; the latter paper popularized the term ``test martingale'' that we borrow, but unlike us, used it primarily for singleton $\Pcal=\{P\}$. 
We highlight an (independently developed) unpublished preprint by~\citet{hendriks2018test} that has overlaps with the current paper in the with-replacement setting, and some complementary results. 
For singleton (parametric) classes $\mathcal P$, Wald's sequential likelihood ratio statistic is a test martingale, so all of the above methods can be viewed as inverting nonparametric or composite generalizations of  Wald's tests. 

Nevertheless, we make two additional comments. First, the requirement in step (b) of the algorithm that the process $(M^m_t)$ be \emph{upper-bounded} by a test (super)martingale for each $P \in \Pcal$ was posited by~\cite{howard_exponential_2018}, and has recently been christened a e-process for $\Pcal$ \citep{ramdas2021exch} (see also~\cite{grunwald_safe_2019}). E-processes are strictly more general than test (super)martingales for $\Pcal$ in the sense that there exist many interesting classes $\Pcal$ for which nontrivial test (super)martingales do not exist, but one can design powerful e-processes for $\Pcal$. Second, one must take care to design test (super)martingales for each $m$  that are tied together across $m$ in a nontrivial manner that improves statistical power while maintaining computational tractability. All the confidence sets in this paper (both in the sequential and batch settings) will be based on this 4-step procedure, but with different carefully chosen processes $(M_t^m)$. In the language of betting, we will come up with new, powerful ways to bet for each $m$, and also tie together the betting strategies for different $m$.


\subsection{Connections to the Chernoff method}
By virtue of $(C_t)_{t=1}^\infty$ being a time-uniform confidence sequence, we also have that $C_n$ is a $(1-\alpha)$-confidence interval for $\mu$ for any fixed sample size $n$. In fact, the celebrated Chernoff method results in such a confidence interval. So, how exactly are the two approaches related? The answer is simple: Theorem~\ref{theorem:4step} generalizes and improves on the Chernoff method. To elaborate, recall that Hoeffding proved that 
\begin{equation}
\small
\sup_{P \in \Pcal^\mu} \EE_P [\exp(\lambda (X-\mu) - \lambda^2/8)] \leq 1, \text{ for any $\lambda \in \RR$, }
\end{equation}
and so if $X_1^n$ are independent (say), the following process can be used in Step (b): 
\begin{equation}
\small
M^m_t := \prod_{i=1}^t \exp\left(\lambda(X_i-m) - \lambda^2/8\right).
\end{equation}
Usually, the only fact that matters for the Chernoff method is that $\EE_P[M^m_t] \leq 1$, and Markov's inequality is applied (instead of Ville's) in Step (c).
To complete the story, the Chernoff method then involves a smart choice for $\lambda$. Setting $\lambda := \sqrt{8\log (1/\alpha) / n}$ recovers the familiar Hoeffding inequality for the batch sample-size setting. Taking a union bound over $X_1^n$ and $-X_1^n$ yields the Hoeffding confidence interval \eqref{eq:hoeffding} exactly. Using our 4-step approach, the resulting confidence sequence is a time-uniform generalization of Hoeffding's inequality, recovering the latter precisely including constants at time $n$; see~\cite{howard_exponential_2018} for this and other generalizations.


In recent parlance, a statistic like $M^m_t$, which has at most unit expectation under the null, has been called a betting score~\citep{shafer2019language} or an $e$-value~\citep{vovk_testing_randomness_2019} and their relationship to sequential testing \citep{grunwald_safe_2019} and estimation \citep{ramdas2020admissible} as an alternative to $p$-values has been recently examined. In parametric settings with singleton nulls and alternative hypotheses, the likelihood ratio is an $e$-value. For composite null testing, the split likelihood ratio statistic~\citep{wasserman2020universal} (and its variants) are e-values. However, our  setup is more complex: $\Pcal^m$ is highly composite, there is no common dominating measure to define likelihood ratios, but Hoeffding's result yields an $e$-value. (In fact, it yields test supermartingale and hence an e-process, which is an e-value even at stopping times.) 

In summary, the Chernoff method is simply one powerful, but as it turns out, rather limited way to construct an $e$-value. 
This paper provides better constructions of $M^m_t$, whose expectation is exactly equal to one, thus removing one source of looseness in the Hoeffding-type approach above, as well as better ways to pick the tuning parameter $\lambda$, which will correspond to our bet.


\section{Warmup: exponential supermartingales and predictable plug-ins}
\label{section:warmup}

A central technique for constructing confidence sequences (CSs) is Robbins' \textit{method of mixtures}~\citep{robbins_statistical_1970}, see also~\cite{darling_confidence_1967,robbins_iterated_1968,robbins_boundary_1970,robbins_class_1972,robbins_expected_1974}. Related ideas of ``pseudo-maximization'' or Laplace's method were further popularized and extended by \citet{de_la_pena_self-normalized_2004,de_la_pena_pseudo-maximization_2007,de_la_pena_self-normalized_2009}, and has led to several other followup works~\citep{abbasi2011improved,balsubramani_sharp_2014,howard_exponential_2018,kaufmann2018mixture}. 

However, beyond the case when the data are (sub)-Gaussian, the method of mixtures rarely leads to a closed-form CS; it yields an \emph{implicit} construction for $C_t$ which can sometimes be computed efficiently (e.g. using conjugate mixtures~\citep{howard_uniform_2019}), but is otherwise analytically opaque and computationally tedious. Below, we provide an alternative construction --- called the ``predictable plug-in'' --- that is exact, explicit and efficient (computationally and statistically). 

In the next section, our CSs avoid exponential supermartingales, and are much tighter than the recent state-of-the-art in~\cite{howard_uniform_2019}. The ones in this section match the latter but are simpler to compute, so we present them first.

\subsection{Predictable plug-in Cramer-Chernoff supermartingales}
Suppose $(X_t)_{t=1}^\infty \sim P$ for some $P \in \Pcal^\mu$ where $\Pcal^\mu$ is the set of all distributions on $\prod_{i=1}^\infty [0, 1]$ so that $\EE_P(X_t \mid \Fcal_{t-1}) = \mu$ for each $t$. The Hoeffding process $(M_t^H(m))_{t=0}^\infty$ for a given candidate mean $m \in [0, 1]$ is given by
\begin{equation} \small M_t^H(m) := \prod_{i=1}^t \exp \left ( \lambda(X_i - m) - \psi_H(\lambda) \right ) \end{equation}
with $M_0^H(m) \equiv 1$ by convention. Here $\psi_H(\lambda) := \lambda^2/8$ is an upper bound on the cumulant generating function (CGF) for $[0, 1]$-valued random variables with $\lambda \in \mathbb R$ chosen in some strategic way. For example, to maximize $M_n^H(m)$ at a fixed sample size $n$, one would set $\lambda := \sqrt{8 \log (1/\alpha) / n}$ as in the classical fixed-time Hoeffding inequality \citep{hoeffding_probability_1963}.

Following \citet{howard_uniform_2019}, we have that $(M_t^H(\mu))_{t=0}^\infty$ is a nonnegative supermartingale with respect to the canonical filtration. Therefore, by Ville's maximal inequality for nonnegative supermartingales~\citep{ville_etude_1939,howard_exponential_2018},
\begin{equation}
\small
\label{eq:HoeffdingVille}
P\left ( \exists t \geq 1 : M_t^H(\mu) \geq 1/\alpha \right ) \leq \alpha.
\end{equation}
Robbins' method of mixtures proceeds by noting that $\int_{\lambda \in \RR} M_t^H(m) dF(\lambda)$ 
is also a supermartingale for any ``mixing'' probability distribution $F(\lambda)$ on $\RR$ and thus 
\begin{equation} \small P\left ( \exists t \geq 1 : \int_{\lambda \in \RR} M_t^H(\mu) dF(\lambda) \geq 1/\alpha \right ) \leq \alpha. \end{equation}
In this particular case, if $F(\lambda)$ is taken to be the Gaussian distribution, then the above integral can be computed in closed-form \citep{howard_exponential_2018}. For other distributions or altogether different supermartingales (i.e. other than Hoeffding), the integral may be computationally tedious or intractable.

To combat this, instead of fixing $\lambda \in \RR$ or integrating over it, consider constructing a sequence $\lambda_1, \lambda_2, \dots$ which is predictable, and thus $\lambda_t$ can depend on $X_1^{t-1}$. Then, 
\begin{equation} 
\label{eq:hoeffdingSuperMG}
\small 
M_t^{\PMH}(m) := \prod_{i=1}^t \exp(\lambda_i(X_i - m) - \psi_H(\lambda_i)) \end{equation}
is also a test supermartingale for $\Pcal^m$ (and hence Ville's inequality applies). We call such a sequence $(\lambda_t)_{t=1}^\infty$ a \textit{predictable plug-in}. While not always explicitly referred to by this exact name, predictable plug-ins have appeared in works on parametric sequential analysis by \citet[Eq. (10:10)]{wald_sequential_1947}, \citet[Eq. (4)]{robbins_expected_1974}, \citet{dawid1984present}, and \citet{lorden_nonanticipating_2005} as well as in the information theory literature \citep{rissanen1984universal}. 
As we will see, these techniques also prove useful in nonparametric testing and estimation problems both in sequential and batch settings.

Using $M_t^\PMH(m)$ as the process in Step (b) of Theorem~\ref{theorem:4step} results in a lower CS for $\mu$, while constructing an analogous supermartingale using $(-X_t)_{t=1}^\infty$ yields an upper CS. Combining these by taking a union bound results in the predictable plug-in Hoeffding CS which we introduce now. 
\begin{proposition}[Predictable plug-in Hoeffding CS {\textbf{[PrPl-H]}}]
\label{proposition:predmixHoeffding}
Suppose that $(X_t)_{t=1}^\infty \sim P$ for some $P \in \Pcal^\mu$. For any chosen real-valued predictable  $(\lambda_t)_{t=1}^\infty$,
\[ C_t^\PMH := \left ( \frac{\sum_{i=1}^t \lambda_i X_i }{\sum_{i=1}^t \lambda_i} \pm \frac{\log (2/\alpha) + \sum_{i=1}^t \psi_H (\lambda_i )}{\sum_{i=1}^t \lambda_i} \right ) ~~~\text{forms a $(1-\alpha)$-CS for $\mu$,} \]
as does its running intersection, $\bigcap_{i \leq t} C_i^H$. 
\end{proposition}

\noindent A sensible choice of predictable plug-in is given by
\begin{equation} \small \lambda_t^\PMH := \sqrt{\frac{8 \log ( 2/\alpha )}{t \log ( t+ 1)}} \land 1, \end{equation}
for reasons which will be discussed in Section~\ref{section:howToMix}.
The proof of Proposition~\ref{proposition:predmixHoeffding} is provided in Section~\ref{proof:hoeffding}. As alluded to earlier, predictable plug-ins are actually the \emph{least} interesting when using Hoeffding's sub-Gaussian bound because of the available closed form Gaussian-mixture boundary. However, the story becomes more interesting when either (a) the method of mixtures is computationally opaque or complex, or (b) the optimal choice of $\lambda$ is based on unknown but estimable quantities. Both (a) and (b) are issues that arise when computing empirical Bernstein-type CSs and CIs. In the following section, we present predictable plug-in empirical Bernstein-type CSs and CIs which are both computationally and statistically efficient.

\subsection{Application: closed-form empirical Bernstein confidence sets}

To prepare for the results that follow, consider the empirical Bernstein-type process,
\begin{equation} 
\label{eq:EBSuperMG}
\small 
M_t^{\PMEB}(m) := \prod_{i=1}^t \exp\left \{ \lambda_i (X_i-m) - v_i \psi_E(\lambda_i) \right \} 
\end{equation}
where, following \citet{howard_exponential_2018,howard_uniform_2019}, we have defined $v_i := 4(X_i - \widehat \mu_{i-1})^2$ and
\begin{equation}\label{eq:psi-E}
\small
\psi_E(\lambda) := (- \log (1-\lambda) - \lambda)/4 \quad \text { for } \lambda \in [0, 1).
\end{equation}
As we revisit later, the appearance of the constant 4 is to facilitate easy comparison to $\psi_H$, since $\lim_{\lambda \to 0^+} \psi_E(\lambda)/\psi_H(\lambda) = 1$. In short, $\psi_E$ is nonnegative, increasing on $[0,1)$, and grows quadratically near 0.

Using $M_t^\PMEB (m)$ in Step (b) in Theorem~\ref{theorem:4step} --- 
and applying the same procedure but with $(X_t)_{t=1}^\infty$ and $m$ replaced by $(-X_t)_{t=1}^\infty$ and $-m$ combined with a union bound over the resulting CSs --- we get the following CS.

\begin{theorem}[Predictable plug-in empirical Bernstein CS {\textbf{[PrPl-EB]}}]
\label{theorem:EBCS}
\noindent Suppose $(X_t)_{t=1}^\infty \sim P$ for some $P \in \Pcal^\mu$. For any  $(0, 1)$-valued predictable $(\lambda_t)_{t=1}^\infty$,
\[ C_t^\PMEB := \left ( \frac{\sum_{i=1}^t \lambda_i X_i}{\sum_{i=1}^t \lambda_i} \pm \frac{\log(2/\alpha) +  \sum_{i=1}^tv_i \psi_E(\lambda_i) }{\sum_{i=1}^t \lambda_i} \right ) ~~~\text{forms a $(1-\alpha)$-CS for $\mu$, }\]
as does its running intersection, $\bigcap_{i \leq t} C_i^\PMEB$.
\end{theorem}
In particular, we recommend the predictable plug-in $(\lambda^\PMEB_t)_{t=1}^\infty$ given by
\begin{equation} 
\label{eq:lambda_pm-eb}
\small
\lambda^\PMEB_t := \sqrt{\frac{2 \log (2/\alpha)}{\widehat \sigma_{t-1}^2 t \log (1+t) }} \land c, ~~~ \widehat \sigma_{t}^2 := \frac{\tfrac1{4} + \sum_{i=1}^t (X_i - \widehat \mu_i)^2}{t + 1}, ~~~ \widehat \mu_t := \frac{\tfrac1{2} + \sum_{i=1}^t X_i}{t + 1} 
\end{equation}
for some $c \in (0, 1)$ (a reasonable default being 1/2 or 3/4). This choice was inspired by the fixed-time empirical Bernstein as well as the widths of time-uniform CSs (more details are provided in Section~\ref{section:howToMix}). The sequences of estimators $(\widehat \mu_t)_{t=1}^\infty$ and $(\widehat \sigma_t^2)_{t=1}^\infty$ can be interpreted as predictable, regularized sample means and variances. This technique was employed by \citet{kotlowski2010following} for misspecified exponential families in the so-called \textit{maximum likelihood plug-in strategy}. 

The proof of Theorem~\ref{theorem:EBCS} relies on establishing that $M_t^\PMEB (m)$ is a test supermartingale for $\Pcal^m$. This latter fact is related to, but cannot be derived directly from, a powerful deterministic inequality for bounded numbers due to \citet{fan_exponential_2015}. One needs an additional trick from \citet[Section A.8]{howard_uniform_2019} which swaps $(X_i-m)^2$ with $(X_i-\widehat \mu_{i-1})^2$, for any predictable $\widehat \mu_{i-1}$, within the variance term $v_i$. It is this additional piece which yields both tighter and \emph{closed-form} CSs; details are in Section~\ref{proof:EBCS}. 
We remark that before taking the running intersection, the above intervals are symmetric around the weighted sample mean, but this symmetry will not carry forward to other CSs in the paper.

\begin{figure}[!htbp]
    \centering \textbf{Time-uniform empirical Bernstein confidence sequences}
    \includegraphics[width=0.9\textwidth]{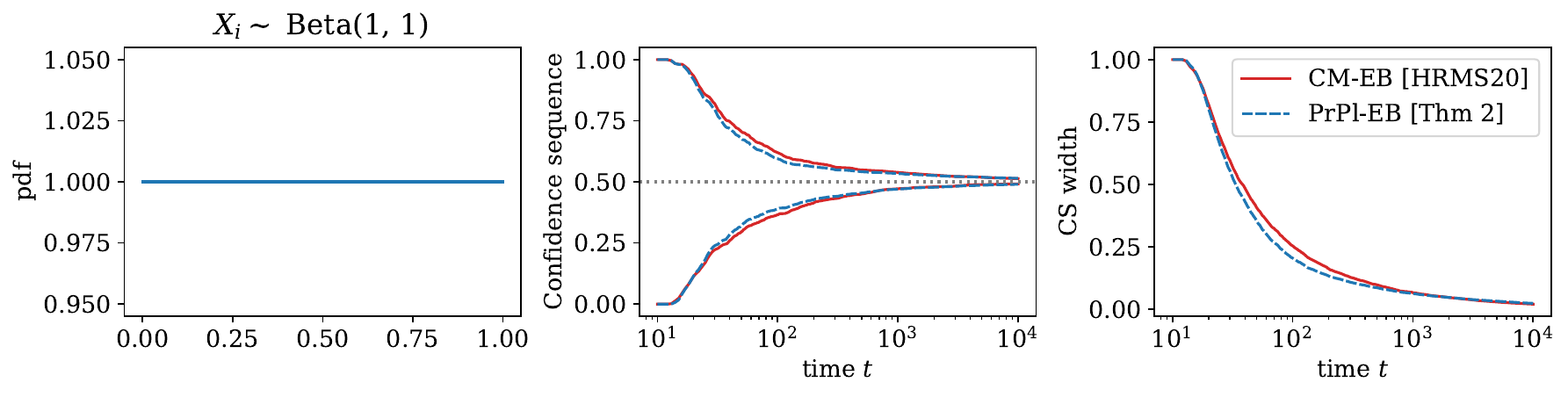}
    \caption{Empirical Bernstein CSs produced via a predictable plug-in (PrPl) with $(\lambda_t)_{t=1}^\infty$ from \eqref{eq:lambda_pm-eb}  match (or slightly improve) those obtained via conjugate mixtures (CM) by \cite{howard_uniform_2019}; the former is closed-form, but the latter is not and requires numerical methods.}
    \label{fig:CMvsPM}
\end{figure}

Figure~\ref{fig:CMvsPM} compares the conjugate mixture empirical-Bernstein CS (CM-EB) due to \citet{howard_uniform_2019} with our predictable plug-in empirical-Bernstein CS ($\PMEB$). The two CSs perform similarly,
    but our closed-form $\PMEB$ is over 500 times faster to compute than CM-EB (in our experience) which requires root finding at each step. However, our later bounds will be tighter than both of these. 

\begin{remark}
\label{remark:EBCI}
Theorem~\ref{theorem:EBCS} yields computationally and statistically efficient empirical Bernstein-type CIs for a fixed sample size $n$. 
Recalling~\eqref{eq:lambda_pm-eb}, we recommend using $\bigcap_{i \leq n}C_i^\PMEB$ along with the predictable sequence
\begin{equation} \small \lambda_{t}^{\VAEB} := \sqrt{\frac{2 \log (2/\alpha)}{n \widehat \sigma_{t-1}^2}} \land c.
\end{equation}
We call the resulting confidence interval the ``predictable plug-in empirical Bernstein confidence interval'' or \textbf{[PrPl-EB-CI]} for short; see Figure~\ref{fig:PMEBvsMP}.
\end{remark}

If $X_1, \dots, X_n$ are independent, then at the expense of computation, the above CI can be effectively derandomized to remove the effect of the ordering of variables. One can randomly permute the data $B$ times to obtain $(\widetilde X_{1, b}, \dots, \widetilde X_{n,b})$ and correspondingly compute $\widetilde M_{n, b}^\PMEB(m)$, one for each permutation ${b \in \{1,\dots, B\}}$. Averaging over these permutations, define
\( \widetilde M_n^\PMEB(m) := \frac{1}{B}\sum_{b=1}^B \widetilde M_{n,b}^\PMEB(m). \)
For each $b$, $M_{n,b}^\PMEB(\mu)$ has expectation at most one (by linearity of expectation). Thus, $\widetilde M_n^\PMEB(\mu)$ is a $e$-value (i.e. it has expectation at most 1). By Markov's inequality, 
$\widetilde C_n^\PMEB := \{ m \in [0, 1] : \widetilde M_n^\PMEB(m) < 1/\alpha \}$
is a $(1-\alpha)$-CI for $\mu$. This set is not available in closed-form and the intersection $\bigcap_{i\leq n} \widetilde C_i^\PMEB$ no longer yield a valid CI\@. In our experience, this derandomization procedure neither helps nor hurts. In any case, both $\bigcap_{i\leq n} C_i$ and $\widetilde C_n$ will be significantly improved in Section~\ref{section:hedgedCapitalProcess}.

In Section~\ref{section:VAEB_scale}, we show that in iid settings the width of [PrPl-EB-CI] scales with the true (unknown) standard deviation:
\begin{equation}\label{eq:asymptotic-eb-width}
    \small
    \sqrt{n} \left (\frac{\log(2/\alpha) + \sum_{i=1}^n v_i \psi_E(\lambda_i)}{\sum_{i=1}^n \lambda_i} \right ) \xrightarrow{a.s.} \sigma \sqrt{2 \log (2/\alpha)}.
\end{equation}
Notice that \eqref{eq:asymptotic-eb-width} is the same asymptotic behavior that one would observe for CIs based on Bernstein's or Bennett's inequalities, both of which require knowledge of the true variance $\sigma^2$, while [PrPl-EB-CI] does not. This is in contrast to the empirical Bernstein CIs of \citet{maurer_empirical_2009} whose limit would be $\sigma \sqrt{2 \log(4/\alpha)}$. In the maximum variance case where $\sigma = 1/2$, \eqref{eq:asymptotic-eb-width} yields the same asymptotic behavior as Hoeffding's CI \eqref{eq:hoeffding}. 
\begin{figure}[!htbp]
    \centering \textbf{Fixed-time empirical Bernstein confidence intervals}
    \includegraphics[width=0.9\textwidth]{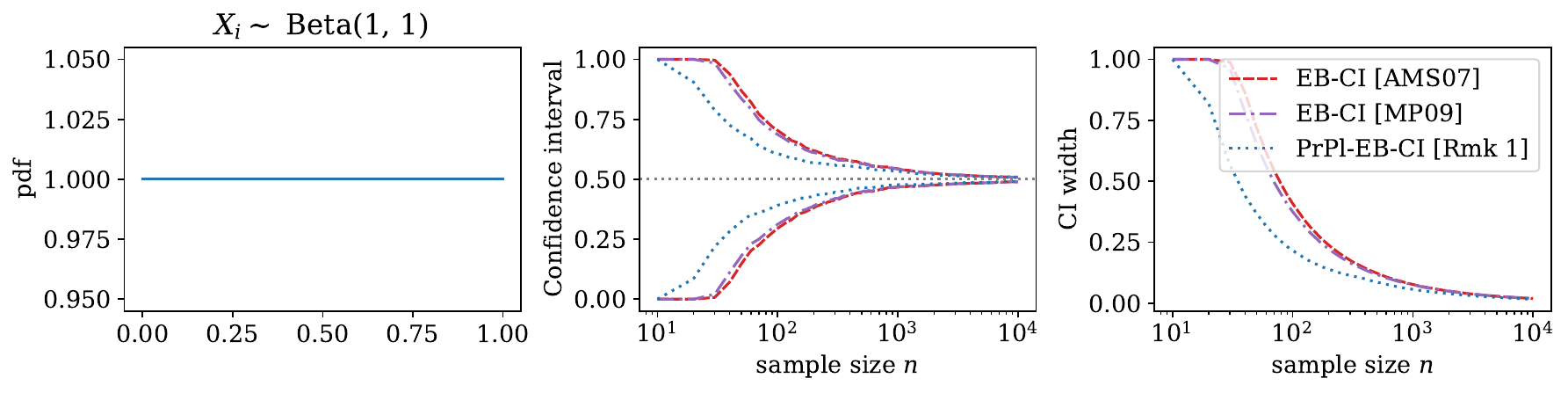}
    \caption{Our predictable plug-in (PrPl) empirical Bernstein (EB) CI is significantly tighter than those of \cite{maurer_empirical_2009}  and \cite{audibert2007tuning}.}
    \label{fig:PMEBvsMP}
\end{figure}

Until now, we presented various predictable plug-ins --- $(\lambda_t^\PMH)_{t=1}^\infty$, $(\lambda_t^\PMEB)_{t=1}^\infty$, and $(\lambda_t^{\PMEB(n)})_{t=1}^n$ --- but have not provided intuition for why these are sensible choices. Next, we discuss guiding principles for deriving predictable plug-ins.

\subsection{Guiding principles for deriving predictable plug-ins}
\label{section:howToMix}
Let us begin our discussion with the predictable plug-in Hoeffding process~\eqref{eq:hoeffdingSuperMG}
and the resulting CS in~Proposition~\ref{proposition:predmixHoeffding}, which has a half-width
\[
W_t = \frac{\log(2/\alpha) +  \sum_{i=1}^t\lambda_i^2/8 }{\sum_{i=1}^t \lambda_i}
\]
To ensure that $W_t \rightarrow 0$ as $t \rightarrow \infty$, it is clear that we want $\lambda_t \xrightarrow{\mathrm{a.s.}} 0$, but at what rate? As a sensible default, we recommend setting $\lambda_t \asymp 1/\sqrt{t \log t}$ so that $W_t = \widetilde O( \sqrt{\log t / t})$ which matches the width of the conjugate mixture Hoeffding CS \citep[Proposition 2]{howard_exponential_2018} (here $\widetilde O$ treats $O(\log \log t)$ factors as constants). See Table~\ref{tab:schedules} for a comparison between rates for $\lambda_t$ and their resulting CS widths.

\begin{table}
\caption{Below, we think of $\log x$ as $\log(x+1)$ to avoid trivialities. The claimed rates are easily checked by approximating the sums as integrals, and taking derivatives. For example, $\tfrac{d}{dx}\log \log x = 1/x \log x$, so the sum of $\sum_{i \leq t} 1/i \log i \asymp \log\log t$. It is worth remarking that for $t=10^{80}$, the number of atoms in the universe, $\log \log t \approx 5.2$, which is why we treat $\log \log t$ as a constant when expressing the rate for $W_t$. The iterated logarithm pattern in the the last two lines can be continued indefinitely.}
\label{tab:schedules}
\centering
 \begin{tabular}{|c c c c|} 
 \hline
 Strategy $(\lambda_i)_{i=1}^\infty$ & $\sum_{i=1}^t \lambda_i$ & $\sum_{i=1}^t \lambda_i^2$ & Width $W_t$ \\ [0.5ex] 
 \hline
 \hline
 $\asymp 1/i $ & $\asymp \log t$ & $\asymp 1$ & $1/\log t$ \\
 \hline
 $\asymp \sqrt{\log i/i}$ & $\asymp \sqrt{t \log t}$ & $\asymp \log^2 t$ & $\asymp \log^{3/2} t / \sqrt{t}$ \\
 \hline
 $\asymp 1/{\sqrt i}$ & $\asymp \sqrt{t}$ & $\asymp\log t$ & $\asymp \log t / \sqrt{t}$ \\ 
 \hline
 $\asymp 1/\sqrt{i \log i} $ & $\asymp \sqrt{t/\log t}$  & $\asymp \log\log t$ & $\asymp \sqrt{\log t/t}$ \\
 \hline
 $\asymp 1/\sqrt{i \log i \log\log i} $ & $\asymp \sqrt{t/\log t}$  & $\asymp \log\log\log t$ & $\asymp \sqrt{\log t/t}$ \\
 \hline
\end{tabular}
\end{table}


Now consider the predictable plug-in empirical Bernstein process~\eqref{eq:EBSuperMG}
and the resulting CS of Theorem~\ref{theorem:EBCS}, which has a half-width
\[
W_t = \frac{\log(2/\alpha) + \sum_{i=1}^t 4(X_i - \widehat \mu_{i-1})^2 \psi_E(\lambda_i) }{\sum_{i=1}^t \lambda_i}
\]
By two applications of L'H\^opital's rule, we have that 
\begin{equation} \small \frac{\psi_E(\lambda)}{\psi_H(\lambda)} \xrightarrow{\lambda \rightarrow 0^+} 1. \end{equation}
Performing some approximations for small $\lambda_i$ to help guide our choice of $(\lambda_t)_{t=1}^\infty$ (without compromising validity of resulting confidence sets) we have that
\begin{align}
    W_t 
    &\approx \frac{\log (2/\alpha) + \sum_{i=1}^t 4(X_i - \mu)^2 \lambda^2_i / 8}{\sum_{i=1}^t \lambda_i}.
\end{align}
Thus, in the special case of i.i.d. $X_i$ with variance $\sigma^2$, for large enough $t$,
\begin{align}
    \EE_P(W_t \mid \Fcal_{t-1}) &\lesssim \frac{\log(2/\alpha) + \sigma^2 \sum_{i=1}^t \lambda_i^2/2 }{\sum_{i=1}^t \lambda_i}.
\end{align}
If we were to set $\lambda_1 = \lambda_2 = \dots = \lambda^\star \in \mathbb R$ and minimize the above expression for a specific time $t^\star$, this amounts to minimizing 
\begin{equation} \small \frac{\log (2/\alpha) + \sigma^2 t^\star {\lambda^\star}^2/2 }{t^\star \lambda^\star}, \end{equation}
which is achieved by setting 
\begin{equation} \small \lambda^\star := \sqrt{\frac{2 \log (2/\alpha)}{\sigma^2 t^\star}}. \end{equation}
This is precisely why we suggested the predictable plug-in $(\lambda^\PM_t)_{t=1}^\infty$ given by~\eqref{eq:lambda_pm-eb}, 
where the additional $\log(t+1)$ is included in an attempt to enforce $W_t = \widetilde O (\sqrt{\log t / t})$.

The above calculations are only used as guiding principles to sharpen the confidence sets, but \emph{all} such schemes retain the validity guarantee. As long as $(\lambda_t)_{t=1}^\infty$ is $[0, 1)$-valued and predictable, we have that $(M_t^E(\mu))_{t=0}^\infty$ is a test supermartingale for $\Pcal^\mu$ which can be used in Theorem~\ref{theorem:4step} to obtain different valid CSs for $\mu$.

Foreshadowing our attempt to generalize this procedure in the next section, notice that the exponential function was used throughout to ensure nonnegativity, but that any other test supermartingale would have sufficed. In fact, if a martingale is used in place of a supermartingale, then Ville's inequality is tighter. 

Next, we present a test \textit{martingale}, removing a source of looseness in the confidence sets derived thus far. We discuss its betting interpretation, provide other guiding principles for setting $\lambda_i$ (equivalently, for betting), which will involve attempting to maximize the expected log-wealth in the betting game.


\section{The capital process, betting, and martingales}
\label{section:capitalProcess}
In Section~\ref{section:warmup}, we generalized the Cramer-Chernoff method to derive predictable plug-in exponential supermartingales and used this result to obtain tight empirical Bernstein CSs and CIs. In this section, we consider an alternative process which can be interpreted as the wealth accumulated from a series of bets in a game. This process is a central object of study in the game-theoretic probability literature where it is referred to as the \textit{capital process} \citep{shafer_probability_2005}. We discuss its connections to the purely statistical goal of constructing CSs and CIs and demonstrate how these sets improve on Cramer-Chernoff approaches, including the empirical Bernstein confidence sets of the previous section.

Consider the same setup as in Section~\ref{section:warmup}: we observe an infinite sequence of conditionally mean-$\mu$ random variables, $(X_t)_{t=1}^\infty \sim P$ from some distribution $P \in \Pcal^\mu$. Define the \textit{capital process} $\Kcal_t(m)$ for any $m \in [0, 1]$,
\begin{equation}
\small
\label{eq:capitalProcess}
\Kcal_t(m) := \prod_{i=1}^t (1 + \lambda_i(m)\cdot(X_i - m)),
\end{equation}
with $\Kcal_0(m) := 1$ and where $(\lambda_t(m))_{t=1}^\infty$ is a $(-1/(1-m), 1/m)$-valued predictable sequence, and thus $\lambda_t(m)$ can depend on $X_1^{t-1}$. Note that for each $t \geq 1$, we have $X_t \in [0, 1]$, $m \in [0, 1]$ and $\lambda_t(m) \in (-1/(1-m), 1/m)$. Here and below, $1/m$ should be interpreted as $\infty$ when $m=0$ and similarly for $1/(1-m)$ and $m=1$, respectively. Importantly, $(1 + \lambda_t(m)\cdot(X_t - m)) \in [0, \infty)$, and thus $\Kcal_t(m) \geq 0$ for all $t \geq 1$. Following similar techniques to the previous section, the reader may easily check that $\Kcal_t(\mu)$ is a test martingale. Moreover, we have the stronger result summarized in the following central proposition.
\begin{proposition}
\label{proposition:betting}
    Suppose a draw from some distribution $P$ yields a sequence $X_1, X_2, \dots$ of $[0, 1]$-valued random variables, and let $\mu \in [0,1]$ be a constant. The following four statements imply each other:
    \begin{enumerate}[(a)]
        \item $\EE_P\left (X_t \mid \Fcal_{t-1} \right) = \mu$ for all $t \in \NN$, where $\Fcal_{t-1} = \sigma(X_1, \dots ,X_{t-1})$.
        \item There exists a constant $\lambda \in \RR \backslash \{0\}$ for which $(\Kcal_t(\mu))_{t=0}^\infty$ is a strictly positive test martingale for $P$.
        \item For every fixed $\lambda \in (-\tfrac1{1-\mu},\tfrac1\mu )$, $(\Kcal_t(\mu))_{t=0}^\infty$ is a test martingale for $P$.
        \item For every $(-\tfrac1{1-\mu},\tfrac1\mu )$-valued predictable sequence $(\lambda_t)_{t=1}^\infty$, $(\Kcal_t(\mu))_{t=0}^\infty$ is a test martingale for $P$.
    \end{enumerate}
    Further, the intervals $(-\tfrac1{1-\mu},\tfrac1\mu)$ mentioned above can be replaced by any subinterval containing at least one nonzero value, like $[-1, 1]$ or $(-0.5, 0.5)$. 
    Finally, \underline{every} test martingale for $\Pcal^\mu$ is of the form $(\Kcal_t(\mu))$ for some predictable sequence $(\lambda_t)$.
\end{proposition}
The proof can be found in Section~\ref{proof:betting}. While the subsequent theorems will primarily make use of $(a) \implies (d)$, the above proposition establishes a core fact: the assumption of the (conditional) means being identically $\mu$ is an \emph{equivalent restatement} of our capital process being a test martingale. Thus, test martingales are not simply ``technical tools'' to deal with means of bounded random variables, they are fundamentally at the very heart of the problem definition itself.

Proposition~\ref{proposition:betting} can be generalized to another remarkable, yet simple, result: for any set of distributions $\Scal$, \emph{every} test martingale for $\Scal$ has the same form.

\begin{proposition}[Universal representation]
For any arbitrary set of (possibly unbounded) distributions $\Scal$, $(M_t)$ is a test martingale for $\Scal$ if and only if $M_t = \prod_{i=1}^t (1 + \lambda_i Z_i)$ for some $Z_i \geq -1$ such that $\EE_S[Z_i | \Fcal_{i-1}] = 0$ for every $S \in \Scal$, and some predictable $\lambda_i$ such that $\lambda_iZ_i\geq -1$. The same claim also holds for test supermartingales for $\Scal$, with the aforementioned ``$=0$'' replaced by ``$\leq 0$''. 
\label{prop:universal}
\end{proposition}
The proof can be found in Section~\ref{proof:universal}.
The above proposition immediately makes this paper's techniques actionable for a wide class of nonparametric testing and estimation problems. We give an example relating to quantiles later. 

\subsection{Connections to betting} 
\label{subsec:connections-to-betting}
It is worth pausing to clarify how the capital process $\Kcal_t(m)$ and Proposition~\ref{proposition:betting} can be viewed in terms of betting. 
We imagine that nature implicitly posits a hypothesis $H_0^m$ --- which we treat as a game providing us a chance to make money if the hypothesis is wrong, by repeatedly betting some of our capital against $H_0^m$. We start the game with a capital of 1 (i.e. $\Kcal_0(m) := 1$), and design a bet of $b_t := s_t |\lambda^m_t|$ at each step, where $s_t \in \{-1, 1\}$. Setting $s_t := 1$ indicates that we believe that $\mu > m$ while $s_t := -1$ indicates the opposite. $|\lambda_t^m|$ indicates the amount of our capital that we are willing to put at stake at time $t$: setting $\lambda_t^m = 0$ results in neither losing nor gaining any capital regardless of the outcome, while setting $\lambda_t^m \in \{ -1/(1-m), 1/m \}$ means that we are willing to risk all of our capital on the next outcome. 

However, if $H^m_0$ is true (i.e. $m = \mu$), then by Proposition~\ref{proposition:betting}, our capital process is a martingale. In betting terms, no matter how clever a betting strategy  $(\lambda_t^m)_{t=1}^\infty$ we devise, we cannot expect to make (or lose) money at each step. If on the other hand, $H^m_0$ is false, then a clever betting strategy will make us a lot of money. In statistical terms, when our capital exceeds $1/\alpha$, we can confidently reject the hypothesis $H^m_0$ since if it were true (and the game were fair) then by Ville's inequality \citep{ville_etude_1939}, the a priori probability of this \textit{ever} occurring is at most $\alpha$. We imagine simultaneously playing this game with $H_0^{m'}$ for each $m' \in [0, 1]$. At any time $t$, the games $m' \in [0, 1]$ for which our capital is small ($<1/\alpha$) form a CS.

Both the Cramer-Chernoff processes of Section~\ref{section:warmup} and $\Kcal_t(m)$ are nonnegative and tend to increase when $\mu > m$. However, only $\Kcal_t(m)$ is a \textit{test martingale} when $m = \mu$; the others are test supermartingales. A test martingale is the wealth accumulated in a ``fair game'' where our capital stays constant in expectation, while a test supermartingale is the wealth accumulated in a game where our capital is expected to decrease (not strictly). Larger values of capital correspond to rejecting $H_0^m$ more readily. Therefore, test supermartingales tend to yield conservative tests compared to their martingale counterparts.

More generally, \emph{every} nonnegative supermartingale can be regarded as the wealth process of a gambler playing a game with odds that are fair or stacked against them. In other words, there is a one-to-one correspondence between wealths of hypothetical gamblers and nonnegative supermartingales. Taking this perspective, every statement involving nonnegative supermartingales (and thus likelihood ratios) are statements about betting, and vice versa. Mixture methods that combine nonnegative supermartingales are simply strategies to hedge across various instruments available to the gambler. Thus, the gambling analogy can be entirely dropped, and our results would find themselves comfortably nestled in the rich literature on martingale methods for concentration inequalities, but we mention the betting analogy for intuition so that the mathematics are animated and easier to absorb.

Ville introduced martingales into modern mathematical probability theory, and centered them around their betting interpretation. Since then, ideas from betting have appeared in various fields, including probability theory, statistical testing and estimation, information theory, and online learning theory. While our paper focuses on the utility of betting in some statistical inference tasks, Section~\ref{section:history} provides a brief overview of the use of betting in other mathematical disciplines.

\subsection{Connections to likelihood ratios}

As alluded to in the previous subsection, useful intuition is provided via the connection to likelihood ratios. $\Kcal_t(m)$ is a ``composite'' test martingale for $\Pcal^m$,  meaning that it is a nonnegative martingale starting at one for every $P \in \Pcal^m$ (recall that $P$ is a distribution over infinite sequences of observations with conditional mean $m$). 

If we were dealing with a single distribution such as $Q^\infty$, meaning a product distribution where every observation is drawn iid from $Q$, then one may pick any alternative $Q'$ that is absolutely continuous with respect to $Q$, to observe that the likelihood ratio $\prod_{i=1}^t Q'(X_i) / Q(X_i)$ is a test martingale for $Q^\infty$. 

However, since $\Pcal^m$ is highly composite and nonparametric and is not even dominated by a single measure (as it contains atomic measures, continuous measures, and all their mixtures), it is unclear how one can even begin to write down a likelihood ratio. Nevertheless, \citet[Proposition 4]{ramdas2020admissible} show that if $(M_t)$ is a composite test martingale for any $\mathcal S$, then for every distribution $Q \in \mathcal S$, $M_t$ equals the likelihood ratio of some $Q'$ against $Q$ (where $Q'$ depends on $Q$). 

Thus, not only is every likelihood ratio a test martingale, but every (composite) test martingale can also be represented as a likelihood ratio.
Hence, in a formal sense, test martingales are nonparametric composite generalizations of likelihood ratios, which are at the very heart of statistical inference. When this observation is combined with Proposition~\ref{proposition:betting}, it should be no surprise any longer that the capital process $\Kcal_t(m)$ (even devoid of any betting interpretation) is fundamental to the problem at hand. In Section~\ref{section:EL} we also observe connections to the empirical likelihood of \cite{owen2001empirical} and  the dual likelihood of \citep{mykland1995dual}. 

\subsection{Adaptive, constrained adversaries}
Despite the analogies to betting, the game described so far appears to be purely stochastic in the sense that nature simply commits to a distribution $P \in \Pcal^\mu$ for some unknown $\mu \in [0,1]$ and presents us observations from $P$. However, Proposition~\ref{proposition:betting} can be extended to a more adversarial setup, but with a constrained adversary.

To elaborate, recall the difference between $\Qcal$ and $\Pcal$ from the start of Section~\ref{sec:setup} and consider a game with three players: an adversary, nature, and the statistician. First, the adversary commits to a $\mu \in [0,1]$. Then, the game proceeds in rounds. At the start of round $t$, the statistician publicly discloses the bets for every $m$, which could depend on $X_1,\dots,X_{t-1}$. The adversary picks a distribution $Q_t \in \Qcal^\mu$, which could depend on $X_1,\dots,X_{t-1}$ and the statistician's disclosed bets, and hands $Q_t$ to nature. Nature simply acts like an arbitrator, first verifying that the adversary chose a $Q_t$ with mean $\mu$, and then draws $X_t \sim Q_t$ and presents $X_t$ to the statistician. 

In this fashion, the adversary does not need to pick $\mu$ and $P \in \Pcal^\mu$ at the start of the interaction, which is the usual stochastic setup, but can instead build the distribution $P$ in a data-dependent fashion over time. In other words, the adversary does not commit to a distribution $P$, but instead to a \emph{rule for building} $P$ from the data. Of course, they do not need to disclose this rule, or even be able express what this rule would do on any other hypothetical outcomes other than the one observed. The results in this paper, which build on the central Proposition~\ref{proposition:betting}, continue to hold in this more general interaction model.

A geometric reason why we can move from the stochastic model first described to the above (constrained) adversarial model, is that the above distribution $P$ lies in the ``fork convex hull'' of $\Pcal^\mu$. Fork-convexity is a sequential analogue of convexity~\citep{ramdas2021exch}. Informally, the fork-convex hull of a set of distributions over sequences is the set of predictable plug-ins of these distributions, and is much larger than their convex hull (mixtures). If a process is a nonnegative martingale under every distribution in a set, then it is also a nonnegative martingale under every distribution in the fork convex hull of that set. No results about fork convexity are used anywhere in this paper, and we only mention it for the mathematically curious.

\subsection{The hedged capital process}
\label{section:hedgedCapitalProcess}
We now return to the purely statistical problem of using the capital process $\Kcal_t(m)$ to construct time-uniform CSs and fixed-time CIs. We might be tempted to use $\Kcal_t(\mu)$ as the nonnegative martingale in Theorem~\ref{theorem:4step} to conclude that 
\( \BI_t := \left \{ m \in [0, 1] : \Kcal_t(m) < 1/\alpha \right \} \text{forms a $(1-\alpha)$-CS for $\mu$.} \)
Unlike the empirical Bernstein CS of Section~\ref{section:warmup}, $\BI_t$ cannot be computed in closed-form. Instead, we theoretically need to compute the family of processes $\{\Kcal_t(m)\}_{m \in [0, 1]}$ and include those $m\in[0, 1]$ for which $\Kcal_t(m)$ remains below $1/\alpha$. This is not practical as the parameter space $[0, 1]$ is uncountably infinite. But if we know a priori that $\BI_t$ is guaranteed to produce an interval for each $t$, then it is straightforward to find a superset of $\BI_t$ by either performing a grid search on $(0, 1/g, 2/g, \dots, (g-1)/g, 1)$ for some large $g \in \mathbb N$, or by employing root-finding algorithms. This motivates the \textit{hedged capital process}, defined for any $\theta, m\in [0, 1]$ as
\begin{align}
    \Kcal_t^\pm(m) &:= \max \left \{ \theta\Kcal_t^+(m), (1-\theta) \Kcal_t^-(m) \right \} \label{eq:hedgedCapitalProcess}, \\
    \text{where}~~ \Kcal_t^+(m) &:=  \prod_{i=1}^t (1 + \lambda^+_i(m)\cdot(X_i - m)), \nonumber\\
    \text{and}~~ \Kcal_t^-(m) &:= \prod_{i=1}^t (1 - \lambda^-_i(m)\cdot(X_i - m)),  \nonumber
\end{align} 
and $(\lambda^+_t(m))_{t=1}^\infty$ and $(\lambda^-_t(m))_{t=1}^\infty$ are predictable sequences of $[0, \tfrac1m)$- and $[0, \tfrac1{1-m})$-valued random variables, respectively. 

$\Kcal_t^\pm(m)$ can be viewed from the betting perspective as dividing one's capital into proportions of $\theta$ and $(1-\theta)$ and making two series of simultaneous bets, positing that $\mu \geq m$, and $\mu < m$, respectively which accumulate capital in $\Kcal_t^+(m)$ and $\Kcal_t^-(m)$. If $\mu \neq m$, then we expect that one of these strategies will perform poorly, while we expect the other to make money in the long term. If $\mu = m$, then we expect neither strategy to make money. The maximum of these processes is upper-bounded by their convex combination, 
\[
\Mcal_t^\pm := \theta \Kcal_t^+ + (1-\theta)\Kcal_t^-.
\] 
Both $\Kcal_t^\pm$ and $\Mcal_t^\pm$ can be used for Step (b) of Theorem~\ref{theorem:4step} to yield a CS. Empirically, both yield intervals, but only the former provably so.

\begin{theorem}[Hedged capital CS {\textbf{[Hedged]}}]
\label{theorem:hedgedCS}
Suppose $(X_t)_{t=1}^\infty \sim P$ for some $P \in \Pcal^\mu$. Let $(\tilde \lambda^+_t)_{t=1}^\infty$ and $(\tilde \lambda^-_t)_{t=1}^\infty$ be real-valued predictable sequences not depending on $m$, and for each $t \geq 1$ let 
\begin{equation} \small \lambda_t^+(m) := |\tilde \lambda_t^+| \land \frac{c}{m}, ~~~ \lambda_t^-(m) := |\tilde \lambda^-_t|\land \frac{c}{1-m},  \end{equation}
for some $c \in [0, 1)$ (some reasonable defaults being $c = 1/2$ or $3/4$).
Then \[ \BI_t^\pm := \left \{ m \in [0, 1] : \Kcal_t^\pm(m) < 1/\alpha \right \} ~~~ \text{forms a $(1-\alpha)$-CS for $\mu$,} \]
as does its running intersection $\bigcap_{i \leq t} \BI_i^\pm$. Further, $\BI_t^\pm$ is an interval for each $t \geq 1$. Finally, replacing $\Kcal_t^\pm(m)$ by $\Mcal_t^\pm(m)$ yields a  tighter $(1-\alpha)$-CS for $\mu$.
\end{theorem}
For reasons given in Section~\ref{section:predmix_goodBetting}, we recommend setting $\tilde \lambda_t^+ = \tilde \lambda_t^- = \lambda_t^\PMpm$ as
\begin{equation}
\label{eq:lambdaPM}
\small
    \lambda_t^\PMpm := \sqrt{\frac{2 \log (2/\alpha)}{\widehat \sigma_{t-1}^2 t \log (t+1) }}, ~~~ \widehat \sigma_{t}^2 := \frac{1/4 + \sum_{i=1}^t (X_i - \widehat \mu_i)^2}{t + 1}, ~~\text{and}~~ \widehat \mu_t := \frac{1/2 + \sum_{i=1}^t X_i}{t + 1},
\end{equation}
for each $t \geq 1$, and truncation level $c := 1/2$ or $3/4$; see Figure~\ref{fig:hedgedVsPM}. A reasonable point estimator for $\mu$ is $\argmin_{m \in [0,1]} \Kcal_t^\pm(m)$ or $\argmin_{m \in [0,1]} \Mcal_t^\pm(m)$ (see Figure~\ref{fig:conf-dist}).

\begin{figure}
 \centering \textbf{Time-uniform confidence sequences: high-variance, symmetric data}
    \includegraphics[width=0.9\textwidth]{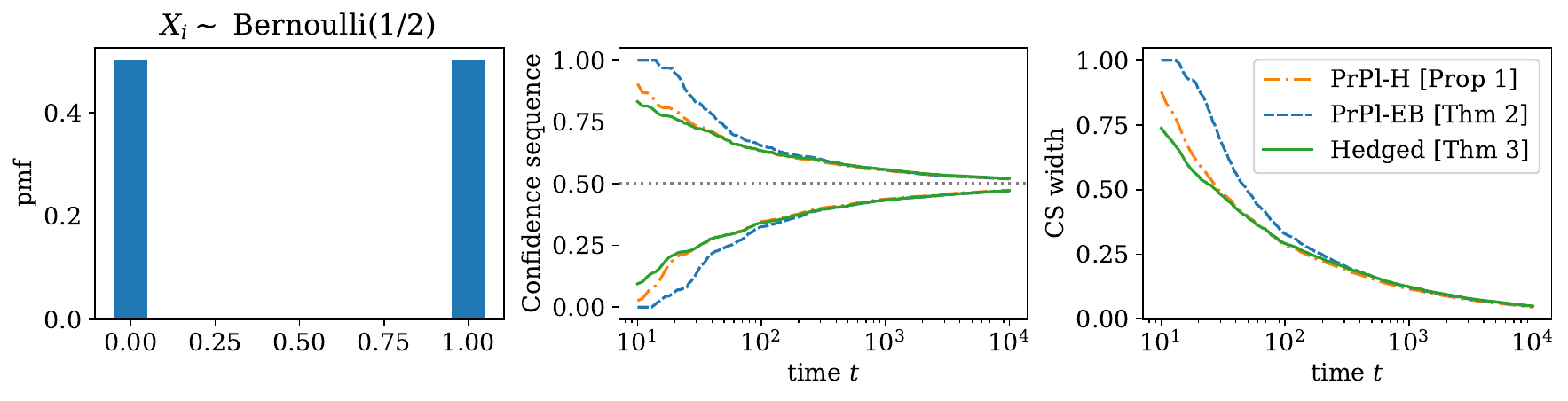}
 \centering \textbf{Time-uniform confidence sequences: low-variance, asymmetric data}
    \centering
    \includegraphics[width=0.9\textwidth]{fig/Time-uniform/Beta_10,_30__time-uniform.pdf}
 \caption{Predictable plug-in Hoeffding, empirical Bernstein, and hedged capital CSs under two distributional scenarios. Notice that the latter roughly matches the others in the Bernoulli$(1/2)$ case, but shines in the low-variance, asymmetric scenario.}
 \label{fig:hedgedVsPM}
\end{figure}

\begin{remark}
Since $\Kcal_t^\pm(m)\leq\Mcal_t^\pm(m)$, the latter confidence sequence is tighter. In the proof of Theorem~\ref{theorem:hedgedCS}, we use a property of the $\max$ function to establish quasiconvexity of $\Kcal_t^\pm(m)$, implying that $\BI_t^\pm$ is an interval. 
We find the difference in empirical performance negligible (Figure~\ref{fig:maxVsSum}).
For the interested reader, Section~\ref{section:aSOSsublevelSets} constructs a (pathological) CS that is almost surely not an interval.
\end{remark}

\begin{figure}[!htbp]
    \centering
    \includegraphics[width=\textwidth]{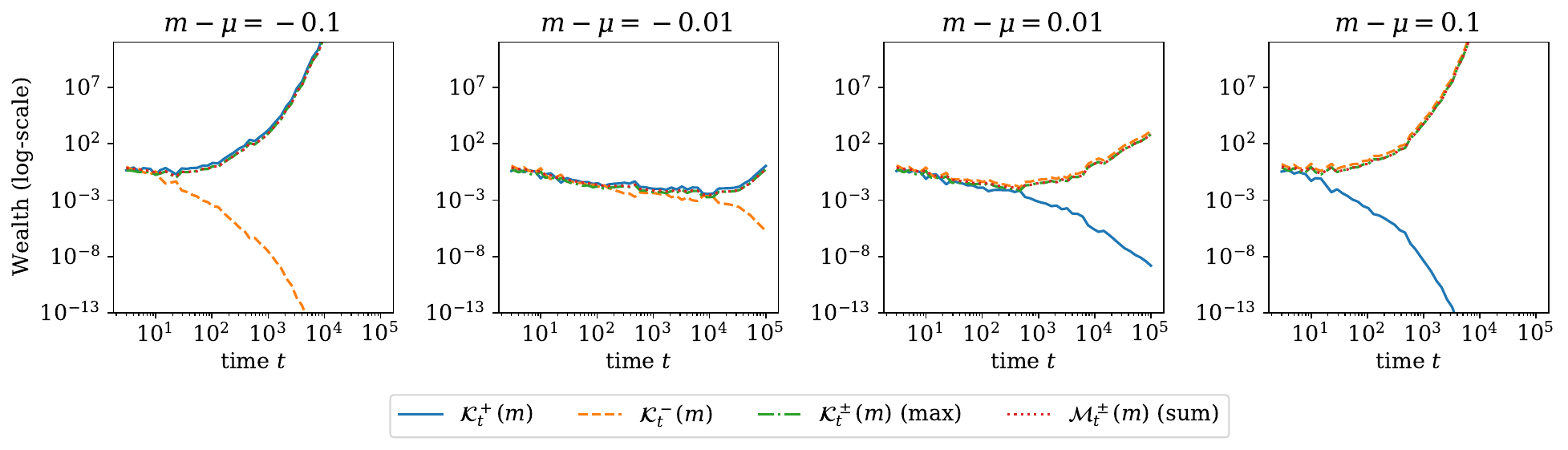}
    \caption{A comparison of capital processes $\Kcal_t^+(m)$, $\Kcal_t^-(m)$, the hedged capital process $\Kcal_t^\pm(m)$, and its upper-bounding nonnegative martingale, $\Mcal_t^\pm(m)$ under four alternatives (from left to right): $m \ll \mu$, $m < \mu$, $m > \mu$, $m \gg \mu$. When $m < \mu$, we see that $\Kcal_t^+(m)$ increases, while $\Kcal_t^-(m)$ approaches zero, but the opposite is true when $m > \mu$. Notice that not much is gained by taking a sum $\Mcal_t^\pm(m)$ rather than a maximum $\Kcal_t^\pm(m)$, since one of $\Kcal_t^+(m)$ and $\Kcal_t^-(m)$ vastly dominates the other, depending on whether $m > \mu$ or $m < \mu$.}
    \label{fig:maxVsSum}
\end{figure}

\begin{remark}
\label{remark:hedgedCI}
Theorem~\ref{theorem:hedgedCS} yields tight hedged CIs for a fixed sample size $n$. 
Recalling~\eqref{eq:lambdaPM}, we recommend using $\bigcap_{i\leq n}\BI_i^\pm$, and setting $\tilde \lambda_t^+ = \tilde \lambda_t^- = \tilde \lambda_t^\pm$ given by
\begin{equation} 
\label{eq:hedgedCI_lambda} \small \tilde \lambda_t^{\pm} := \sqrt{\frac{2 \log (2/\alpha)}{n \widehat \sigma_{t-1}^2}}.
\end{equation}
We refer to the resulting CI as the ``hedged capital confidence interval'' or \textbf{[Hedged-CI]} for short, and demonstrate its superiority to past work in Figure~\ref{fig:hedgedVsOtherCI}.
\end{remark}

Similar to the discussion after Remark~\ref{remark:EBCI}, if $X_1, \dots, X_n$ are independent, then one can permute the data many times and average the resulting capital processes to effectively derandomize the procedure. 

\begin{figure}
 \centering \textbf{Fixed-time confidence intervals: high-variance, symmetric data}
    \includegraphics[width=0.9\textwidth]{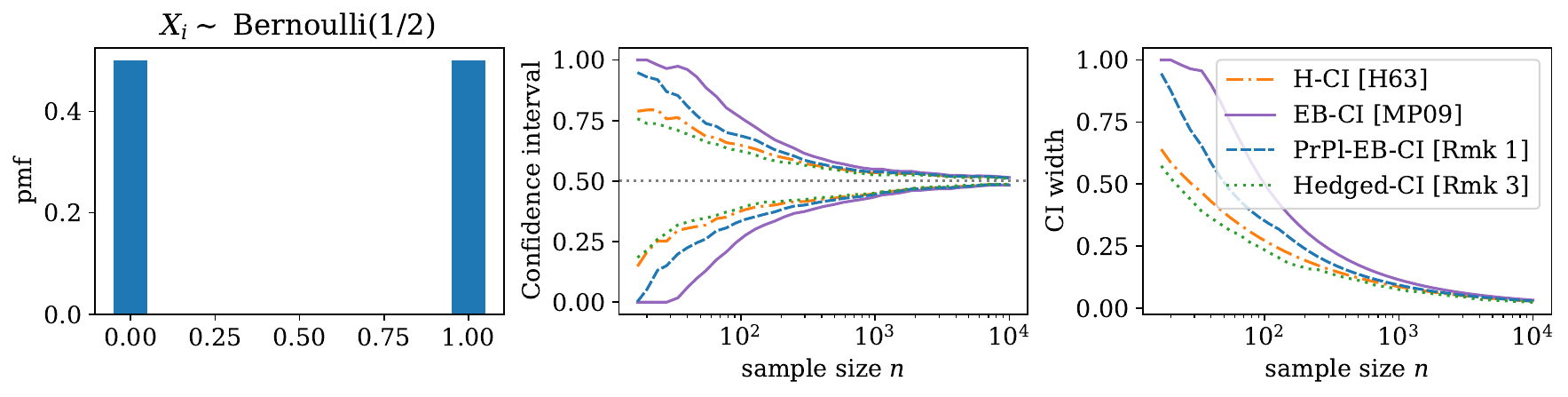}
 \centering \textbf{Fixed-time confidence intervals: low-variance, asymmetric data}
    \centering
    \includegraphics[width=0.9\textwidth]{fig/Fixed-time/Beta_10,_30__fixed-time.pdf}
 \caption{Hoeffding (H), empirical Bernstein (EB), and hedged capital CIs under two distributional scenarios. Similar to the time-uniform setting, the betting approach tends to outperform the other bounds, especially for low-variance, asymmetric data.}
 \label{fig:hedgedVsOtherCI}
\end{figure}
The proof of Theorem~\ref{theorem:hedgedCS} is in Section~\ref{proof:hedgedCS}. Unlike the empirical Bernstein-type CSs and CIs of Section~\ref{section:warmup}, those based on the hedged capital process are not necessarily symmetric. In fact, we empirically find through simulations that these CSs and CIs are able to adapt and benefit from this asymmetry (see Figures~\ref{fig:hedgedVsPM} and \ref{fig:hedgedVsOtherCI}). While it is not obvious from the definition of $\BI_t^\pm$, bets can be chosen such that hedged capital CSs and CIs converge at the optimal rates of $O(\sqrt{\log \log t / t})$ and $O(1/\sqrt{n})$, respectively (see Section~\ref{section:optimalRates}) and such that for sufficiently large $n$, hedged capital CIs almost surely dominate those based on Hoeffding's inequality (see Section~\ref{section:bettingBetterThanHoeffding}). However, the implications of time-uniform convergence rates are subtle, and optimal rates are not always desirable in practical applications (see \cite[Section 3.5]{howard_uniform_2019}). Nevertheless, we find that hedged capital CSs and CIs significantly outperform past works even for small sample sizes (see Section~\ref{section:simulations}). Some additional tools for visualizing CSs across $\alpha$ and $t$ are provided in Section~\ref{subsec:visual}.

In Section~\ref{section:howToBet}, we discuss some guiding principles for deriving powerful betting strategies, presenting the hedged capital CSs and CIs as special cases along with the following game-theoretic betting schemes:
\begin{itemize}\setlength\itemsep{-0.1em}
    \item Growth rate adaptive to the particular alternative (GRAPA),
    \item Approximate GRAPA (aGRAPA),
    \item Lower-bound on the wealth (LBOW),
    \item Online Newton step-$m$ (ONS-$m$),
    \item Diversified Kelly betting (dKelly),
    \item Confidence boundary bets (ConBo), and
    \item Sequentially rebalanced portfolio (SRP).
\end{itemize}
Each of these betting strategies have their respective benefits, whether computational, conceptual, or statistical which are discussed further in Section~\ref{section:howToBet}.


\section{Betting while sampling without replacement (WoR)}
\label{sec:WoR}


This section tackles a slightly different problem, that of sampling without replacement (WoR) from a finite set of real numbers in order to estimate its mean. Importantly, the $N$ numbers in the finite population $(x_1,\dots,x_N)$ are fixed and nonrandom. 
What is random is only the order of observation; the model for sampling uniformly at random without replacement (WoR) posits that at time $t \geq 1$,
\begin{equation}\label{eq:sampling-wo-replacement}
X_t \mid (X_1,\dots,X_{t-1}) \sim \text{Uniform} \left((x_1,\dots,x_N) \backslash (X_1,\dots,X_{t-1}) \right ).
\end{equation}
All probabilities are thus to be understood as solely arising from observing fixed entities in a random order, with no distributional assumptions being made on the finite population. 
We consider the same canonical filtration $\Fcal = (\mathcal F_t)_{t =0}^N$ as before. For $t \geq 1$, let $\Fcal_t := \sigma(X_1^t)$ be the sigma-field generated by $X_1, \dots, X_t$ and let $\Fcal_0$ be the empty sigma-field.
For succinctness, we use the notation $[a]:=\{1,\dots,a\}$.

For each $m \in [0,1]$, let $\Lcal^m:=\{ x_1^N \in [0,1]^N : \sum_{i=1}^N x_i / N = m \}$ be the set of all unordered lists of $N \geq 2$ real numbers in $[0,1]$ whose average is $m$. For instance, $\Lcal^0$ and $\Lcal^1$ are both singletons, but otherwise $\Lcal^m$ is uncountably infinite. Let $\Pcal^m$ be the set of all measures on $\Fcal_N$ that are formed as follows: pick an arbitrary element of $\Lcal^m$, apply a uniformly random permutation, and reveal the elements one by one. Thus, every element of $\Pcal^m$ is a uniform measure on the $N!$ permutations of some element in $\Lcal^m$, so there is a one-to-one mapping between $\Lcal^m$ and $\Pcal^m$.

 Define $\Pcal := \bigcup_{m} \Pcal^m$ and let $\mu$ represent the true unknown mean, meaning that the data is drawn from some $P \in \Pcal^\mu$. For every $m\in[0,1]$, we posit a composite null hypothesis $H^0_m: P \in \Pcal^m$, but clearly only one of these nulls is true. We will design betting strategies to test these nulls and thus find efficient confidence intervals or sequences for $\mu$.
It is easier to present the sequential case first, since that is arguably more natural for sampling WoR, and discuss the fixed-time case later.

\subsection{Existing (super)martingale-based confidence sequences or tests}

Several papers have considered estimating the mean of a finite set of nonrandom numbers when sampling WoR, often by constructing concentration inequalities \citep{hoeffding_probability_1963, serfling1974probability, bardenet2015concentration, waudby2020confidence}.
Notably, \citet{hoeffding_probability_1963} showed that the same bound for sampling with replacement \eqref{eq:hoeffding} can be used when sampling WoR. \citet{serfling1974probability} improved on this bound, which was then further refined by \citet{bardenet2015concentration}.
While test supermartingales appeared in some of the aforementioned works, \citet{waudby2020confidence} identified better test supermartingales which yield explicit Hoeffding- and empirical Bernstein-type concentration inequalities and CSs for sampling WoR that significantly improved on previous bounds. Consider their exponential Hoeffding-type supermartingale,
\begin{equation}
\small	\label{eqn:exponentialHoeffding}
			M_t^\HWoR := \exp \left \{ \sum_{i=1}^t \left [ \lambda_i \left ( X_i - \mu + \frac{1}{N-(i-1)}\sum_{j=1}^{i-1} (X_j - \mu) \right ) - \psi_H(\lambda_i)\right] \right \},
\end{equation}
and their exponential empirical Bernstein-type supermartingale,
\begin{equation}
\small
\label{eqn:exponentialEmpiricalBernstein}
	M_t^\EBWoR := \exp \left \{ \sum_{i=1}^t \left [ \lambda_i \left ( X_i - \mu + \frac{1}{N-(i-1)}\sum_{j=1}^{i-1} (X_j - \mu) \right ) - v_i\psi_E(\lambda_i)\right] \right \},
\end{equation}
where $(\lambda_t)_{t=1}^N$ is any predictable $\lambda$-sequence (real-valued for $M_t^\HWoR$, but $[0, 1)$-valued for $M_t^\EBWoR$), $v_i = 4(X_i - \widehat \mu_{i-1})^2$ as before, and $\psi_H(\cdot)$ and $\psi_E(\cdot)$ are defined as in Section~\ref{section:warmup}.
Defining $M_0^\HWoR \equiv M_0^\EBWoR := 1$, \cite{waudby2020confidence} prove that $(M_t^\HWoR)_{t=0}^N$ and $(M_t^\EBWoR)_{t=0}^N$ are test supermartingales with respect to $\Fcal$, and hence can be used in Step (b) of Theorem~\ref{theorem:4step}.

In recent work on election audits, \citet{stark2020sets} credits Harold Kaplan for proposing
\begin{equation}
\small
M_t^K := \int_0^1 \prod_{i=1}^t \left (1 + \gamma\left [ X_i \frac{1-(i-1)/N}{ \mu - \sum_{j=1}^{i-1}X_i / N} - 1 \right ] \right ) d\gamma.
\end{equation}
The ``Kaplan martingale'' $(M_t^K)_{t=0}^N$ was employed for election auditing, but it is a polynomial of degree $t$ and is computationally expensive for large $t$~\citep{stark2020sets}.

Next, we mimic the approach of Section~\ref{section:capitalProcess} to derive a capital process for sampling WoR. We then derive WoR analogues of the efficient betting strategies from Section~\ref{section:howToBet}.

\subsection{The capital process for sampling without replacement}
Define the predictable sequence $(\mu^\WoR_t)_{t\in [N]}$ where 
\begin{equation} \small \mu^\WoR_t := \mathbb E[X_t | \mathcal F_{t-1}] = \frac{N\mu - \sum_{i=1}^{t-1} X_i}{N-(t-1)}. \end{equation}
It is clear that $\mu_t^\WoR \in [0,1]$, since it is the mean of the unobserved elements of $\{x_i\}_{i \in [N]}$. $(\mu_t^\WoR)_{t\in[N]}$ is unobserved since $\mu$ is unknown, so it is helpful to define
\begin{equation}
\small
m_t^\WoR := \frac{N m - \sum_{i=1}^{t-1} X_i}{N-(t-1)}.
\end{equation}
Now, let $(\lambda_t(m))_{t=1}^N$ be a predictable sequence such that $\lambda_t(m)$ is $\left (-\tfrac1{1-m_t^\WoR}, \tfrac1{m_t^\WoR}\right )$-valued. Define the \textit{without-replacement capital process} $\Kcal_t^{\WoR}(m)$,
\begin{equation}
\small
\Kcal^\WoR_t(m) := \prod_{i=1}^t \left(1 + \lambda_i(m) \cdot (X_i - m_i^\WoR) \right)
\end{equation}
with $\Kcal^\WoR_0(m) := 1$. The following result is analogous to Proposition~\ref{proposition:betting}.
\begin{proposition}
\label{proposition:bettingWoR}
    Let $X_1^N$ be a WoR sample from $x_1^N\in [0, 1]^N$. The following two statements imply each other:
    \begin{enumerate}
        \item $\EE_P\left (X_t \mid \Fcal_{t-1} \right) = \mu_t^\WoR$ for each $t \in [N]$.
        \item For every predictable sequence with $\lambda_t(m) \in \left (-\tfrac{1}{(1-\mu_t^\WoR)},\tfrac{1}{\mu_t^\WoR}\right )$, $(\Kcal^\WoR_t(\mu))_{t=0}^\infty$ is a test martingale.
    \end{enumerate}
\end{proposition}

The other claims within Proposition~\ref{proposition:betting} also hold above with minor modification, but we do not mention them again for brevity.
Further, Proposition~\ref{prop:universal} technically covers WoR sampling as well.
We now present a ``hedged'' capital process and powerful betting schemes for sampling WoR, to construct a CS for $\mu = \frac{1}{N} \sum_{i=1}^N x_i$.

\subsection{Powerful betting schemes} 

Similar to Section~\ref{section:hedgedCapitalProcess}, define the hedged capital process for sampling WoR: 
\begin{align*}
\Kcal_t^{\WoR, \pm}(m) := \max \Bigg \{ & \theta \prod_{i=1}^t \left (1 + \lambda_i^+(m)\cdot (X_i - m_t^\WoR ) \right ), \\ 
&(1-\theta) \prod_{i=1}^t \left (1 - \lambda_i^-(m)\cdot (X_i - m_t^\WoR ) \right ) \Bigg \} 
\end{align*}
for some predictable $(\lambda_t^+(m))_{t=1}^N$ and $(\lambda_t^-(m))_{t=1}^N$ taking values in $[0, 1/m_t^\WoR]$ and $[0, 1/(1-m_t^\WoR)]$ at time $t$, respectively. Using $\left(\Kcal_t^{\WoR,\pm}(m)\right)_{t=0}^\infty$ as the process in Step (b) of Theorem~\ref{theorem:4step}, we obtain the CS summarized in the following theorem.
\begin{theorem}[WoR hedged capital CS {\textbf{[Hedged-WoR]}}]
\label{theorem:hedgedCSWoR}
Given a finite population $x_1^N \in [0, 1]^N$ with mean $\mu := \frac{1}{N} \sum_{i=1}^N x_i = \mu$, suppose that $X_1, X_2, \dots X_N$ are sampled WoR from $x_1^N$. 
Let $(\dot \lambda^+_t)_{t=1}^\infty$ and $(\dot \lambda^-_t)_{t=1}^\infty$ be real-valued predictable sequences not depending on $m$, and for each $t \geq 1$ let 
\[ \lambda_t^+(m) := |\dot \lambda_t^+| \land \frac{c}{m_t^\WoR}, ~~~ \lambda_t^-(m) := |\dot \lambda^-_t|\land \frac{c}{1-m_t^\WoR},  \]
for some $c \in [0, 1)$ (some reasonable defaults being $c = 1/2$ or $3/4$).
Then
\[ \BI_t^{\pm, \WoR} := \left \{ m \in [0, 1] : \Kcal_t^{\pm, \WoR}(m) < 1/\alpha \right \} ~~~ \text{forms a $(1-\alpha)$-CS for $\mu$,} \]
as does $\bigcap_{i \leq t} \BI_i^{\pm, \WoR}$. Furthermore, $\BI_t^{\pm, \WoR}$ is an interval for each $t \geq 1$.
\end{theorem}
The proof of Theorem~\ref{theorem:hedgedCSWoR} is in Section~\ref{proof:hedgedCSWoR}. We recommend setting $\dot \lambda_t^+ = \dot \lambda_t^- = \lambda_t^\PMpm$ as was done earlier in \eqref{eq:lambdaPM}; for each $t \geq 1$, and $c := 1/2$, let 
\[
\lambda_t^\PMpm := \sqrt{\frac{2 \log (2/\alpha)}{\widehat \sigma_{t-1}^2 t \log (t+1) }}, ~~ \widehat \sigma_{t}^2 := \frac{\tfrac14 + \sum_{i=1}^t (X_i - \widehat \mu_i)^2}{t + 1}, ~\text{and}~ \widehat \mu_t := \frac{\tfrac12 + \sum_{i=1}^t X_i}{t + 1},
\]
See Figure~\ref{fig:WoR} for a comparison to the best prior work.

\begin{figure}[!htbp]
\centering \textbf{WoR time-uniform confidence sequences: high-variance, symmetric data}
    \includegraphics[width=0.88\textwidth]{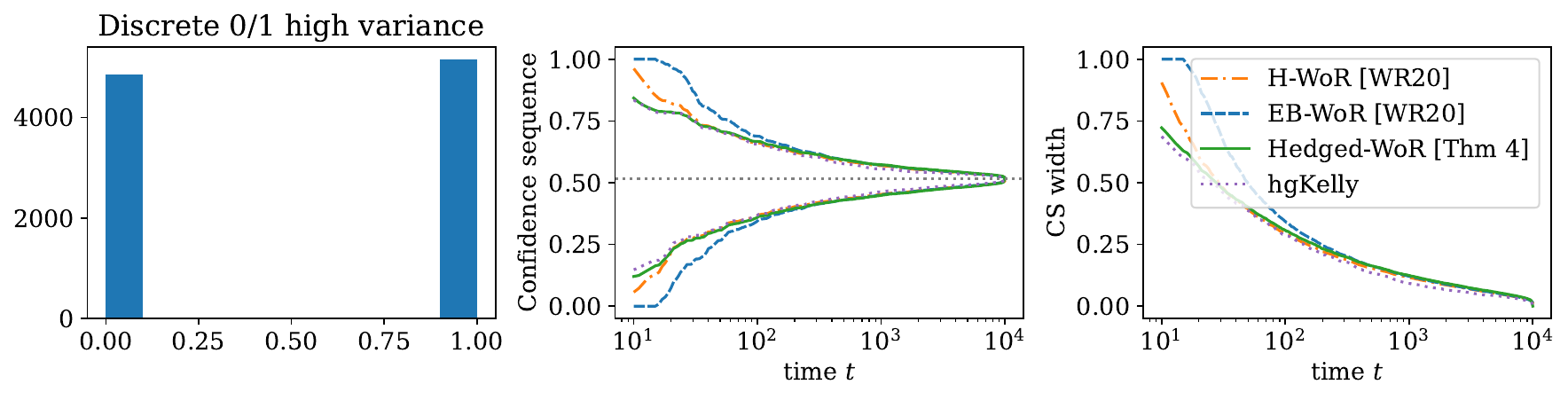}
    
\centering \textbf{WoR time-uniform confidence sequences: low-variance, asymmetric data}
    \centering
    \includegraphics[width=0.88\textwidth]{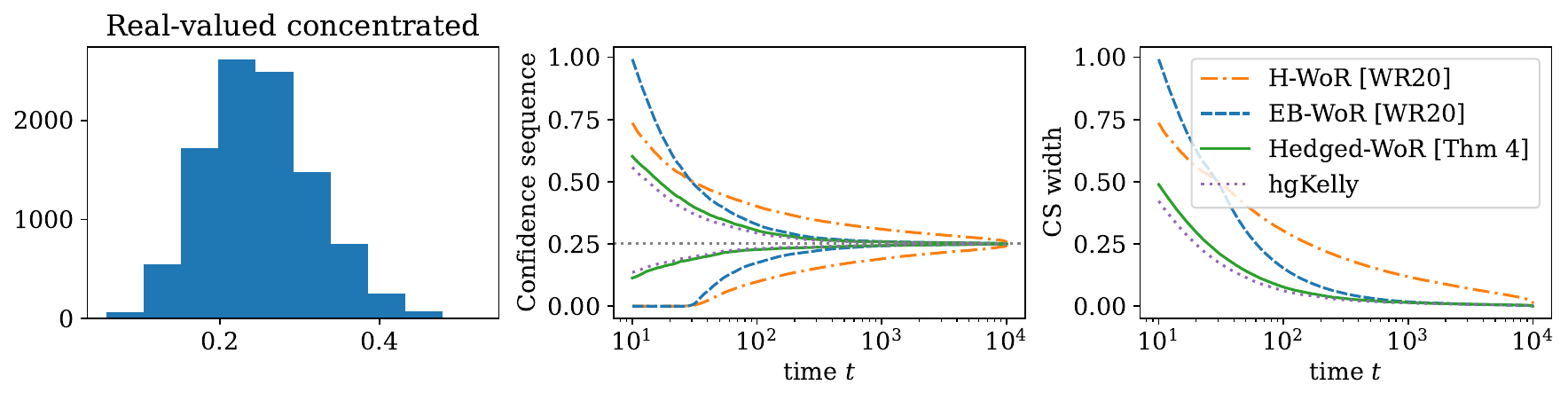}
    \caption{\small Without-replacement betting CSs versus the predictable plug-in supermartingale-based CSs \citep{waudby2020confidence}. Similar to the with-replacement case, the betting approach matches or vastly outperforms past state-of-the art methods.}
    \label{fig:WoR}
\end{figure}
\begin{remark}
\label{remark:hedgedCIWoR}
As before, we can use Theorem~\ref{theorem:hedgedCSWoR} to derive powerful CIs for the mean of a nonrandom set of bounded numbers given a fixed sample size $n \leq N$. 
We recommend using $\bigcap_{i\leq n}\BI_i^{\pm, \WoR}$, and setting $\dot \lambda_t^+ = \dot \lambda_t^- = \dot \lambda_t^\pm$ as in \eqref{eq:hedgedCI_lambda}:
\(
\small \dot \lambda_t^{\pm} := \sqrt{\frac{2 \log (2/\alpha)}{n \widehat \sigma_{t-1}^2}}.
\)
We refer to the resulting CI as
\textbf{[Hedged-WoR-CI]};
see Figure~\ref{fig:WoRCI}.
\end{remark}
\begin{figure}[!htbp]
\centering \textbf{WoR fixed-time confidence intervals: high-variance, symmetric data}
    \includegraphics[width=0.88\textwidth]{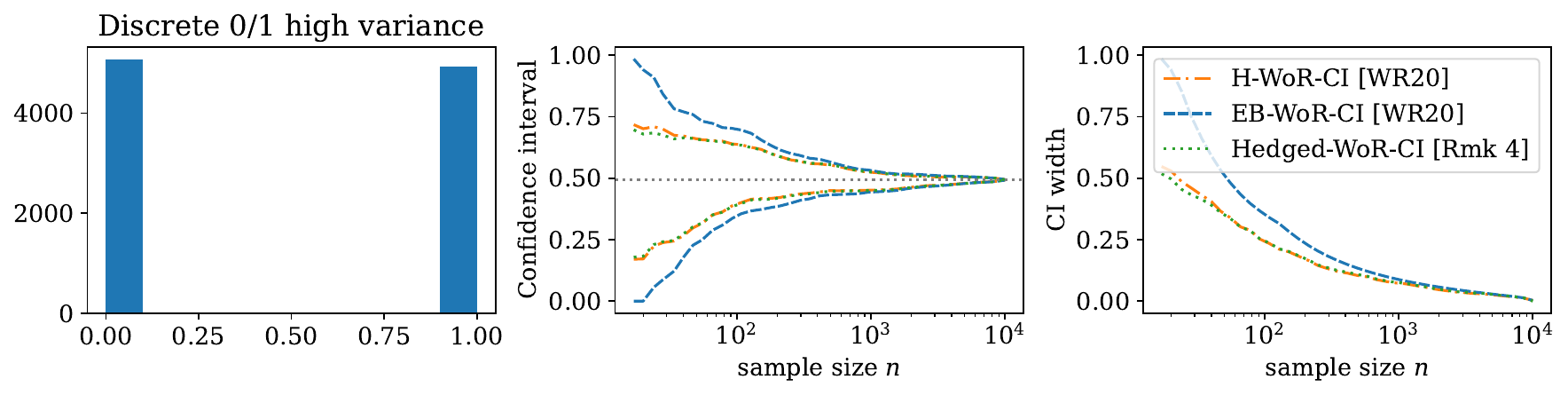}
 \hfill
 
\centering \textbf{WoR fixed-time confidence intervals: low-variance, asymmetric data}
    \centering
    \includegraphics[width=0.88\textwidth]{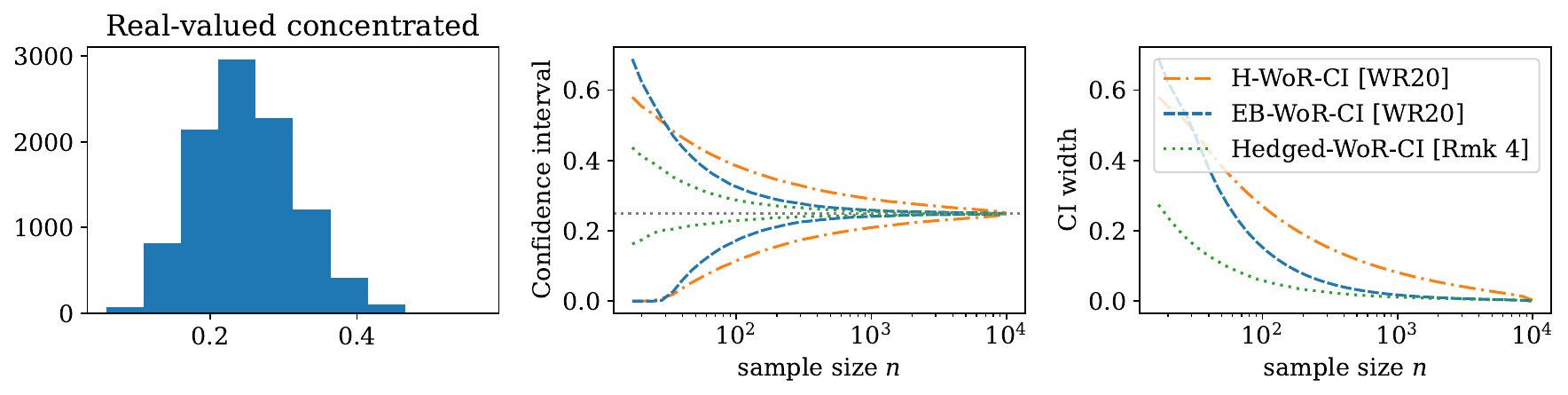}
    \caption{\small WoR analogue of the hedged capital CI versus the WoR Hoeffding- and empirical Bernstein-type CIs \citep{waudby2020confidence}. Similar to with-replacement, the betting approach has excellent performance.}
    \label{fig:WoRCI}
\end{figure}

Notice that constructing a WoR test martingale only relies on changing the fixed conditional mean $\mu$ to the time-varying conditional mean $\mu_t^\WoR := \frac{N\mu - \sum_{i=1}^{t-1} X_i }{N-t+1}$ and now designing $(-1/(1-\mu_t^\WoR), 1/\mu_t^\WoR)$-valued bets instead of $(-1/(1-\mu), 1/\mu)$-valued ones. In this way, it is possible to adapt any of the betting strategies in Section~\ref{section:howToBet} to sampling WoR, yielding a wide array of solutions to this estimation problem.

\subsection{Relationship to composite null testing}
This paper focuses primarily on estimation, but we end with a note that our CSs (or CIs) yield valid, sequential (or batch) tests for composite null hypotheses $H_0: \mu \in S$ for any $S \subset [0,1]$. Specifically, for any of our capital processes $\Kcal_t(m)$, 
\[
\mathfrak{p}_t :=  \sup_{m \in S} \frac1{\Kcal_t(m)}
\] 
is an ``anytime-valid p-value'' for $H_0$, as is $\widetilde{\mathfrak{p}}_t := \inf_{s \leq t} \mathfrak{p}_s$, meaning that 
\[
\sup_{P \in \bigcup_{m \in S}\Pcal^m} P(\mathfrak{p}_\tau \leq \alpha) \leq \alpha \text{ for arbitrary stopping times $\tau$.}
\]
Alternately, $\mathfrak{p}_t$ is also the smallest $\alpha$ for which our $(1-\alpha)$-CS does not intersect $S$.
Similarly, $\mathfrak{e}_t := \inf_{m \in S} \Kcal_t(m)$ is an ``e-process'' for $H_0$, meaning that
\[
\sup_{P \in \bigcup_{m \in S}\Pcal^m} \EE_P[\mathfrak{e}_\tau] \leq 1 \text{ for arbitrary stopping times $\tau$.}
\]
For more details on inference at arbitrary stopping times, we refer the reader to \citet{howard_exponential_2018,howard_uniform_2019,grunwald_safe_2019,ramdas2020admissible}.

\section{A brief selective history on betting and its mathematical applications}\label{section:history-short}

From a purely statistical perspective, this paper could be viewed as tackling the problem of deriving sharp confidence sets for means of bounded random variables. In this pursuit, we find that a technique with excellent empirical performance happens to have strong connections to the topics of betting and gambling. While we provide a more detailed discussion in Section~\ref{section:history}, here we briefly summarize some of the ways in which betting ideas have appeared in and shaped probability, statistical inference, information theory, and online learning, in the broad context of our paper.
\begin{itemize}
    \item \textbf{Probability}: The 1939 PhD thesis of~\cite{ville_etude_1939} brought betting and martingales to the forefront of modern probability theory, by giving actionable interpretations to Kolmogorov's newly developed measure-theoretic probability, and dealing a near-fatal blow to the theory of collectives by von Mises. Ville showed that for \emph{any} event $A$ of probability measure zero (like sequences violating the law of large numbers), he could design an explicit betting strategy that never bets more than it has, whose wealth (a test martingale) grows without to infinity if the event $A$ occurs. Ville worked with binary sequences, but his result holds more generally; see~\cite{shafer_probability_2005}. 
    
    One may view Ville's result as a theorem in measure-theoretic probability theory; what he effectively proved was: the event that a test (super)martingale exceeds $1/\alpha$ has probability at most $\alpha$ (Ville's inequality in this paper). This holds for any $\alpha \in [0,1]$, treating $1/0 \equiv +\infty$, with the $\alpha=0$ case being the most remarkable part. But Ville's result is also an axiomatic building block for \emph{game-theoretic probability} \citep{vovk1993logic,shafer_probability_2005,shafer2019game}. Many classical results in probability can been derived in completely game-theoretic terms~\citep{shafer_probability_2005,shafer2019game}. The capital processes used for deriving CSs are of the same form as those used to derive these foundational theorems of game-theoretic probability, despite the two goals being quite different. 
    
    \item \textbf{Statistical inference}: The famous book of \citet{wald_sequential_1945} was the first to thoroughly present and study sequential hypothesis testing. Despite not being presented in this way by Wald, we know in hindsight that the sequential probability ratio test (SPRT) is quite centrally based on the fact that the likelihood ratio is a nonnegative martingale. Two decades later, Robbins and colleagues built on Wald's sequential testing work in several ways, including to estimation via confidence sequences \citep{darling_confidence_1967,darling_inequalities_1967,darling_iterated_1967,robbins_iterated_1968,robbins_probability_1969,robbins_boundary_1970, robbins_class_1972, robbins_expected_1974,robbins_statistical_1970,lai_confidence_1976}. The recent work of \citet{howard_exponential_2018,howard_uniform_2019,ramdas2021exch,wasserman2020universal} extends the early work of Wald, Robbins and colleagues to a broader class of problems using exponential supermartingales and ``e-processes'', which can be seen as nonparametric, composite generalizations of the SPRT martingale.
    Connections between \emph{betting} and the works of Wald, Robbins et al., and Howard et al. are implicit in those works, but can now be seen in hindsight, and our paper makes these connections explicit.  
    \item \textbf{Information theory}: Working in the new field of information theory, \citet{kelly1956new} made direct connections to betting by showing that the capacity of a channel (itself fundamentally related to entropy and the Kullback-Leibler divergence) is given by the maximal rate of growth of wealth of a gambler in a simple game with iid Bernoulli$(p)$ observations and known $p$. \cite{breiman1961optimal} generalized Kelly's results significantly, and \citet{krichevsky1981performance} extended these results beyond the case of known $p$ using a mixture method. Thomas Cover's interest in these techniques spans several decades~\citep{cover1974universal,cover1984algorithm,cover1987log,bell1980competitive,bell1988game}, culminating  in his famous universal portfolio algorithm \citep{cover1991universal}. The results of Krichevsky-Trofimov and Cover are essentially regret inequalities, leading directly to the final subfield below.
    \item \textbf{Online learning}: The techniques of Krichevsky, Trofimov and Cover found extensive applications to \emph{sequential prediction with the logarithmic loss} \citep{cesa2006prediction}. Here, one derives \emph{regret inequalities} for the total loss accumulated when predicting the next observation from  a potentially adversarial sequence. This problem is fundamentally connected to online convex optimization, for which Orabona and colleagues use parameter-free betting algorithms to derive regret inequalities \citep{orabona2016coin,orabona2017training,jun2017improved,cutkosky2018black,jun2019parameter}.
\citet{rakhlin2017equivalence} articulated a deep connection between martingale concentration and deterministic regret inequalities, and \citet[Section 7.1]{jun2019parameter}  derive concentration bounds for the general setting of Banach space-valued observations with sub-exponential noise. 
\end{itemize}

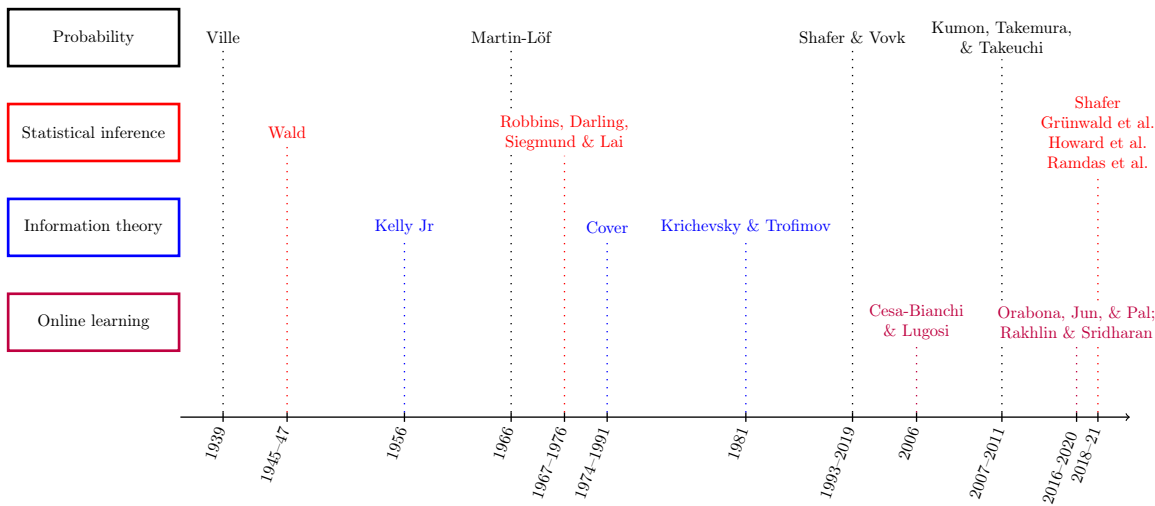
\begin{figure}[!htbp]
\resizebox{\textwidth}{!}{
\begin{tikzpicture}[]
    
    \newcommand\timelineStartX{0}
    \newcommand\timelineEndX{20}
    \coordinate (coord_beginning) at (\timelineStartX, 0) {} ;
    \coordinate (coord_now) at (\timelineEndX, 0) {} ;

    \draw[->, black, thick, text=black]
        (coord_beginning)--(coord_now);
    
    \newcommand\probY{8}
    \newcommand\statY{6}
    \newcommand\infoY{4}
    \newcommand\onlineY{2}
    \newcommand\probColour{black}
    \newcommand\statColour{red}
    \newcommand\infoColour{blue}
    \newcommand\onlineColour{purple}
    \newcommand{\ticksBottomY}{-0.1}
    \newcommand{\ticksTopY}{0.1}
    
    \newcommand{\literatureNodeWidth}{3.6cm}
    \newcommand{\literatureNodeHeight}{1.2cm}

    \node[fill=white, draw = \probColour, ultra thick, align = center, minimum height = \literatureNodeHeight, minimum width = \literatureNodeWidth, anchor=east]  (node_begin_OL)  at(0,\probY) {Probability} ;
    \node[fill=white, draw = \statColour, ultra thick, align = center, minimum height = \literatureNodeHeight,minimum width = \literatureNodeWidth, anchor=east]  (node_begin_OL)  at(0,\statY) {Statistical inference} ;
    \node[fill=white, draw = \infoColour, ultra thick, align = center, minimum height = \literatureNodeHeight,minimum width = \literatureNodeWidth, anchor=east]  (node_begin_OL)  at(0,\infoY) {Information theory} ;
    \node[fill=white, draw = \onlineColour, ultra thick, align = center, minimum height = \literatureNodeHeight, minimum width = \literatureNodeWidth, anchor=east]  (node_begin_OL)  at(0,\onlineY) {Online learning} ;
        
    \newcommand{\lowerYear}{1935}
    \newcommand{\upperYear}{2024}

    \foreach \yrForX\y\colour\auth\yrForDisplay in {
        1939/\probY/\probColour/Ville/1939,
        1966/\probY/\probColour/Martin-L\"of/1966,
        1998/\probY/\probColour/Shafer \& Vovk/1993--2019,
        2012/\probY/\probColour/Kumon{,} Takemura{,}\\\& Takeuchi/2007--2011,
        1945/\statY/\statColour/Wald/1945--47,
        1971/\statY/\statColour/Robbins{,} Darling{,}\\ Siegmund \& Lai/1967--1976,
        2021/\statY/\statColour/Shafer\\Gr\"unwald et al.\\Howard et al.\\Ramdas et al./2018--21,
        1956/\infoY/\infoColour/Kelly Jr/1956,
        1975/\infoY/\infoColour/Cover/1974--1991,
        1988/\infoY/\infoColour/Krichevsky \& Trofimov/1981,
        2004/\onlineY/\onlineColour/Cesa-Bianchi\\\& Lugosi/2006,
        2019/\onlineY/\onlineColour/Orabona{,} Jun{,} \& Pal;\\Rakhlin \& Sridharan/2016--2020
    }
    {
        \FPeval{\x}{(\yrForX - \lowerYear) / (\upperYear - \lowerYear) * (\timelineEndX - \timelineStartX)}
        \draw[-, black, thick, text=black]
            (\x, \ticksBottomY)--(\x, \ticksTopY);
        \node[align=center, anchor=east, rotate=70] at (\x, \ticksBottomY) {\yrForDisplay};
        \node[fill=white, align = center, anchor=center, text=\colour](authorNode) at (\x,\y) {\auth};
        \draw[loosely dotted, \colour, thick, text=black] (\x, \ticksTopY)--(authorNode) ;
    }

    

    \newcommand\villeX{1}

\end{tikzpicture}
}
\caption{A brief selective history of betting ideas appearing (often implicitly) in various literatures. As discussed further in Section~\ref{section:history}, these subfields have rarely cited each other, but ideas are now beginning to permeate. Several authors did not explicitly use the language of betting, and their inclusion above is due to reinterpreting their work in hindsight.} 
\label{fig:history}
\end{figure}

\section{Summary}

Nonparametric confidence sequences are particularly useful in sequential estimation because they enable valid inference at arbitrary stopping times, but they are underappreciated as powerful tools to provide accurate inference even at fixed times. Recent work~\citep{howard_exponential_2018,howard_uniform_2019} has developed several time-uniform generalizations of the Cramer-Chernoff technique utilizing ``line-crossing'' inequalities and using various variants of Robbins' method of mixtures (discrete mixtures, conjugate mixtures and stitching) to convert them to ``curve-crossing'' inequalities. 

This work adds new techniques to the toolkit: to complement the aforementioned mixture methods, we develop a ``predictable plug-in'' approach. When coupled with existing nonparametric supermartingales, it yields (for example) computationally efficient empirical-Bernstein confidence sequences. One of our major contributions is to thoroughly develop the theory and methodology for a new nonnegative martingale approach to estimating means of bounded random variables in both with- and without-replacement settings. These convincingly outperform all existing published work that we are aware of, for CIs and CSs, both with and without replacement.

Our methods are particularly easy to interpret in terms of evolving capital processes and sequential testing by betting~\citep{shafer2019language} but we go much further by developing powerful and efficient betting strategies that lead to state-of-the-art variance-adaptive confidence sets that are significantly tighter than past work in all considered settings. In particular, Shafer espouses \emph{complementary} benefits of such approaches, ranging from improved scientific communication, ties to historical advances in probability, and reproducibility via continued experimentation (also see~\cite{grunwald_safe_2019}), but our focus here has been on developing a new state of the art for a set of classical, fundamental problems. 

There appear to be nontrivial connections to online learning theory~\citep{kotlowski2010following,kumon2011sequential,orabona2017training,cutkosky2018black}, and to empirical  and dual likelihoods (see Section~\ref{section:EL} and an extended historical review of betting in Section~\ref{section:history}). The   reductions from regret inequalities to concentration bounds described in~\cite{rakhlin2017equivalence} and~\cite{jun2019parameter} are fascinating, but existing published bounds are loose in the constants and not competitive in practice compared to our direct approach. Exploring deeper connections may yield other confidence sequences or betting strategies.

It is clear to us, and hopefully to the reader as well, that the ideas behind this work (adaptive statistical inference by betting) form the tip of the iceberg---they lead to powerful, efficient, nonasymptotic, nonparametric inference and can be adapted to a range of other problems. As just one example, let $\Pcal^{p,q}$ represent the set of all continuous distributions such that the $p$-quantile of $X_t$, conditional on the past, is equal to $q$. This is also a nonparametric, convex set of distributions with no common reference measure. Nevertheless, for any predictable $(\lambda_i)$, it is easy to check that
\[
M_t = \prod_{i=1}^t (1 + \lambda_i (\mathbf{1}_{X_i \leq q} - p))
\]
is a test martingale for $\Pcal^{p,q}$. Setting $p=1/2$ and $q=0$, for example, we can sequentially test if the median of the underlying data distribution is the origin. The continuity assumption can be relaxed, and this test can be inverted to get a confidence sequence for any quantile. We do not pursue this idea further in the current paper because the recent (rather different) nonnegative martingale methods of~\citet{howard2019sequential} already provide a challenging benchmark for that problem. Typically, one test martingale-based method cannot  uniformly dominate another, and the large gains in this paper were made possible because all previous published approaches implicitly or explicitly employed test \emph{super}martingales, while we employ test martingales that are computationally simple to implement.

To conclude, we opine that ``game-theoretic statistical inference'' is in its nascency, and we expect much theoretical and practical progress in coming years. We hope the reader shares our excitement in this regard.


\paragraph{Acknowledgments.}
AR acknowledges funding from NSF DMS 1916320, an Adobe faculty research award and an NSF DMS (CAREER) 1945266. This work used the Extreme Science and Engineering Discovery Environment (XSEDE), which is supported by National Science Foundation grant number ACI-1548562. Specifically, it used the Bridges system, which is supported by NSF award number ACI-1445606, at the Pittsburgh Supercomputing Center (PSC) \citep{Nystrom:2015:BUF:2792745.2792775}.  The authors thank Harrie Hendriks, Philip Stark, Francesco Orabona, Kwang-Sung Jun, Nikos Karampatziakis and Arun Kuchibhotla for discussions on an early preprint, as well as Glenn Shafer, Vladimir Vovk and Peter Gr\"unwald for broader discussions.

\bibliographystyle{plainnat}
\bibliography{main.bib}

\begin{thebibliography}{89}
\providecommand{\natexlab}[1]{#1}
\providecommand{\url}[1]{\texttt{#1}}
\expandafter\ifx\csname urlstyle\endcsname\relax
  \providecommand{\doi}[1]{doi: #1}\else
  \providecommand{\doi}{doi: \begingroup \urlstyle{rm}\Url}\fi

\bibitem[Abbasi-Yadkori et~al.(2011)Abbasi-Yadkori, P{\'a}l, and
  Szepesv{\'a}ri]{abbasi2011improved}
Yasin Abbasi-Yadkori, D{\'a}vid P{\'a}l, and Csaba Szepesv{\'a}ri.
\newblock Improved algorithms for linear stochastic bandits.
\newblock In \emph{Advances in Neural Information Processing Systems}, pages
  2312--2320, 2011.

\bibitem[Anderson(1969)]{anderson1969confidence}
Theodore Anderson.
\newblock Confidence limits for the expected value of an arbitrary bounded
  random variable with a continuous distribution function.
\newblock Technical report, Stanford University Department of Statistics, 1969.

\bibitem[Audibert et~al.(2007)Audibert, Munos, and
  Szepesv{\'a}ri]{audibert2007tuning}
Jean-Yves Audibert, R{\'e}mi Munos, and Csaba Szepesv{\'a}ri.
\newblock Tuning bandit algorithms in stochastic environments.
\newblock In \emph{Algorithmic Learning Theory}, 2007.

\bibitem[Balsubramani(2014)]{balsubramani_sharp_2014}
Akshay Balsubramani.
\newblock Sharp finite-time iterated-logarithm martingale concentration.
\newblock \emph{arXiv:1405.2639}, 2014.

\bibitem[Bardenet and Maillard(2015)]{bardenet2015concentration}
R{\'e}mi Bardenet and Odalric-Ambrym Maillard.
\newblock Concentration inequalities for sampling without replacement.
\newblock \emph{Bernoulli}, 21\penalty0 (3):\penalty0 1361--1385, 2015.

\bibitem[Bell and Cover(1988)]{bell1988game}
Robert Bell and Thomas~M Cover.
\newblock Game-theoretic optimal portfolios.
\newblock \emph{Management Science}, 34\penalty0 (6):\penalty0 724--733, 1988.

\bibitem[Bell and Cover(1980)]{bell1980competitive}
Robert~M Bell and Thomas~M Cover.
\newblock Competitive optimality of logarithmic investment.
\newblock \emph{Mathematics of Operations Research}, pages 161--166, 1980.

\bibitem[Bennett(1962)]{bennett_probability_1962}
George Bennett.
\newblock Probability {Inequalities} for the {Sum} of {Independent} {Random}
  {Variables}.
\newblock \emph{Journal of the American Statistical Association}, 57\penalty0
  (297):\penalty0 33--45, 1962.

\bibitem[Bentkus(2004)]{bentkus2004hoeffding}
Vidmantas Bentkus.
\newblock On {Hoeffding’s} inequalities.
\newblock \emph{The Annals of Probability}, 32\penalty0 (2):\penalty0
  1650--1673, 2004.

\bibitem[Bentkus et~al.(2006)Bentkus, Kalosha, and
  Van~Zuijlen]{bentkus2006domination}
Vidmantas Bentkus, N~Kalosha, and M~Van~Zuijlen.
\newblock On domination of tail probabilities of (super) martingales: explicit
  bounds.
\newblock \emph{Lithuanian Mathematical Journal}, 46\penalty0 (1):\penalty0
  1--43, 2006.

\bibitem[Bernstein(1927)]{bernstein_theory_1927}
Sergei Bernstein.
\newblock \emph{Theory of probability}.
\newblock Gastehizdat Publishing House, 1927.

\bibitem[Boucheron et~al.(2013)Boucheron, Lugosi, and
  Massart]{boucheron_concentration_2013}
Stéphane Boucheron, Gábor Lugosi, and Pascal Massart.
\newblock \emph{Concentration inequalities: a nonasymptotic theory of
  independence}.
\newblock Oxford University Press, Oxford, 1st edition, 2013.

\bibitem[Breiman(1961)]{breiman1961optimal}
Leo Breiman.
\newblock Optimal gambling systems for favorable games.
\newblock \emph{Berkeley Symposium on Mathematical Statistics and Probability},
  4.1\penalty0 (65-78), 1961.

\bibitem[Capp{\'e} et~al.(2013)Capp{\'e}, Garivier, Maillard, Munos, and
  Stoltz]{cappe2013kullback}
Olivier Capp{\'e}, Aur{\'e}lien Garivier, Odalric-Ambrym Maillard, R{\'e}mi
  Munos, and Gilles Stoltz.
\newblock {Kullback-Leibler} upper confidence bounds for optimal sequential
  allocation.
\newblock \emph{The Annals of Statistics}, pages 1516--1541, 2013.

\bibitem[Cesa-Bianchi and Lugosi(2006)]{cesa2006prediction}
Nicol\`{o} Cesa-Bianchi and G{\'a}bor Lugosi.
\newblock \emph{Prediction, learning, and games}.
\newblock Cambridge university press, 2006.

\bibitem[Cover(1974)]{cover1974universal}
Thomas~M Cover.
\newblock Universal gambling schemes and the complexity measures of
  {K}olmogorov and {C}haitin.
\newblock \emph{Technical Report, no. 12}, 1974.

\bibitem[Cover(1984)]{cover1984algorithm}
Thomas~M Cover.
\newblock An algorithm for maximizing expected log investment return.
\newblock \emph{IEEE Transactions on Information Theory}, 30\penalty0
  (2):\penalty0 369--373, 1984.

\bibitem[Cover(1987)]{cover1987log}
Thomas~M Cover.
\newblock Log optimal portfolios.
\newblock In \emph{Chapter in “Gambling Research: Gambling and Risk
  Taking,” Seventh International Conference}, volume~4, 1987.

\bibitem[Cover(1991)]{cover1991universal}
Thomas~M Cover.
\newblock Universal portfolios.
\newblock \emph{Mathematical Finance}, 1\penalty0 (1):\penalty0 1--29, 1991.

\bibitem[Cutkosky and Orabona(2018)]{cutkosky2018black}
Ashok Cutkosky and Francesco Orabona.
\newblock Black-box reductions for parameter-free online learning in {B}anach
  spaces.
\newblock In \emph{Proceedings of the 31st Conference On Learning Theory},
  volume~75, 2018.

\bibitem[Darling and Robbins(1967{\natexlab{a}})]{darling_confidence_1967}
D.~A. Darling and Herbert Robbins.
\newblock Confidence sequences for mean, variance, and median.
\newblock \emph{Proceedings of the National Academy of Sciences}, 58\penalty0
  (1):\penalty0 66--68, 1967{\natexlab{a}}.

\bibitem[Darling and Robbins(1967{\natexlab{b}})]{darling_inequalities_1967}
D.~A. Darling and Herbert Robbins.
\newblock Inequalities for the {Sequence} of {Sample} {Means}.
\newblock \emph{Proceedings of the National Academy of Sciences}, 57\penalty0
  (6):\penalty0 1577--1580, 1967{\natexlab{b}}.

\bibitem[Darling and Robbins(1967{\natexlab{c}})]{darling_iterated_1967}
D.~A. Darling and Herbert Robbins.
\newblock Iterated {Logarithm} {Inequalities}.
\newblock \emph{Proceedings of the National Academy of Sciences}, 57\penalty0
  (5):\penalty0 1188--1192, 1967{\natexlab{c}}.

\bibitem[Dawid(1984)]{dawid1984present}
A~Philip Dawid.
\newblock Present position and potential developments: Some personal views
  statistical theory the prequential approach.
\newblock \emph{Journal of the Royal Statistical Society: Series A (General)},
  147\penalty0 (2):\penalty0 278--290, 1984.

\bibitem[Dawid(1997)]{dawid1997prequential}
A~Philip Dawid.
\newblock Prequential analysis.
\newblock \emph{Encyclopedia of Statistical Sciences}, 1:\penalty0 464--470,
  1997.

\bibitem[de~la Peña et~al.(2004)de~la Peña, Klass, and
  Lai]{de_la_pena_self-normalized_2004}
Victor~H. de~la Peña, Michael~J. Klass, and Tze~Leung Lai.
\newblock Self-normalized processes: exponential inequalities, moment bounds
  and iterated logarithm laws.
\newblock \emph{The Annals of Probability}, 32\penalty0 (3):\penalty0
  1902--1933, 2004.
\newblock ISSN 0091-1798, 2168-894X.

\bibitem[de~la Peña et~al.(2007)de~la Peña, Klass, and
  Lai]{de_la_pena_pseudo-maximization_2007}
Victor~H. de~la Peña, Michael~J. Klass, and Tze~Leung Lai.
\newblock Pseudo-maximization and self-normalized processes.
\newblock \emph{Probability Surveys}, 4:\penalty0 172--192, 2007.

\bibitem[de~la Peña et~al.(2009)de~la Peña, Lai, and
  Shao]{de_la_pena_self-normalized_2009}
Victor~H. de~la Peña, Tze~Leung Lai, and Qi-Man Shao.
\newblock \emph{Self-normalized processes: limit theory and statistical
  applications}.
\newblock Springer, Berlin, 2009.

\bibitem[Doob(1953)]{doob1953stochastic}
Joseph~Leo Doob.
\newblock \emph{Stochastic Processes}, volume~10.
\newblock New York Wiley, 1953.

\bibitem[Fan et~al.(2015)Fan, Grama, and Liu]{fan_exponential_2015}
Xiequan Fan, Ion Grama, and Quansheng Liu.
\newblock Exponential inequalities for martingales with applications.
\newblock \emph{Electronic Journal of Probability}, 20\penalty0 (1):\penalty0
  1--22, 2015.

\bibitem[Gr{\"u}nwald(2007)]{grunwald2007minimum}
Peter Gr{\"u}nwald.
\newblock \emph{The minimum description length principle}.
\newblock MIT press, 2007.

\bibitem[Gr{\"u}nwald et~al.(2019)Gr{\"u}nwald, de~Heide, and
  Koolen]{grunwald_safe_2019}
Peter Gr{\"u}nwald, Rianne de~Heide, and Wouter Koolen.
\newblock Safe testing.
\newblock \emph{arXiv preprint arXiv:1906.07801}, 2019.

\bibitem[Hall and La~Scala(1990)]{hall1990methodology}
Peter Hall and Barbara La~Scala.
\newblock Methodology and algorithms of empirical likelihood.
\newblock \emph{International Statistical Review}, pages 109--127, 1990.

\bibitem[Hendriks(2018)]{hendriks2018test}
Harrie Hendriks.
\newblock Test martingales for bounded random variables.
\newblock \emph{arXiv preprint arXiv:1801.09418}, 2018.

\bibitem[Hoeffding(1963)]{hoeffding_probability_1963}
Wassily Hoeffding.
\newblock Probability inequalities for sums of bounded random variables.
\newblock \emph{Journal of the American Statistical Association}, 58\penalty0
  (301):\penalty0 13--30, March 1963.

\bibitem[Honda and Takemura(2010)]{honda2010asymptotically}
Junya Honda and Akimichi Takemura.
\newblock An asymptotically optimal bandit algorithm for bounded support
  models.
\newblock In \emph{COLT}, pages 67--79. Citeseer, 2010.

\bibitem[Howard and Ramdas(2022)]{howard2019sequential}
Steven~R. Howard and Aaditya Ramdas.
\newblock {Sequential estimation of quantiles with applications to A/B testing
  and best-arm identification}.
\newblock \emph{Bernoulli}, 28\penalty0 (3):\penalty0 1704 -- 1728, 2022.

\bibitem[Howard et~al.(2020)Howard, Ramdas, McAuliffe, and
  Sekhon]{howard_exponential_2018}
Steven~R. Howard, Aaditya Ramdas, Jon McAuliffe, and Jasjeet Sekhon.
\newblock Time-uniform {Chernoff} bounds via nonnegative supermartingales.
\newblock \emph{Probability Surveys}, 17:\penalty0 257--317, 2020.

\bibitem[Howard et~al.(2021)Howard, Ramdas, McAuliffe, and
  Sekhon]{howard_uniform_2019}
Steven~R. Howard, Aaditya Ramdas, Jon McAuliffe, and Jasjeet Sekhon.
\newblock Time-uniform, nonparametric, nonasymptotic confidence sequences.
\newblock \emph{The Annals of Statistics}, 49\penalty0 (2):\penalty0
  1055--1080, 2021.

\bibitem[Jun and Orabona(2019)]{jun2019parameter}
Kwang-Sung Jun and Francesco Orabona.
\newblock Parameter-free online convex optimization with sub-exponential noise.
\newblock In \emph{Conference on Learning Theory}, pages 1802--1823. PMLR,
  2019.

\bibitem[Jun et~al.(2017)Jun, Orabona, Wright, and Willett]{jun2017improved}
Kwang-Sung Jun, Francesco Orabona, Stephen Wright, and Rebecca Willett.
\newblock Improved strongly adaptive online learning using coin betting.
\newblock In \emph{Artificial Intelligence and Statistics}, pages 943--951.
  PMLR, 2017.

\bibitem[Kaufmann and Koolen(2021)]{kaufmann2018mixture}
Emilie Kaufmann and Wouter~M. Koolen.
\newblock Mixture martingales revisited with applications to sequential tests
  and confidence intervals.
\newblock \emph{Journal of Machine Learning Research}, 22\penalty0
  (246):\penalty0 1--44, 2021.

\bibitem[Kearns and Saul(1998)]{kearns_large_1998}
Michael Kearns and Lawrence Saul.
\newblock Large deviation methods for approximate probabilistic inference.
\newblock In \emph{Proceedings of the {Fourteenth} {Conference} on
  {Uncertainty} in {Artificial} {Intelligence}}, {UAI}'98, pages 311--319,
  1998.

\bibitem[Kelly~Jr(1956)]{kelly1956new}
John~L Kelly~Jr.
\newblock A new interpretation of information rate.
\newblock \emph{Bell System Technical Journal}, 35\penalty0 (4):\penalty0
  917--926, 1956.

\bibitem[Kot{\l}owski et~al.(2010)Kot{\l}owski, Gr{\"u}nwald, and
  De~Rooij]{kotlowski2010following}
Wojciech Kot{\l}owski, Peter Gr{\"u}nwald, and Steven De~Rooij.
\newblock Following the flattened leader.
\newblock In \emph{Conference on Learning Theory}, pages 106--118. Citeseer,
  2010.

\bibitem[Krichevsky and Trofimov(1981)]{krichevsky1981performance}
Raphail Krichevsky and Victor Trofimov.
\newblock The performance of universal encoding.
\newblock \emph{IEEE Transactions on Information Theory}, 27\penalty0
  (2):\penalty0 199--207, 1981.

\bibitem[Kuchibhotla and Zheng(2021)]{kuchibhotla2020near}
Arun~K Kuchibhotla and Qinqing Zheng.
\newblock Near-optimal confidence sequences for bounded random variables.
\newblock \emph{International Conference on Machine Learning}, 2021.

\bibitem[Kumon et~al.(2011)Kumon, Takemura, and Takeuchi]{kumon2011sequential}
Masayuki Kumon, Akimichi Takemura, and Kei Takeuchi.
\newblock Sequential optimizing strategy in multi-dimensional bounded
  forecasting games.
\newblock \emph{Stochastic Processes and their Applications}, 121\penalty0
  (1):\penalty0 155--183, 2011.

\bibitem[Lai(1976)]{lai_confidence_1976}
Tze~Leung Lai.
\newblock On {Confidence} {Sequences}.
\newblock \emph{The Annals of Statistics}, 4\penalty0 (2):\penalty0 265--280,
  March 1976.
\newblock ISSN 0090-5364, 2168-8966.

\bibitem[Learned-Miller and Thomas(2019)]{learned2019new}
Erik Learned-Miller and Philip~S Thomas.
\newblock A new confidence interval for the mean of a bounded random variable.
\newblock \emph{arXiv preprint arXiv:1905.06208}, 2019.

\bibitem[Li(1999)]{li1999estimation}
Qiang~Jonathan Li.
\newblock \emph{Estimation of mixture models}.
\newblock PhD thesis, Yale University, 1999.

\bibitem[Lorden and Pollak(2005)]{lorden_nonanticipating_2005}
Gary Lorden and Moshe Pollak.
\newblock Nonanticipating estimation applied to sequential analysis and
  changepoint detection.
\newblock \emph{The Annals of Statistics}, 33\penalty0 (3):\penalty0
  1422--1454, June 2005.
\newblock ISSN 0090-5364, 2168-8966.

\bibitem[Martin-L{\"o}f(1966)]{martin1966definition}
Per Martin-L{\"o}f.
\newblock The definition of random sequences.
\newblock \emph{Information and control}, 9\penalty0 (6):\penalty0 602--619,
  1966.

\bibitem[Maurer and Pontil(2009)]{maurer_empirical_2009}
Andreas Maurer and Massimiliano Pontil.
\newblock Empirical {Bernstein} bounds and sample variance penalization.
\newblock In \emph{Proceedings of the Conference on Learning Theory}, 2009.

\bibitem[McKerns et~al.(2011)McKerns, Strand, Sullivan, Fang, and
  Aivazis]{mckerns2012building}
Michael McKerns, Leif Strand, Tim Sullivan, Alta Fang, and Michael Aivazis.
\newblock Building a framework for predictive science.
\newblock \emph{Proceedings of the 10th Python in Science Conference}, 2011.

\bibitem[Mnih et~al.(2008)Mnih, Szepesvári, and Audibert]{mnih_empirical_2008}
Volodymyr Mnih, Csaba Szepesvári, and Jean-Yves Audibert.
\newblock Empirical {B}ernstein stopping.
\newblock In \emph{Proceedings of the 25th International Conference on
  {Machine} Learning}, pages 672--679. ACM, 2008.

\bibitem[Mykland(1995)]{mykland1995dual}
Per~Aslak Mykland.
\newblock Dual likelihood.
\newblock \emph{The Annals of Statistics}, pages 396--421, 1995.

\bibitem[Nystrom et~al.(2015)Nystrom, Levine, Roskies, and
  Scott]{Nystrom:2015:BUF:2792745.2792775}
Nicholas~A. Nystrom, Michael~J. Levine, Ralph~Z. Roskies, and J.~Ray Scott.
\newblock Bridges: A uniquely flexible hpc resource for new communities and
  data analytics.
\newblock In \emph{Proceedings of the 2015 XSEDE Conference: Scientific
  Advancements Enabled by Enhanced Cyberinfrastructure}, XSEDE '15, pages
  30:1--30:8, New York, NY, USA, 2015. ACM.
\newblock ISBN 978-1-4503-3720-5.
\newblock \doi{10.1145/2792745.2792775}.
\newblock URL \url{http://doi.acm.org/10.1145/2792745.2792775}.

\bibitem[Orabona and Pal(2016)]{orabona2016coin}
Francesco Orabona and David Pal.
\newblock Coin betting and parameter-free online learning.
\newblock \emph{Advances in Neural Information Processing Systems},
  29:\penalty0 577--585, 2016.

\bibitem[Orabona and Tommasi(2017)]{orabona2017training}
Francesco Orabona and Tatiana Tommasi.
\newblock Training deep networks without learning rates through coin betting.
\newblock In \emph{Advances in Neural Information Processing Systems}, pages
  2160--2170, 2017.

\bibitem[Owen(2001)]{owen2001empirical}
Art~B Owen.
\newblock \emph{Empirical likelihood}.
\newblock CRC press, 2001.

\bibitem[Phan et~al.(2021)Phan, Thomas, and Learned-Miller]{phan2021towards}
My~Phan, Philip~S Thomas, and Erik Learned-Miller.
\newblock Towards practical mean bounds for small samples.
\newblock \emph{International Conference on Machine Learning}, 2021.

\bibitem[Rakhlin and Sridharan(2017)]{rakhlin2017equivalence}
Alexander Rakhlin and Karthik Sridharan.
\newblock On equivalence of martingale tail bounds and deterministic regret
  inequalities.
\newblock In \emph{Conference on Learning Theory}, pages 1704--1722. PMLR,
  2017.

\bibitem[Ramdas et~al.(2020)Ramdas, Ruf, Larsson, and
  Koolen]{ramdas2020admissible}
Aaditya Ramdas, Johannes Ruf, Martin Larsson, and Wouter Koolen.
\newblock Admissible anytime-valid sequential inference must rely on
  nonnegative martingales.
\newblock \emph{arXiv preprint arXiv:2009.03167}, 2020.

\bibitem[Ramdas et~al.(2021)Ramdas, Ruf, Larsson, and Koolen]{ramdas2021exch}
Aaditya Ramdas, Johannes Ruf, Martin Larsson, and Wouter~M Koolen.
\newblock Testing exchangeability: Fork-convexity, supermartingales and
  e-processes.
\newblock \emph{International Journal of Approximate Reasoning}, 2021.

\bibitem[Rissanen(1984)]{rissanen1984universal}
Jorma Rissanen.
\newblock Universal coding, information, prediction, and estimation.
\newblock \emph{IEEE Transactions on Information Theory}, 30\penalty0
  (4):\penalty0 629--636, 1984.

\bibitem[Rissanen(1998)]{rissanen1998stochastic}
Jorma Rissanen.
\newblock \emph{Stochastic Complexity in Statistical Inquiry}, volume~15.
\newblock World Scientific, 1998.

\bibitem[Robbins(1970)]{robbins_statistical_1970}
Herbert Robbins.
\newblock Statistical methods related to the law of the iterated logarithm.
\newblock \emph{The Annals of Mathematical Statistics}, 41\penalty0
  (5):\penalty0 1397--1409, 1970.

\bibitem[Robbins and Siegmund(1968)]{robbins_iterated_1968}
Herbert Robbins and David Siegmund.
\newblock Iterated logarithm inequalities and related statistical procedures.
\newblock In \emph{Mathematics of the {Decision} {Sciences}, {Part} {II}},
  pages 267--279. American Mathematical Society, Providence, 1968.

\bibitem[Robbins and Siegmund(1969)]{robbins_probability_1969}
Herbert Robbins and David Siegmund.
\newblock Probability distributions related to the law of the iterated
  logarithm.
\newblock \emph{Proc. of the National Academy of Sciences}, 62\penalty0
  (1):\penalty0 11--13, January 1969.
\newblock ISSN 0027-8424.

\bibitem[Robbins and Siegmund(1970)]{robbins_boundary_1970}
Herbert Robbins and David Siegmund.
\newblock Boundary {crossing} {probabilities} for the {Wiener} {process} and
  {sample} {sums}.
\newblock \emph{The Annals of Mathematical Statistics}, 41\penalty0
  (5):\penalty0 1410--1429, 1970.

\bibitem[Robbins and Siegmund(1972)]{robbins_class_1972}
Herbert Robbins and David Siegmund.
\newblock A class of stopping rules for testing parametric hypotheses.
\newblock In \emph{Proceedings of the Sixth Berkeley Symposium on Mathematical
  Statistics and Probability}, volume~4, pages 37--41, 1972.

\bibitem[Robbins and Siegmund(1974)]{robbins_expected_1974}
Herbert Robbins and David Siegmund.
\newblock The expected sample size of some tests of power one.
\newblock \emph{The Annals of Statistics}, 2\penalty0 (3):\penalty0 415--436,
  May 1974.

\bibitem[Serfling(1974)]{serfling1974probability}
Robert~J Serfling.
\newblock Probability inequalities for the sum in sampling without replacement.
\newblock \emph{The Annals of Statistics}, pages 39--48, 1974.

\bibitem[Shafer(2021)]{shafer2019language}
Glenn Shafer.
\newblock The language of betting as a strategy for statistical and scientific
  communication.
\newblock \emph{Journal of the Royal Statistical Society, Series A}, 2021.

\bibitem[Shafer and Vovk(2001)]{shafer_probability_2005}
Glenn Shafer and Vladimir Vovk.
\newblock \emph{Probability and {Finance}: {It}'s {Only} a {Game}!}
\newblock John Wiley \& Sons, February 2001.
\newblock ISBN 978-0-471-46171-5.

\bibitem[Shafer and Vovk(2019)]{shafer2019game}
Glenn Shafer and Vladimir Vovk.
\newblock \emph{Game-Theoretic Foundations for Probability and Finance}.
\newblock John Wiley \& Sons, 2019.

\bibitem[Shafer et~al.(2011)Shafer, Shen, Vereshchagin, and
  Vovk]{shafer_test_2011}
Glenn Shafer, Alexander Shen, Nikolai Vereshchagin, and Vladimir Vovk.
\newblock Test {Martingales}, {Bayes} {Factors} and $p$-{Values}.
\newblock \emph{Statistical Science}, 26\penalty0 (1):\penalty0 84--101, 2011.

\bibitem[Shannon(1948)]{shannon1948mathematical}
Claude~Elwood Shannon.
\newblock A mathematical theory of communication.
\newblock \emph{The Bell System Technical Journal}, 27\penalty0 (3):\penalty0
  379--423, 1948.

\bibitem[Stark(2020)]{stark2020sets}
Philip~B Stark.
\newblock Sets of half-average nulls generate risk-limiting audits: {SHANGRLA}.
\newblock In \emph{International Conference on Financial Cryptography and Data
  Security}, pages 319--336. Springer, 2020.

\bibitem[Ville(1939)]{ville_etude_1939}
Jean Ville.
\newblock \emph{Étude {Critique} de la {Notion} de {Collectif}.}
\newblock PhD thesis, Paris, 1939.

\bibitem[Vovk(2021)]{vovk_testing_randomness_2019}
Vladimir Vovk.
\newblock Testing randomness online.
\newblock \emph{Statistical Science}, 36\penalty0 (4):\penalty0 595--611, 2021.

\bibitem[Vovk et~al.(2005)Vovk, Gammerman, and Shafer]{vovk2005algorithmic}
Vladimir Vovk, Alexander Gammerman, and Glenn Shafer.
\newblock \emph{Algorithmic learning in a random world}.
\newblock Springer Science \& Business Media, 2005.

\bibitem[Vovk(1993)]{vovk1993logic}
Vladimir~G Vovk.
\newblock A logic of probability, with application to the foundations of
  statistics.
\newblock \emph{Journal of the Royal Statistical Society: Series B
  (Methodological)}, 55\penalty0 (2):\penalty0 317--341, 1993.

\bibitem[Wald(1945)]{wald_sequential_1945}
Abraham Wald.
\newblock Sequential {Tests} of {Statistical} {Hypotheses}.
\newblock \emph{Annals of Mathematical Statistics}, 16\penalty0 (2):\penalty0
  117--186, 1945.

\bibitem[Wald(1947)]{wald_sequential_1947}
Abraham Wald.
\newblock \emph{Sequential {Analysis}}.
\newblock John Wiley \& Sons, New York, 1947.

\bibitem[Wasserman et~al.(2020)Wasserman, Ramdas, and
  Balakrishnan]{wasserman2020universal}
Larry Wasserman, Aaditya Ramdas, and Sivaraman Balakrishnan.
\newblock Universal inference.
\newblock \emph{Proceedings of the National Academy of Sciences}, 2020.
\newblock ISSN 0027-8424.

\bibitem[Waudby-Smith and Ramdas(2020)]{waudby2020confidence}
Ian Waudby-Smith and Aaditya Ramdas.
\newblock Confidence sequences for sampling without replacement.
\newblock \emph{Advances in Neural Information Processing Systems}, 33, 2020.

\bibitem[Zhao et~al.(2016)Zhao, Zhou, Sabharwal, and Ermon]{zhao_adaptive_2016}
Shengjia Zhao, Enze Zhou, Ashish Sabharwal, and Stefano Ermon.
\newblock Adaptive {concentration} {inequalities} for {sequential} {decision}
  {problems}.
\newblock In \emph{30th {Conference} on {Neural} {Information} {Processing}
  {Systems}}, 2016.

\end{thebibliography}

\newpage
\appendix

\section{Proofs of main results}
\label{sec:app-proofs}
We first introduce a lemma which will aid in the proofs to follow. 



\begin{lemma}[Predictable plug-in Chernoff supermartingales]
\label{lemma:predmix}
Suppose that $X_1,X_2,\dots \sim P$, and for some $\mu,v_t$ and $\psi(\lambda)$, we have that for any $\lambda \in \Lambda \subseteq \mathbb R$,
\begin{equation}
\label{eq:superMG}
    \EE_P\left [ \exp(\lambda(X_t - \mu) - v_t\psi(\lambda)) \mid \Fcal_{t-1} \right ] \leq 1 \text{ ~ for each $t \geq 1$ }.
\end{equation}
Then, for any  $\Lambda$-valued sequence $(\lambda_t)_{t=1}^\infty$ that is predictable with respect to $\Fcal$,
\[ M_t^\psi(\mu) := \prod_{i=1}^t \exp \left ( \lambda_i(X_i - \mu) - v_i\psi(\lambda_i) \right )\]
forms a test supermartingale with respect to $\Fcal$.
\end{lemma}

\begin{proof}
Writing out the conditional expectation of $M_t^\psi$ for any $t \geq 2$,
\begin{align*}
    \EE \left ( M_t^\psi(\mu) \mid \Fcal_{t-1} \right ) &= \EE \left ( \prod_{i=1}^t \exp\left ( \lambda_i(X_i -\mu) - v_i \psi(\lambda_i)  \right ) \Bigm | \Fcal_{t-1} \right )\\
    &\stackrel{(i)}{=} \prod_{i=1}^{t-1} \exp\left ( \lambda_i(X_i -\mu) - v_i \psi(\lambda_i)  \right ) \underbrace{\EE\left [ \exp \left ( \lambda_t (X_t - \mu) - v_t \psi(\lambda_t) \right ) \mid \Fcal_{t-1} \right ]}_{\leq 1 \text{ by assumption}}\\
    &= M_{t-1}^\psi(\mu),
\end{align*}
where $(i)$ follows from the fact that $\exp\left ( \lambda_i(X_i - \mu) - v_i\psi(\lambda_i)\right)$ is $\Fcal_{t-1}$-measurable for $i \leq t-1$. Since $\Fcal_0$ was assumed to be trivial, for $M_1$ we have that 
\[ \EE[M_1^\psi(\mu)|\Fcal_0] = \underbrace{\EE \left [ \exp\left ( \lambda_1(X_1 -\mu) - v_1 \psi(\lambda_1)  \right ) \right ]}_{\leq 1 \text{ by assumption}},\]
which completes the proof.
\myqed
\end{proof}

\subsection{Proof of Proposition~\ref{proposition:predmixHoeffding}}
\label{proof:hoeffding}
The proof proceeds in three steps. First, apply a standard MGF bound by \citet{hoeffding_probability_1963}. Second, we apply Lemma~\ref{lemma:predmix}. Finally, we apply Theorem~\ref{theorem:4step} to obtain a CS and take a union bound.

\bigskip
\noindent \textbf{Step 1.} By \citet{hoeffding_probability_1963}, we have that $\EE \left [\exp(\lambda_t(X_t - \mu) - \psi_H(\lambda_t) )\mid \Fcal_{t-1} \right ] \leq 1$ since $X_t \in [0, 1]$ almost surely and since $\lambda_t$ is $\Fcal_{t-1}$-measurable.

\bigskip
\noindent \textbf{Step 2.} By Step 1 and Lemma~\ref{lemma:predmix}, we have that 
\[ M_t^\PMH(\mu) := \prod_{i=1}^t \exp \left ( \lambda_i(X_i - \mu) - \psi_H(\lambda_i) \right ) \]
forms a test supermartingale.

\bigskip
\noindent \textbf{Step 3.} By Step 2 combined with Theorem~\ref{theorem:4step}, we have that
\begin{align*}
    &P\left ( \exists t \geq 1 : \mu \leq \frac{\sum_{i=1}^t \lambda_i X_i }{ \sum_{i=1}^t \lambda_i } - \frac{\log(1/\alpha) + \sum_{i=1}^t \psi_H(\lambda_i)}{\sum_{i=1}^t \lambda_i} \right )\\
    =\ &P\left ( \exists t \geq 1 : M_t^\PMH(\mu) \geq 1/\alpha \right ) \leq \alpha.
\end{align*}
Applying the same bound to $(-X_t)_{t=1}^\infty$ with mean $-\mu$ and taking a union bound, we have the desired result,
\[P\left ( \exists t \geq 1 : \mu \notin \left ( \frac{\sum_{i=1}^t \lambda_i X_i }{ \sum_{i=1}^t \lambda_i } \pm \frac{\log(2/\alpha) + \sum_{i=1}^t \psi_H(\lambda_i)}{\sum_{i=1}^t \lambda_i} \right ) \right ) \leq \alpha, \]
which completes the proof.
\qed

\subsection{Proof of Theorem~\ref{theorem:EBCS}}
\label{proof:EBCS}
By Lemma~\ref{lemma:predmix} combined with Theorem~\ref{theorem:4step}, it suffices to prove that 
\[ \EE_P\left [ \exp\left \{ \lambda_t(X_t - \mu) - v_t \psi_E(\lambda_t) \right \} \mid \Fcal_{t-1} \right ] \leq 1. \]
For succinctness, denote 
\[Y_{t} := X_{t} - \mu ~~~\text{ and }~~~ \delta_t := \widehat \mu_t - \mu. \]
Note that $\EE_P(Y_t \mid \Fcal_{t-1}) = 0$.
It then suffices to prove that for any $[0, 1)$-bounded, $\Fcal_{t-1}$- measurable $\lambda_{t} \equiv \lambda_{t}(X_1^{t-1})$, 
\[ \EE \Bigg [\exp \Bigg \{ \lambda_{t} Y_{t} - 4(Y_{t} - \delta_{t-1})^2\psi_{E}(\lambda_{t}) \Bigg \}  \Bigm | \Fcal_{t-1} \Bigg] \leq 1.\]
Indeed, in the proof of Proposition 4.1 in \citet{fan_exponential_2015}, $\exp\{\xi\lambda  - 4\xi^2 \psi_{E} (\lambda)\} \leq 1 + \xi \lambda$ for any $\lambda \in [0, 1)$ and $\xi \geq -1$. Setting $\xi := Y_{t} - \delta_{t-1} = X_{t} - \widehat \mu_{t-1}$,
\begin{align*}
	&\EE \Bigg [\exp \Bigg \{ \lambda_{t} Y_{t} - 4(Y_{t} - \delta_{t-1})^2 \psi_{E}(\lambda_{t} ) \Bigg \}  \Bigm | \Fcal_{t-1} \Bigg]\\
	&=\EE \Big [\exp \Big \{ \lambda_{t} (Y_{t}-\delta_{t-1}) - 4(Y_{t} - \delta_{t-1})^2 \psi_{E}(\lambda_{t}) \Big \}  \bigm | \Fcal_{t-1} \Big] \exp(\lambda_{t} \delta_{t-1})\\
	    &\leq \EE\Big [1 + (Y_{t} - \delta_{t-1} )\lambda_{t} \mid \Fcal_{t-1} \Big ]\exp(\lambda_{t} \delta_{t-1}) ~\stackrel{(i)}{=} \EE\big [1 - \delta_{t-1} \lambda_{t} \mid \Fcal_{t-1} \big ]\exp(\lambda_{t} \delta_{t-1}) ~\stackrel{(ii)}{\leq} 1,
\end{align*}
where equality $(i)$ follows from the fact that $Y_{t}$ is conditionally mean zero,
and inequality $(ii)$ follows from the inequality $1-x \leq \exp(-x)$ for all $x \in \mathbb R$. This completes the proof. \qed


\subsection{Proof of Proposition~\ref{proposition:betting}}
\label{proof:betting}
We proceed by proving $(d) \implies (c) \implies (b) \implies (a) \implies (d)$.
\\

\noindent \textbf{Proof of $(d)\implies(c)$.}
This claim follows from the fact that for $\lambda \in [-1/(1-\mu), 1/\mu]$, we have that $(\lambda, \lambda, \dots)$ is a $[-1/(1-\mu), 1/\mu]$-valued predictable sequence.

\smallskip
\noindent \textbf{Proof of $(c)\implies(b)$.}
By the assumption of $(c)$, we have that for $\lambda = 0.5$, $\Kcal_t(\mu)$ forms a test martingale. Furthermore, since $X_i, \mu \in [0, 1]$ for each $i \in \{1, 2, \dots \}$, we have that $1 + 0.5 (X_i - \mu) > 0$ almost surely for each $i$. Therefore, $(\Kcal_t(\mu))_{t=1}^\infty$ is a strictly positive test martingale.

\smallskip
\noindent \textbf{Proof of $(b)\implies(a)$.}
Suppose that there exists $\lambda \in \RR\setminus\{0\}$ such that $\Kcal_t(\mu)$ forms a strictly positive martingale. Then we must have 
\begin{align*}
    \Kcal_{t-1}(\mu) &= \EE \left ( \Kcal_t(\mu) \mid \Fcal_{t-1} \right) \\
    &= \Kcal_{t-1}(\mu) \cdot \EE \left ( 1 + \lambda (X_i - \mu) \mid \Fcal_{t-1} \right) \\
    &= \Kcal_{t-1}(\mu) \cdot \left [1 + \lambda(\EE(X_t \mid \Fcal_{t-1}) - \mu) \right]. 
\end{align*} 
Now since $\Kcal_{t-1}(\mu) > 0$, we have that
\[  1 + \lambda (\EE(X_t \mid \Fcal_{t-1}) - \mu) = 1. \]
Since $\lambda \neq 0$ by assumption, we have that $\EE (X_t \mid \Fcal_{t-1}) = \mu$ as required.

\smallskip
\noindent \textbf{Proof of $(a)\implies(d)$.}
Let $(\lambda_t(\mu))_{t=1}^\infty$ be a $(-1/(1-\mu), 1/\mu)$-valued predictable sequence. Then $\Kcal_t(\mu)$ is clearly nonnegative and $\Kcal_0(\mu)=1$ by definition. Writing out the conditional mean of the capital process for any $t \geq 1$, 
\begin{align*}
\EE \left ( \Kcal_t(\mu) \mid \Fcal_{t-1} \right ) &= \Kcal_{t-1}(\mu) \cdot \EE\left ( 1 + \lambda_t(\mu) (X_i - m) \mid \Fcal_{t-1} \right ) \\
&= \Kcal_{t-1}(\mu) \cdot \left [ 1 + \lambda_t(\mu) (\EE(X_i\mid \Fcal_{t-1}) - \mu) \right] \\
&= \Kcal_{t-1}(\mu),
\end{align*}
and thus $\Kcal_t(\mu)$ forms a test martingale.
\smallskip

The proof of the final part of the proposition is simple. Let $(M_t)$ be a test martingale for $\Pcal^\mu$. Define $Y_t := M_t/M_{t-1}$ if $M_{t-1} > 0$, and as $Y_t :=0$ otherwise. Now note that $M_t = \prod_{i=1}^t Y_t$ and $\EE_P[Y_t | \Fcal_{t-1}] = 1$ for any $P \in \Pcal^\mu$. In other words, every test martingale is a product of nonnegative random variables with conditional mean one. Now rewrite $Y_t$ as $(1+f_t(X_t))$ for some predictable function $f_t$. Since $Y_t$ is nonnegative, we must have $f_t(X_t) \geq -1$ , and since $Y_t$ is conditional mean one, we must have $f_t(X_t)$ is conditional mean zero. Such a representation in fact holds true for any test martingale, and we have not yet used the fact that we are working with test martingales for $\Pcal^\mu$. Now, the proof ends by noting that the only predictable functions $f_t$ with the latter property under every $P \in \Pcal^\mu$ has the form $\lambda_t(X_t - \mu)$ for some predictable $\lambda_t$; any nonlinear function of $X_t$ would not have mean zero under \emph{every distribution} with mean $\mu$. 

This completes the proof of Proposition~\ref{proposition:betting} altogether.
\qed

\subsection{Proof of Proposition~\ref{prop:universal}}\label{proof:universal}

We only prove the martingale part of the proposition, since the supermartingale aspect follows analogously, and as mentioned early in the paper, inequalities and equalities are meant in an almost sure sense.

First, it is easy to check that if $(M_t)$ is a test martingale for  $\Scal$, then $M_t$ is the product of nonnegative conditionally unit mean terms, that is $M_t = \prod_{i=1}^t Y_i$ such that for all $S \in \Scal$, we have $\EE_S[Y_i|\Fcal_{i-1}]=1$ and $Y_i\geq 0$. (Indeed, one can identify $Y_i:=\tfrac{M_i}{M_{i-1}}\mathbf{1}_{M_{i-1} > 0}$.) Now, define $Z'_i := Y_i - 1$, and note that $Z'_i \geq -1$, and $\EE_S[Z'_i|\Fcal_{t-1}]=0$. 
Thus, $M_t$ has been represented as $\prod_{i=1}^t (1+Z'_i)$.
Now, the proof is completed by noting that any such $Z'_i$ can be written as $\lambda_i Z_i$ for a predictable $\lambda_i$ (this step is purely cosmetic). \qed

\subsection{Proof of Theorem~\ref{theorem:hedgedCS}}
\label{proof:hedgedCS}
First, we present Lemma~\ref{lemma:convexityOfHedged} which establishes that the hedged capital process is a quasiconvex function of $m$ (and thus has convex sublevel sets). We then invoke this lemma to prove the main result.
\begin{lemma}
\label{lemma:convexityOfHedged}
    Let $\theta \in [0, 1]$ and 
    \begin{align*}
        \Kcal_t^\pm(m) &:= \max \left \{ \theta \Kcal_t^+(m), (1-\theta)\Kcal_t^-(m) \right \} \\ 
        &\equiv \max \left \{ \theta \prod_{i=1}^t (1 + \lambda_i^+(m) \cdot (X_i - m)), (1-\theta) \prod_{i=1}^t (1 - \lambda_i^-(m) \cdot (X_i - m)) \right \}
    \end{align*}
    be the hedged capital process as in Section~\ref{section:capitalProcess}. Consider the $(1-\alpha)$ confidence set of the same theorem,
    \[ \BI^\pm_t \equiv \BI^\pm(X_1, \dots, X_t) := \left \{ m \in [0, 1] : \Kcal_t^\pm(m) < \frac{1}{\alpha} \right \}. \]
    Then $\BI^\pm_t$ is an interval on $[0, 1]$.
\end{lemma}

\begin{proof}
Since sublevel sets of quasiconvex functions are convex, it suffices to prove that $\Kcal_t^\pm(m)$ is a quasiconvex function of $m \in [0, 1]$. The crux of the argument is: the product of nonnegative nonincreasing functions is quasiconvex, the product of nonnegative nondecreasing functions is also quasiconvex, and the maximum of quasiconvex functions is quasiconvex.

To elaborate, we will proceed in two steps. First, we use an induction argument to show that $\Kcal_t^+(m)$ and $\Kcal_t^-(m)$ are nonincreasing and nondecreasing, respectively, and hence quasiconvex. Finally, we note that $\Kcal_t^\pm(m) := \max\left \{\theta\Kcal_t^+(m), (1-\theta)\Kcal_t^-(m) \right \}$ is a maximum of quasiconvex functions and is thus itself quasiconvex.

\paragraph{Step 1.} First, since $\dot \lambda_t^+$ does not depend on $m$, we have that
\[ 1 + \lambda^+_t(m)(X_t - m) := 1 + \left ( |\dot \lambda_t^+| \land \frac{c}{m}\right )(X_t - m) \]
is nonnegative and nonincreasing in $m$ for each $t \in \{1, 2, \dots\}$. (To see this, consider the terms with and without truncation separately.) Suppose for the sake of induction that
\[ \prod_{i=1}^{t-1}\left (1 + \lambda_i^+(m)(X_i - m)\right )\]
is nonnegative and nonincreasing in $m$. Then,
\begin{align*}
    \Kcal_t^+(m) &:= \prod_{i=1}^{t}\left (1 + \lambda_i^+(m)(X_i - m)\right )\\
    &= \left ( 1 + \lambda_t^+(m)(X_t - m) \right ) \cdot \prod_{i=1}^{t-1}\left (1 + \lambda_i^+(m)(X_i - m)\right ) 
\end{align*}
is a product of nonnegative and nonincreasing functions, and is thus itself nonnegative and nonincreasing. By a similar argument, $\Kcal_t^-(m)$ is nonnegative and \textit{nondecreasing}. $\Kcal_t^+(m)$ and $\Kcal_t^-(m)$ are thus both quasiconvex.

\paragraph{Step 2.} Since the maximum of quasiconvex functions is quasiconvex, we infer that 
\[ \Kcal_t^\pm(m) := \max \left \{ \theta\Kcal_t^+(m), (1-\theta)\Kcal_t^-(m) \right \} \]
is quasiconvex. In particular, the sublevel sets of quasiconvex functions is convex, and thus
\[ \BI_t^\pm := \left \{ m \in [0, 1] : \Kcal_t^\pm(m) < \frac{1}{\alpha} \right \} \]
is an interval, which completes the proof of Lemma~\ref{lemma:convexityOfHedged}. \myqed
\end{proof}


\begin{proof}[\myproofname{Theorem~\ref{theorem:hedgedCS}}]
The proof proceeds in three steps. First we show that $\Kcal_t^\pm(\mu)$ is upper-bounded by test martingale. Second, we apply the 4-step procedure in Theorem~\ref{theorem:4step} to get a CS for $\mu$. Third and finally, we invoke Lemma~\ref{lemma:convexityOfHedged} to conclude that the CS is indeed convex at each time $t$.

\smallskip 
\paragraph{Step 1.} We first upper bound $\Kcal_t^\pm(m)$ as follows:
\begin{align*}
\Kcal_t^\pm(m) &:= \max \left \{ \theta\Kcal_t^+(m), (1-\theta)\Kcal_t^-(m) \right \} \\
&\leq \theta\Kcal_t^+(m) + (1-\theta) \Kcal_t^-(m) =: \Mcal_t^\pm(m).
\end{align*}
By Proposition~\ref{proposition:betting}, we have that $\Kcal_t^+(\mu)$ and $\Kcal_t^-(\mu)$ are test martingales for $\Pcal$.
For each $P\in\Pcal$, writing out the conditional expectation of $\Mcal_t^\pm(\mu)$ for any $t\geq1$,
\begin{align*}
\EE_P\left [ \Mcal_t^\pm(\mu) \mid \Fcal_{t-1} \right ] &= \EE_P\left [\theta \Kcal_t^+(\mu) + (1-\theta) \Kcal_t^-(\mu) \Bigm | \Fcal_{t-1} \right ] \\
&= \theta \EE_P(\Kcal_t^+(\mu) \mid \Fcal_{t-1}) + (1-\theta) \EE_P(\Kcal_t^-(\mu) \mid \Fcal_{t-1})\\
&= \theta\Kcal_{t-1}^+(\mu) + (1-\theta)\Kcal_{t-1}^-(\mu)\\
&= \Mcal_{t-1}^\pm(\mu),
\end{align*}
and $\Mcal_0^\pm(\mu) = \theta\Kcal_0^+(\mu) + (1-\theta)\Kcal_0^-(\mu) = 1$.
Therefore, $(\Mcal_t^\pm(\mu))_{t=0}^\infty$ is a test martingale for $\Pcal$.

\paragraph{Step 2.} By Step 1 combined with Theorem~\ref{theorem:4step} we have that
\[ \BI_t^\pm := \left \{ m \in [0, 1] : \Kcal_t^\pm(m) < \frac{1}{\alpha} \right \} \]
forms a $(1-\alpha)$-CS for $\mu$.

\paragraph{Step 3.} Finally, by Lemma~\ref{lemma:convexityOfHedged}, we have that $\BI_t^\pm$ is an interval for each $t \in \{1, 2, \dots\}$, which completes the proof of Theorem~\ref{theorem:hedgedCS}.
\myqed
\end{proof}

\subsection{Proof of Lemma~\ref{lemma:extendedFan}}
\label{proof:extendedFan}
Following the proof of Lemma~4.1 in \citet{fan_exponential_2015}, we have that the function
\begin{equation}
\label{eq:incFun}
    f(x) :=
    \begin{dcases}
        \frac{\log(1 + x) - x}{x^2 /2} & x \in (-1, \infty)\setminus\{ 0 \}\\
        -1 & x = 0
    \end{dcases}
\end{equation}
is an increasing and continuous function in $x$ (note that $f(0)$ is defined as $-1$ because it is a removable singularity). For any $y \geq -m$ and $\lambda \in [0, 1/m)$ we have 
\begin{equation}
\label{eq:lambdaIneq}
    \lambda y \geq -m\lambda > -1 .
\end{equation} 
Combining \eqref{eq:incFun} and \eqref{eq:lambdaIneq}, we have
\begin{align*}
    & \frac{\log(1+\lambda y) - \lambda y}{\lambda^2 y^2 / 2} \geq \frac{\log(1-m\lambda) + m\lambda }{\lambda^2m^2 / 2}, \\
    \text{ and thus, } ~ ~ & \log (1 + \lambda y) - \lambda y \overset{(i)}{\geq} \frac{y^2}{m^2} \left (\log ( 1- m\lambda) + m\lambda \right ).
\end{align*}
Above, $(i)$ can be quickly verified for the case when $\lambda y = 0$, and follows from \eqref{eq:incFun} and \eqref{eq:lambdaIneq} otherwise. Rearranging terms, we obtain the first half of the desired result,
\begin{equation}
\label{eq:FanFirstHalfProof}
\log ( 1 + \lambda y) \geq \lambda y + \frac{y^2}{m^2}(\log (1-m\lambda) + m\lambda). 
\end{equation}
Now, for any $y \leq 1-m$ and $\lambda \in (-1/(1-m), 0]$, we have 
\[
    \label{eq:lambdaIneqNegative}
    \lambda y \geq (1-m)\lambda > -1,
\]
and proceed similarly to before to obtain
\[ \log (1+\lambda y) \geq \lambda y + \frac{y^2}{(1-m)^2} (\log (1+(1-m)\lambda) - (1-m)\lambda),\]
which completes the proof.
\qed

\subsection{Proof of Proposition~\ref{proposition:hedgedKellyInterval}}
\label{proof:hedgedKellyInterval}

Since sublevel sets of convex functions are convex, it suffices to prove that with probability one, $\Kcal_n^\hgKelly(m)$ is a convex function in $m$ on the interval $[0, 1]$.

We proceed in three steps. First, we show that if two functions are (a) both nonincreasing (or both nondecreasing), (b) nonnegative, and (c) convex, then their product is convex. Second, we use Step~1 and an induction argument to prove that $\prod_{i=1}^t (1 + \gamma(X_i/m - 1))$ is convex for any fixed $\gamma \in [0, 1]$. Third and finally, we show that $\Kcal_n^\hgKelly(m)$ is a convex combination of convex functions and is thus itself convex.

\paragraph{Step 1.} The claim is that if two functions $f$ and $g$ are (a) both nonincreasing (or both nondecreasing), (b) nonnegative, and (c) convex on a set $\mathcal S \subseteq \mathbb R$, then their product is also convex on $\mathcal S$. Let $x_1, x_2 \in \mathcal S$, and let $t \in [0, 1]$. Furthermore, abbreviate $f(x_1)$ by $f_1$, $g(x_1)$ by $g_1$, and similarly for $f_2$ and $g_2$. Writing out the product $fg$ evaluated at $tx_1 + (1-t)x_2$,
\begin{align*}
    (fg) (tx_1 + (1-t)x_2) &= f(tx_1 + (1-t)x_2)g(tx_1 + (1-t)x_2) \\
    &= |f(tx_1 + (1-t)x_2)||g(tx_1 + (1-t)x_2)| \\
    &\leq | t f_2 + (1-t)f_2 | | t g_1 + (1-t)g_2 |\\
    &= t^2f_1g_1 + t(1-t)\left ( f_1g_2 + f_2g_1 \right ) + (1-t)^2f_2g_2,
\end{align*}
where the second equality follows from assumption that $f$ and $g$ are nonnegative, and the inequality follows from the assumption that they are both convex. To show convexity of $(fg)$, it then suffices to show that, 
\begin{equation}
\label{eqn:needToSatisfyConvex}
\Big ( tf_1g_1 + (1-t) f_2g_2 \Big ) - \Big (t^2f_1g_1 + t(1-t)\left [ f_1g_2 + f_2g_1 \right ] + (1-t)^2f_2g_2 \Big)\geq 0.
\end{equation}
To this end, write out the above expression and group terms,
\begin{align*}
    &\Big ( tf_1g_1 + (1-t) f_2g_2 \Big ) - \Big (t^2f_1g_1 + t(1-t)\left [ f_1g_2 + f_2g_1 \right ] + (1-t)^2f_2g_2 \Big) \\
    &= (1-t) tf_1g_1 + t(1-t)f_2g_2 - t(1-t)[f_1g_2 + f_2g_1] \\
    &= t(1-t) \Big ( f_1g_1 +f_2g_2 - f_1g_2 - f_2g_1 \Big ) \\
    &= t(1-t)(f_1 - f_2)(g_1 - g_2).
\end{align*}
Now, notice that $t(1-t) \geq 0$ since $t \in [0, 1]$ and that $(f_1 - f_2)(g_1 - g_2) \geq 0$ by the assumption that $f$ and $g$ are both nonincreasing or nondecreasing. Therefore, we have satisfied the inequality in \eqref{eqn:needToSatisfyConvex}, and thus $fg$ is convex on $\mathcal S$.

\paragraph{Step 2.} Now, we prove convexity of $\prod_{i=1}^t (1 + \gamma(X_i/m - 1))$ for a fixed $\gamma \in [0, 1]$. First note that for any $\gamma \in [0, 1]$, $1 + \gamma(X_i / m - 1)$ is a nonincreasing, nonnegative, and convex function in $m \in [0, 1]$. Suppose for the sake of induction that conditions (a), (b), and (c) hold for $\prod_{i=1}^{n-1}(1+\gamma(X_i/m -1))$. By the inductive hypothesis, we have that
\[ \prod_{i=1}^n (1+\gamma(X_i/m - 1)) = (1+\gamma(X_n/m - 1)) \cdot \prod_{i=1}^{n-1} (1+\gamma(X_i/m - 1)) \]
is a product of functions satisfying (a) through (c). By Step 1, $\prod_{i=1}^n(1+\gamma(X_i/m - 1))$ is convex in $m \in [0, 1]$. A similar argument can be made for $\Kcal_n^-(m)$, but instead of the multiplicands being nonincreasing, they are now nondecreasing.

\paragraph{Step 3.} Now, notice that for the evenly-spaced points $(\lambda^{1+}, \dots, \lambda^{G+})$ on $[0, 1/m]$, we have that $(\gamma^{1+}, \dots, \gamma^{G+}) = (m \lambda^{1+}, \dots, m\lambda^{G+})$ are $G$ evenly-spaced points on $[0, 1]$. It then follows that for any $m$ and any $g \in \{0, 1, \dots, G\}$,
\[ m \mapsto \prod_{i=1}^n(1 + \lambda^{g+}(X_i - m)) \]
is a nonincreasing, nonnegative, and convex function in $m \in [0, 1]$. It follows that
\[ \frac{1}{G} \sum_{g=1}^G \prod_{i=1}^n (1 + \lambda^{g+}(X_i - m)) \]
is convex in $m \in [0, 1]$. A similar argument goes through for $\frac{1}{G} \sum_{g=1}^G \prod_{i=1}^n (1 + \lambda^{g+}(X_i - m)$. Finally, since $\theta \in [0, 1]$, we have that 
\[ \frac{\theta }{G} \sum_{g=1}^G \prod_{i=1}^n (1 + \lambda^{g+} (X_i - m) ) + \frac{1-\theta }{G} \sum_{g=1}^G \prod_{i=1}^n (1 + \lambda^{g-} (X_i - m))\]
is a convex combination of convex functions in $m \in [0, 1]$. It then follows that 
\[ \{ m\in [0, 1] : \Kcal_t^\hgKelly(m) < 1/\alpha \} \]
is an interval, which completes the proof.
\qed

\subsection{Proof of Proposition~\ref{proposition:bettingWoR}}
\textbf{Proof of (1)$\implies$(2).}
By definition of $\Kcal^\WoR_t(\mu)$, we have
\begin{align*}
    \EE\left ( \Kcal_t^\WoR(\mu) \mid \Fcal_{t-1} \right ) &= \prod_{i=1}^{t-1} \left (1 + \lambda_i(\mu) \cdot (X_i - \mu_t^\WoR)\right ) \cdot \EE\left (1 + \lambda_t(\mu) \cdot (X_t - \mu_t^\WoR ) \mid \Fcal_{t-1} \right )\\
    &= \Kcal_{t-1}^\WoR(\mu) \cdot \left ( 1 + \lambda_t(\mu)\cdot (\EE(X_t\mid \Fcal_{t-1}) - \mu_t^\WoR ) \right )\\
    &= \Kcal_{t-1}^\WoR(\mu).
\end{align*}
Since $\Kcal_0^\WoR(\mu) \equiv 1$ by convention, we have that $\Kcal_t^\WoR(\mu)$ is a martingale. 

Now, note that since $X_t \in [0, 1]$ and $\lambda_t^\WoR(\mu) \in [-1/(1-\mu_t^\WoR), 1/\mu_t^\WoR]$ for each $t$ by assumption, we have that 
\( 1 + \lambda_t(\mu) \cdot \left ( X_t - \mu_t^\WoR \right ) \geq 0 \)
and thus $\Kcal_t^\WoR(\mu) \geq 0$. Therefore, $\Kcal_t^\WoR(\mu)$ is a test martingale.

\smallskip
\noindent \textbf{Proof of (2)$\implies$(1).}
Suppose that $\Kcal_t^\WoR(\mu)$ is a test martingale for any $(\lambda_t(\mu))_{t=1}^N$ with $\lambda_t(\mu) \in [-1/(1-\mu_t^\WoR, 1/\mu_t^\WoR]$, but suppose for the sake of contradiction that $\EE(X_{t^\star} \mid \Fcal_{t^\star-1}) \neq \mu_{t^\star}^\WoR$ for some $t^\star \in \{1, 2, \dots\}$. Set $\lambda_1 = \lambda_2 = \cdots = \lambda_{t^\star-1} = 0$ and $\lambda_{t^\star} = 1$. Then,
\[ \Kcal_{t^\star}^\WoR(\mu) \equiv \Kcal_{t^\star-1}^\WoR(\mu) \cdot (1 + \lambda_{t^\star}(X_{t^\star} - \mu_{t^\star}^\WoR)) = 1 + X_{t^\star} - \mu_{t^\star}^\WoR. \]
By assumption of $\Kcal_t^\WoR(\mu)$ forming a martingale, we have that $\EE\left ( \Kcal_{t^\star}^\WoR(\mu) \mid \Fcal_{t^\star-1} \right ) = \Kcal_{t^\star-1}^\WoR(\mu) = 1$. On the other hand, since $\EE \left (X_{t^\star} \mid \Fcal_{t^\star-1} \right) \neq \mu_{t^\star}^\WoR$, we have
\[ \EE\left ( \Kcal_{t^\star}^\WoR(\mu) \mid \Fcal_{t^\star-1} \right ) = \EE\left (1 + X_{t^\star} - \mu_{t^\star}^\WoR \mid \Fcal_{t^\star - 1}\right ) \neq 1, \]
a contradiction. Therefore, we must have that $\EE\left ( X_t \mid \Fcal_{t-1} \right ) = \mu_t^\WoR$ for each $t$, which completes the proof of (2) $\implies$ (1) and Proposition~\ref{proposition:bettingWoR}.
\qed

\subsection{Proof of Theorem~\ref{theorem:hedgedCSWoR}}
\label{proof:hedgedCSWoR}
The proof that $\BI_t^{\pm, \WoR}$ forms a $(1-\alpha)$-CS for $\mu$ proceeds in exactly the same manner as Theorem~\ref{theorem:hedgedCS}, noting that $\EE\left (X_t \mid \Fcal_{t-1} \right ) = \mu_t^\WoR$ instead of $\mu$.

To show that $\BI_t^{\pm, \WoR}$ is indeed an interval for each $t \geq 1$, we note that the proof of Theorem~\ref{theorem:hedgedCS} applies since $m_t^\WoR$ is increasing or decreasing if and only if $m$ is increasing or decreasing, respectively. \qed


\section{How to bet: deriving adaptive betting strategies}
\label{section:howToBet}

In Section~\ref{section:hedgedCapitalProcess}, we presented CSs and CIs via the hedged capital process. We suggested a specific betting scheme which has strong empirical performance but did not discuss where it came from. In this section, we derive various betting strategies and discuss their statistical and computational properties.

\subsection{Predictable plug-ins yield good betting strategies}
\label{section:predmix_goodBetting}
First and foremost, we will examine why any predictable plug-in for empirical Bernstein-type CSs and CIs (i.e. those recommended in Theorem~\ref{theorem:EBCS} and Remark~\ref{remark:EBCI}) yield effective betting strategies. Consider the hedged capital process
\begin{align*}
\Kcal_t^\pm(m) &:= \max \left \{ \theta \prod_{i=1}^t (1 + \lambda_i^+(X_i - m)), (1-\theta)\prod_{i=1}^t (1 - \lambda_i^-(X_i - m)) \right \} \\
&\equiv \max \left \{ \theta \Kcal_t^+(m), (1-\theta)\Kcal_t^-(m) \right \} ~,
\end{align*}
where $(\lambda_t^+(m))_{t=1}^\infty$ and $(\lambda_t^-(m))_{t=1}^\infty$ are $[0, 1/m]$-valued and $[0, 1/(1-m)]$-valued predictable sequences as in Theorem~\ref{theorem:hedgedCS}. First, consider the ``positive'' capital process, $\Kcal_t^+(\mu)$ evaluated at $m = \mu$. An inequality that has been repeatedly used to derive empirical Bernstein inequalities \citep{howard_exponential_2018, howard_uniform_2019, waudby2020confidence}, including the current paper is the following due to 
\citet[equation 4.12]{fan_exponential_2015}: for any $y \geq -1$ and $\lambda \in [0,1)$, we have
\begin{equation}
\label{eq:fan}
\log(1+\lambda y) \geq \lambda y - 4\psi_E(\lambda) y^2.
\end{equation}
where $\psi_E(\lambda)$ is as defined in \eqref{eq:psi-E}.
If the predictable sequence $(\lambda_t^+(m))_{t=1}^\infty$ is further restricted to $[0, 1)$, then by \eqref{eq:fan} we have 
\begin{align*}
    \Kcal_t^+(\mu) &:= \prod_{i=1}^t (1 + \lambda_i^+(X_i - \mu)) ~\geq \exp \left ( \sum_{i=1}^t \lambda_i^+(X_i - \mu) - \sum_{i=1}^t 4 (X_i - \mu)^2 \psi_E(\lambda_i^+) \right )\\
    &\overset{(i)}\approx \exp \left ( \sum_{i=1}^t \lambda_i^+(X_i - \mu) - \sum_{i=1}^t4(X_i - \widehat \mu_{i-1})^2 \psi_E(\lambda_i^+)  \right )\\
    &= M_t^\PMEB(\mu),
\end{align*}
where $(i)$ follows from the approximations $\widehat \mu_{t-1} \approx \mu$ for large $t$. Not only does the approximate inequality $\Kcal^+_t(\mu) \gtrsim M^\PMEB_t(\mu)$  shed light on why a sensible empirical Bernstein predictable plug-in translates to a sensible betting strategy, but also why we might expect $\Kcal_t^+(m)$ to be more powerful than $M^\PMEB_t(m)$ for the same $[0, 1)$-valued predictable sequence $(\lambda_t^+(m))_{t=1}^\infty$. Moreover, $\Kcal_t^+(m)$ has the added flexibility of allowing $(\lambda_t(m))_{t=1}^\infty$ to take values in $[0, 1/m] \supset [0, 1)$ which we find --- through simulations --- tends to improves empirical performance (see Figure~\ref{fig:hedgedCI-tweaks} in Section~\ref{section:root-n-convergence-hedgedCI}). Finally, a similar story holds for $\Kcal_t^-(\mu)$ with the added caveat that $(\lambda^-_t)_{t=1}^\infty$ can instead take values in $[0, 1/(1-m)] \supset [0, 1)$ which as before, seems to improve empirical performance.

Despite the success of predictable plug-ins as betting strategies, it is natural to wonder whether it is preferable to focus on directly maximizing capital over time. As will be seen in the following section, these capital-maximizing approaches tend to have improved empirical performance, but are not always guaranteed to produce convex confidence sets (i.e. intervals). Nevertheless, it is worth examining some of these strategies both for their intuitive appeal and excellent empirical performance.


\subsection{Growth rate adaptive to the particular alternative (GRAPA)}
As alluded to in Section~\ref{section:history-short}, \citet{kelly1956new} dealt with capital processes, betting strategies, etc. in the fields of information and communication theory in the pursuit of maximizing the information rate over a channel. Kelly suggested that an effective betting strategy is one that maximizes a gambler's expected \textit{log-capital} --- i.e. the growth rate of the gambler's capital --- under a particular alternative.\footnote{This objective has also been arrived at indirectly as the dual in optimization programs for deriving regret bounds for Kullback-Leibler-based UCB algorithms in multi-armed bandit problems \citep{honda2010asymptotically,cappe2013kullback}.} However, Kelly's setup was a simplified special case of ours: Kelly's observations were binary, and the exact alternative was assumed known, while ours are merely bounded in $[0, 1]$ with an unknown alternative. Nevertheless, the principle of maximizing the log-capital can be adapted to our setting under bounded observations and an unknown alternative. We summarize this adaptation here and refer to it as maximizing the ``growth rate adaptive to the particular alternative'' or ``GRAPA'' for short.

Write the log-capital process at time $t$ as
\begin{equation}\label{eq:log-wealth}
\ell_t(\lambda_1^t, m) := \log(\Kcal_t(m)) = \sum_{i=1}^t \log (1+\lambda_i(m)(X_i-m)), 
\end{equation} 
for a general $[-1/(1-m), 1/m]$-valued sequence $(\lambda_t(m))_{t=1}^\infty$. If we were to choose a single value of $\lambda^\HS := \lambda_1 = \cdots = \lambda_t$  which maximizes the log-capital $\ell_t$ ``in hindsight'' (i.e. based on \textit{all} of the previous data), then this value is given by
\[ \frac{\partial \ell_t(\lambda^\HS, m)}{\partial\lambda^\HS} = \sum_{i=1}^t \frac{X_i - m}{1 + \lambda^\HS(X_i - m)} \overset{\text{set}}{=} 0.\]
However, $\lambda^\mathrm{HS}$ is clearly not predictable. Following \citet{kumon2011sequential} (who referred to this as the ``sequential optimization strategy''), we set $(\lambda^\SOS_t(m))_{t=1}^\infty$ such that 
\begin{equation} 
\label{eqn:sos}
\frac{1}{t-1}\sum_{i=1}^{t-1} \frac{X_i - m}{1 + \lambda^\SOS_t(m)(X_i - m)} \overset{\text{set}}{=} 0,
\end{equation}
truncated to lie between $(-c/(1-m), c/m)$ using some $c \leq 1$. Importantly, $\lambda^\SOS_t(m)$ only depends on $X_1, \dots, X_{t-1}$, and is thus predictable. 

This rule is a sequentially adaptive version of the worst-case ``GROW'' criterion of~\citet{grunwald_safe_2019}. To see the connection, one can derive~\eqref{eqn:sos} from a slightly different motivation. At the $t$-th step, we want to choose $\lambda_t(m)$ so that the wealth multiplier $(1+\lambda_t(m)(X_t-m))$ is as large as possible. The ideal choice would be
\begin{equation}
\label{eq:GRAPA}
\lambda^*_t(m) := \argmax_{\lambda \in [-1/(1-m), 1/m]} \EE_{P^\mu}[\log(1+\lambda(X_t-m)) \mid \Fcal_{t-1}],
\end{equation}
where $P^\mu$ is the unknown true distribution. Writing down the stationary condition for this optimization problem by differentiating through the expectation, we get
\begin{equation} 
\label{eqn:kkt}
\EE_{P^\mu} \left[ \frac{X_t - m}{1 + \lambda^*_t(m)(X_t - m)} \mid \Fcal_{t-1} \right] = 0.
\end{equation}
Since $P^\mu$ is unknown, using a simple empirical plug-in estimator yields~\eqref{eqn:sos}.

CSs constructed from $(\lambda^\SOS_t(m))_{t=1}^\infty$ tend to have excellent empirical performance, but can be prohibitively slow due to the required root-finding in \eqref{eqn:sos} for each time $t$ and $m\in [0, 1]$ (or a sufficiently fine grid of $[0, 1]$).  A similar but computationally inexpensive alternative to GRAPA is ``approximate GRAPA'' (aGRAPA), which we derive now.  

\subsection{Approximate GRAPA (aGRAPA)}
Rather than solve \eqref{eqn:sos}, we take the Taylor approximation of $(1+y)^{-1}$ by $(1-y)$ for $y \approx 0$ to obtain
\begin{align*}
    \frac{1}{t-1}\sum_{i=1}^{t-1} \frac{X_i - m}{1 + \lambda^\aSOS_t(m)(X_i - m)} &\approx \frac{1}{t-1}\sum_{i=1}^{t-1} \left (1 - \lambda^\aSOS_t(m)(X_i - m) \right)(X_i - m) \\
    &= \frac{1}{t-1}\sum_{i=1}^{t-1}(X_i - m) - \frac{\lambda^\aSOS_t(m)}{t-1} \sum_{i=1}^{t-1} (X_i - m)^2 \\
    &\overset{\text{set}}{=} 0,
\end{align*}
which, after appropriate truncation leads what we call the ``approximate GRAPA'' (aGRAPA) betting strategy,
\begin{align*}
    \lambda_t^\aSOS (m) := -\frac{c}{1-m} \lor \frac{\widehat \mu_{t-1} - m}{\widehat \sigma_{t-1}^2 + (\widehat \mu_{t-1} - m)^2} \land \frac{c}{m},
\end{align*}
for some truncation level $c \leq 1$. 
This expression is quite natural: we bet more aggressively if our empirical mean is far away from $m$, and are further emboldened if the empirical variance is small.

\begin{figure}[!htbp]
    \centering
    \includegraphics[width=\textwidth]{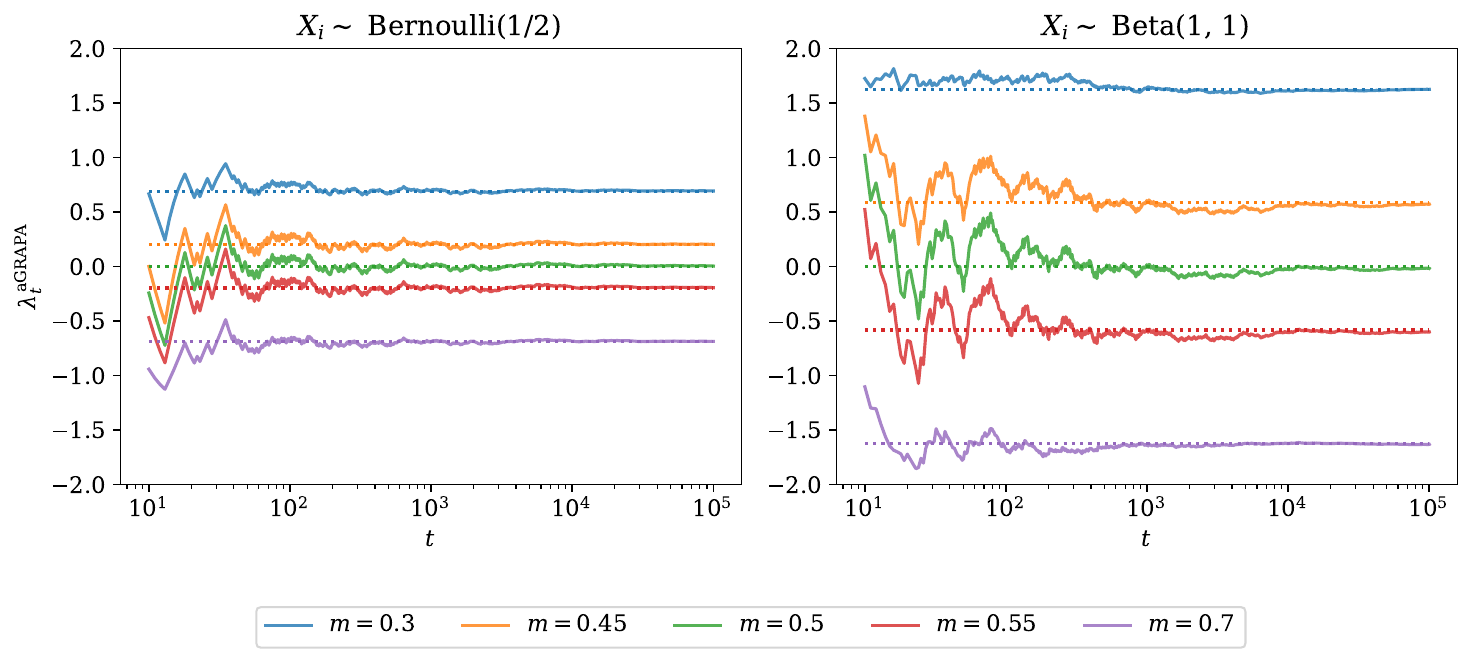}
    \caption{$\lambda_t^\aSOS$ for various values of $m$ under two distributions: Bernoulli(1/2) and Beta(1, 1). The dotted lines show the ``oracle'' bets, meaning $\lambda^\aSOS_t$ with estimates of the mean and variance replaced by their true values. As time passes, bets stabilize and approach their oracle quantities.}
    \label{fig:lambdaGROW}
\end{figure}

As alluded to at the end of Section~\ref{section:predmix_goodBetting}, CSs derived using the capital process $\Kcal_t(m)$ with arbitrary betting schemes are not always guaranteed to produce a convex set (interval). In fact, it is possible to construct scenarios where the sublevel sets of $\Kcal_t^\aSOS(m)$ are nonconvex in $m$ (see Section~\ref{section:aSOSsublevelSets} for an example). In our experience, this type of situation is not common, and one must actively search for such pathological examples.

\subsection{Lower-bound on the wealth (LBOW)}
Instead of maximizing $\log (\Kcal_t(m))$, we may aim to do so for a tight lower-bound on the wealth (LBOW). This technique has proven useful in the game-theoretic probability literature \citep[Proof of Lemma 3.3]{shafer_probability_2005} and \citep[Proof of Theorem 1]{cutkosky2018black}. Our lower bound will rely on an extension of Fan's inequality \eqref{eq:fan} to $\lambda \in (-1/(1-m), 1/m)$, summarized in the following lemma.
\begin{lemma}
\label{lemma:extendedFan}
    If $y \geq -m$, then for any $\lambda \in [0, 1/m)$, we have
    \[ \log (1 + \lambda y) \geq \lambda y + \frac{y^2}{m^2}(\log (1-m\lambda) + m\lambda). \]
    On the other hand, if $y \leq 1-m$, then for any $\lambda \in (-1/(1-m), 0]$, we have
    \[ \log (1 + \lambda y) \geq \lambda y + \frac{y^2}{(1-m)^2}(\log (1+(1-m)\lambda) - (1-m)\lambda). \]
Thus, for $y \in [-m,1-m]$, both of the above inequalities hold.
\end{lemma}
The proof is an easy generalization of inequality~\eqref{eq:fan} by~\citet{fan_exponential_2015}, and also follows from similar observations about the subexponential function $\psi_E$ in~\citet{howard_exponential_2018,howard_uniform_2019}, but we prove it from first principles in Section~\ref{proof:extendedFan} for completeness. Using Lemma~\ref{lemma:extendedFan}, we have for $\lambda^{\LBOW+} \in [0, 1/m)$, the following lower-bound on $\ell(\lambda^{\LBOW+}, m)$,
\begin{align}
     \ell(\lambda^{\LBOW+}, m) &:= \log \left ( \prod_{i=1}^t (1 + \lambda^{\LBOW+} (X_i - m)) \right ) \nonumber \\
     &\geq \lambda^{\LBOW+}\sum_{i=1}^t (X_i - m) + \frac{\log(1-m\lambda^{\LBOW +}) + m\lambda^{\LBOW +} }{m^2}\sum_{i=1}^t (X_i - m)^2, \label{eq:LBOW+}
\end{align}
and for $\lambda^{\LBOW-} \in (-1/(1-m), 0]$, we have
\begin{align}
     \ell(\lambda^{\LBOW-}, m) &:= \log \left ( \prod_{i=1}^t (1 + \lambda^{\LBOW-} (X_i - m)) \right ) \nonumber \\
     \geq\ &\lambda^{\LBOW-}\sum_{i=1}^t (X_i - m) + \frac{\log(1+(1-m)\lambda^{\LBOW -}) - (1-m)\lambda^{\LBOW -} }{(1-m)^2}\sum_{i=1}^t (X_i - m)^2 \label{eq:LBOW-}.
\end{align}
Importantly, if $\sum_{i=1}^t(X_i - m)$ is positive, then \eqref{eq:LBOW+}  is concave, while if negative, \eqref{eq:LBOW-} is concave. Maximizing \eqref{eq:LBOW+} or \eqref{eq:LBOW-} depending on the sign of $\sum_{i=1}^t (X_i - m)$ we obtain the following ``hindsight'' choice for $\lambda^\LBOW$,
\begin{align*}
    \lambda^{\LBOW} &=
    \begin{dcases}
        \frac{\sum_{i=1}^t (X_i - m) }{m \sum_{i=1}^t (X_i - m) + \sum_{i=1}^t (X_i - m)^2 } & \text{if $\sum_{i=1}^t (X_i - m ) \geq 0$}, \\
        \frac{\sum_{i=1}^t (X_i - m) }{-(1-m) \sum_{i=1}^t (X_i - m) + \sum_{i=1}^t (X_i - m)^2 } & \text{if $\sum_{i=1}^t (X_i - m ) \leq 0$}. 
    \end{dcases} 
\end{align*}
Of course, this choice of $\lambda^\LBOW$ is not predictable and thus is not a valid betting strategy in the framework of the current paper. This motivates the following strategy, $(\lambda_t^\LBOW(m))_{t=1}^\infty$ given by
\begin{equation}
\label{eq:LBOW}
    \lambda^\LBOW_t(m) := \frac{-c}{1-m} \lor \frac{\widehat \mu_{t-1} - m}{\omega_{t-1}| \widehat \mu_{t-1} - m| + \widehat \sigma_{t-1}^2 + (\widehat \mu_{t-1} - m)^2 } \land \frac{c}{m},
\end{equation}
\[ \text{where}~~~\omega_t :=
\begin{cases}
    m & \text{if $\widehat \mu_t - m \geq 0$} ~, \\
    1-m & \text{if $\widehat \mu_t - m < 0$} ~.
\end{cases} \]
Similarly to the aGRAPA betting procedure, LBOW is computationally-inexpensive but is not guaranteed to produce an interval.
The expression also carries similar intuition to the GRAPA case.

\subsection{Online Newton Step (ONS-\texorpdfstring{$m$}{m}) }
Betting algorithms play an essential role in online learning as several optimization problems can be framed in terms of coin-betting games \citep{cutkosky2018black, orabona2017training, jun2017improved, jun2019parameter}. While the downstream application is different, the game-theoretic techniques of maximizing wealth are almost immediately applicable to the problem at hand. Here, we consider a slight modification to the Online Newton Step (ONS) algorithm due to \citet{cutkosky2018black}.\\

\begin{algorithm}[H]
\SetAlgoLined
\KwResult{$(\lambda_t^\ONS(m))_{t=1}^T$}
 $\lambda_1^\ONS(m) \gets 1$\;
 \For{$t \in \{1, \dots, T-1\}$}{
    $y_t \gets X_t - m$ \;
    Set $z_t \gets y_t / (1 - y_t\lambda^\ONS_{t}(m))$ \;
    $A_t \gets 1 + \sum_{i=1}^t z_i^2$ \;
    $\lambda_{t+1}^\ONS(m) \gets \frac{-c}{1-m} \lor \left ( \lambda^\ONS_{t}(m) - \frac{2}{2-\log(3)} \frac{z_t}{A_t}\right ) \land \frac{c}{m}$ \;
 }
\caption{ONS-$m$.}
\end{algorithm}
\vspace{0.05in}

Through simulations, we find that ONS-$m$ performs competitively. However, its lack of closed-form expression makes it a slightly more computationally-expensive alternative to aGRAPA and LBOW, but not nearly as expensive as GRAPA (see Table~\ref{tab:simulationTimes}). 

\begin{figure}
    \centering
    \includegraphics[width=\textwidth]{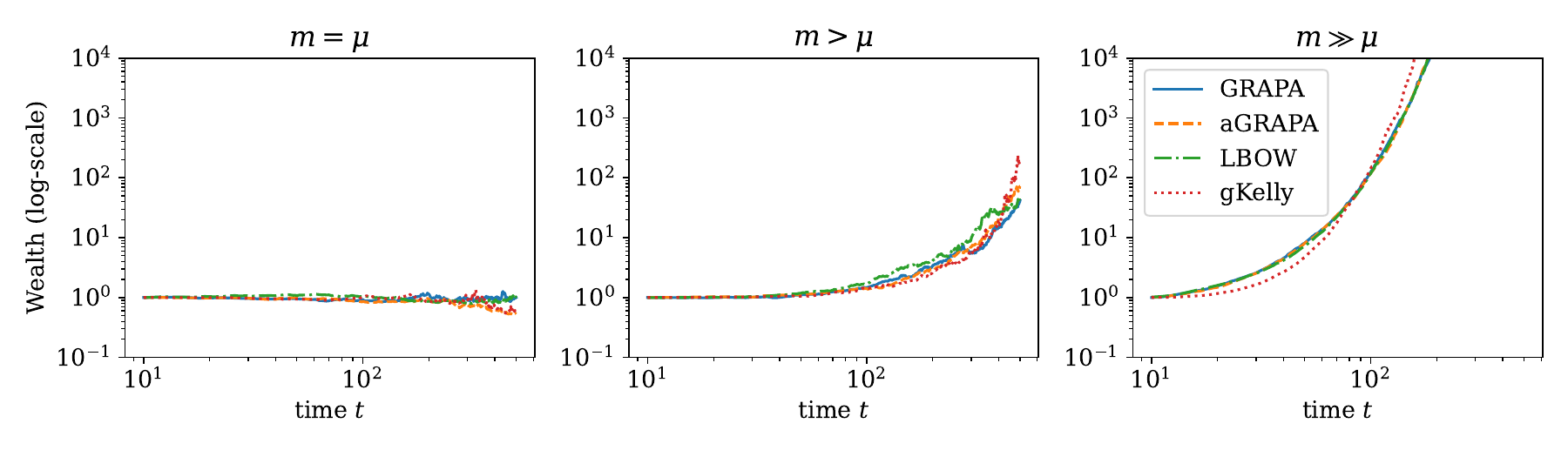}
    \caption{Comparison of the wealth process under various game-theoretic betting strategies with 100 repeats. In this example, the 1000 observations are drawn from a Beta(10, 10) distribution, and the candidate means $m$ being tested are 0.5, 0.51, and 0.55 (from left to right). Notice that these strategies perform similarly, but have varying computational costs (see Table~\ref{tab:simulationTimes}).}
    \label{fig:gameTheoreticBettingStrategiesCompared}
\end{figure}

\subsection{Diversified Kelly betting (dKelly)}
\label{section:dKelly}

Instead of committing to one betting strategy such as aGRAPA or LBOW, we can simply take the average capital among $D$ separate strategies. This follows from the fact that an average of test martingales is itself a test martingale. That is, if $(\lambda_t^1)_{t=1}^\infty, (\lambda_t^2)_{t=1}^\infty, \dots, (\lambda_t^D)_{t=1}^\infty$ are $D$ separate betting strategies, then
\[ \Kcal_t^\dKelly(\mu) := \frac{1}{D}\sum_{d=1}^D \prod_{i=1}^t \left ( 1 + \lambda_i^d(\mu) (X_i - \mu) \right ) \]
forms a test martingale. Following Kelly's original motivation to maximize (expected) log-capital, notice that by Jensen's inequality,
\[ \log \left ( \Kcal_t^\dKelly \right ) > \frac{1}{D} \sum_{d=1}^D \log \left ( \prod_{i=1}^t \left (1 + \lambda_i^d(\mu)(X_i - \mu) \right ) \right).
\]
In other words, the log-capital of the diversified bets is strictly larger than the average log-capital among the diverse candidate bets.

\paragraph{Grid Kelly betting (gKelly).}
While it is possible to use any finite collection of strategies, we focus our attention on a particularly simple (and useful) example where the bets are constant values on a grid. Specifically, divide the interval $[-1/(1-m), 1/m]$ up into $G$ evenly-spaced points $\lambda^1, \dots, \lambda^G$.
Then define the $\gKelly$ capital process $\Kcal_t^\gKelly$ by 
\[ \Kcal_t^\gKelly(m) := \frac{1}{G} \sum_{g=1}^G \prod_{i=1}^t \left ( 1 + \lambda^g(X_i - m) \right ). \]
When used to construct confidence sequences for $\mu$, $\Kcal_t^\gKelly$ demonstrates excellent empirical performance. Moreover, this procedure can be slightly modified into ``Hedged gKelly'' (hgKelly) so that confidence sequences constructed using gKelly are intervals almost surely. 

In order to mimic the unknown optimal $\lambda^*$, $D$ or $G$ should not be kept constant, but itself grow slowly (say logarithmically) with $t$. In game-theoretic terms, one should  slowly add more strategies to the portfolio, in order to asymptotically match the performance of the optimal one over time. (When adding a new $\lambda^g$ to an existing mixture, it obviously only begins to contribute to the wealth from the following step onwards; formally $G$ would be replaced by $G_t$, and $\prod_{i=1}^t (1 + \lambda^g(X_i - m)$ would be replaced by $\prod_{i=t_g}^t (1 + \lambda^g(X_i - m)$ if $\lambda^g$ was first introduced after $t_g-1$ steps.)

\paragraph{Hedged gKelly.}
First, divide the interval $[-1/(1-m), 0]$ and $[0, 1/m]$ into $G$ evenly-spaced points: $(\lambda^{1-}, \dots, \lambda^{G-})$ and $(\lambda^{1+}, \dots, \lambda^{G+})$, respectively. Then define the ``Hedged grid Kelly capital process'' $\Kcal_t^\hgKelly$ given by 
\[ \Kcal_t^\hgKelly(m) := \frac{\theta}{G} \sum_{g=1}^G \prod_{i=1}^t \left ( 1 + \lambda^{g+} (X_i - m) \right ) +  \frac{1-\theta}{G} \sum_{g=1}^G \prod_{i=1}^t \left ( 1 + \lambda^{g-} (X_i - m) \right ), \]
where $\theta \in [0, 1]$ (a reasonable default being $\theta = 1/2$).

\begin{proposition}
\label{proposition:hedgedKellyInterval}
    If $(X_t)_{t=1}^\infty \sim P$ for some $P \in \Pcal^\mu$, then $\Kcal_t^\hgKelly(\mu)$ forms a test martingale and
    \( \BI^\hgKelly_t := \left \{ m \in [0, 1] : \Kcal_t^\hgKelly(m) < 1/\alpha \right \} \) is a CS for $\mu$ that
    forms an interval for each $t \geq 1$.
\end{proposition}
The proof in Section~\ref{proof:hedgedKellyInterval} proceeds by showing that $\Kcal_t^\hgKelly$ is a convex function of $m$ and hence its sublevel sets are intervals.

\subsection{Confidence Boundary (ConBo)}

The aforementioned strategies benefit from targeting bets against a particular null hypothesis, $H_0^m$ for each $m \in [0, 1]$, but this has the drawback of $\Kcal_t(m)$ potentially not being quasiconvex in $m$. One of the advantages of the hedged capital process as described in Theorem~\ref{theorem:hedgedCS} is that $\Kcal_t^\pm(m)$ is always quasiconvex, and thus its sublevel sets (and hence the confidence sets $\BI_t^\pm$) are intervals. 

In an effort to develop game-theoretic betting strategies which generate confidence sets which are intervals, we present the Confidence Boundary (ConBo) bets. Rather than bet against the null hypotheses $H_0^m$ for each $m \in [0, 1]$, consider two sequences of nulls, $(H_0^{u_t})_{t=1}^\infty$ and $(H_0^{l_t})_{t=1}^\infty$ corresponding to upper and lower confidence boundaries, respectively. The ConBo bet $\lambda_{t}^\CB$ is then targeted against $u_{t-1}$ and $l_{t-1}$ using \textit{any} game-theoretic betting strategy (e.g. $\ast$GRAPA, $\ast$Kelly, LBOW, or ONS-$m$). Letting $\lambda^\G_t(m)$ be any such strategy, we summarize the ConBo betting scheme in Algorithm~\ref{algorithm:ConBo}.

\begin{corollary}[Confidence boundary CS \textbf{[ConBo]}]
\label{corollary:ConBo}
In Algorithm~\ref{algorithm:ConBo},
\[ \BI_t^\CB ~~~ \text{forms a $(1-\alpha)$-CS for $\mu$}, \]
as does $\bigcap_{i\leq t} \BI_i^\CB$. Further, $\BI_t^\CB$ is an interval for any $t \geq 1$.
\end{corollary}

\begin{algorithm}[h!]
\SetAlgoLined
\KwResult{$(\Kcal_t^\CB(m))_{t=1}^T$}
 $l_0 \gets 0\ ;\ u_0 \gets 1$\;
 $\Kcal_0^{\CB+}(m) \gets \Kcal_0^{\CB-}(m) \gets 1$\;
 \For{$t \in \{1, \dots, T\}$}{
    $\lambda_t^{\CB+} \gets \max \left \{\lambda^\G_t(l_{t-1}) , 0 \right \} \land c/m$\tcp*{Compute ConBo bets}
    $\lambda_t^{\CB-} \gets \left |\min \left \{ \lambda^\G_t(u_{t-1}), 0 \right \} \right | \land c/(1-m)$\; 
    $\Kcal_t^{\CB+}(m) \gets \left [1 + \lambda_t^{\CB+}(X_t - m) \right ] \cdot \Kcal^{\CB+}_{t-1}(m)$\tcp*{Update capital}
    $\Kcal_t^{\CB-}(m) \gets \left [1 - \lambda_t^{\CB-}(X_t - m) \right ] \cdot \Kcal^{\CB-}_{t-1}(m)$\; 
    $\Kcal_t^\CB(m) \gets \max \left \{ \theta \Kcal^{\CB+}_t(m), (1-\theta) \Kcal^{\CB-}_t(m) \right \}$\tcp*{Hedging}
    $\BI^\CB_t \gets \{ m\in [0, 1] : \Kcal_t(m) < 1/\alpha \} $ \ \; 
    $l_t \gets \inf \BI_t^\CB$\tcp*{Update confidence boundaries to bet against}
    $u_t \gets \sup \BI_t^\CB$\;
 }
\caption{ConBo}
\label{algorithm:ConBo}
\end{algorithm}
\vspace{0.05in}

We can also adapt the ConBo betting scheme outlined in Algorithm~\ref{algorithm:ConBo} to the without-replacement setting by replacing $m$ by $m_t^\WoR$ for each time $t$.

\begin{corollary}[WoR confidence boundary CS {\textbf{[ConBo-WoR]}}]
\label{corollary:ConBoWoR}
Under the same conditions as Theorem~\ref{theorem:hedgedCSWoR}, define $\lambda_t^{\CBWoR+}$ and $\lambda_t^{\CBWoR-}$ as in Algorithm~\ref{algorithm:ConBo} but with $m$ replaced by $m_t^\WoR$. Then,
\[ \BI_t^{\CBWoR} := \left \{ m \in [0, 1] : \Kcal_t^{\CBWoR} < 1/\alpha \right \} ~~~\text{forms a $(1-\alpha)$-CS for $\mu$,}\]
as does $\bigcap_{i\leq t} \BI_i^{\CBWoR}$. Further, $\BI_t^{\CBWoR}$ is an interval for each $t \geq 1$.
\end{corollary}

\subsection{Sequentially Rebalanced Portfolio (SRP)}

Implicitly, none of the aforementioned strategies take advantage of ``rebalancing'', meaning the ability to take ones capital $\Kcal_{t}$ at time $t$, diversify it in any manner at time $t+1$, and repeat. This has had the mathematical advantage of being able to write the resulting capital process $(\Kcal_{t}(m))_{t=1}^{\infty}$ in the following general, but closed-form expression:
\[ \Kcal_{t}(m) := \sum_{d=1}^{D}\theta_d \prod_{i=1}^{t} (1 + \lambda_{i}^{d}(m) \cdot (X_{i} - m)), \]
where $D \geq 1$ is as in Section~\ref{section:dKelly}, $(\lambda_{t}^{1}(m))_{t=1}^{\infty}, \dots, (\lambda_{t}^{D}(m))_{t=1}^{\infty}$ are $[-1/(1-m), 1/m]$-valued predictable sequences as usual, and $(\theta_d)_{d=1}^D$ are convex weights such that $\sum_{d=1}^D \theta_d = 1$.
However, a more general capital process martingale can be written but instead of having a closed-form product expression, it can be written recursively as
\begin{equation}
  \label{eq:SRP}
  \Kcal_t^\SRP(m) := \sum_{d=1}^{D_t} (1 + \lambda_t^d(m)\cdot (X_t - m)) \cdot \theta_t^d \cdot \Kcal_{t-1}^\SRP(m),
\end{equation}
where $(\lambda_t^d)_{d=1}^{D_t}$ are $[1/(1-m), 1/m]$-valued predictable bets, $(\theta_t^d)_{d=1}^{D_t}$ are predictable convex weights that sum to 1 (conditional on $X_1^{t-1}$), and we have set the initial capital $\Kcal_0^\SRP(m)$ to 1 as usual.

Adopting the betting interpretation, \eqref{eq:SRP} is a rather intuitive procedure. At each time step $t$, the gambler divides their previous capital $\Kcal_{t-1}^\SRP(m)$ up into $D_t \geq 1$ portions given by $\theta_t^1 \cdot K_{t-1}^\SRP(m), \dots, \theta_t^{D_t} \cdot \Kcal_{t-1}^\SRP(m)$, then invests these wealths with bets $\lambda_t^1(m), \dots, \lambda_t^{D_t}(m)$, respectively. The gambler's wealths are then updated to \[ (1+\lambda_t^1(m)\cdot (X_t - m)) \cdot \theta_t^1\cdot \Kcal_{t-1}^\SRP(m), \dots, (1+\lambda_t^{D_t}(m) \cdot (X_t - m)) \cdot \theta_t^{D_t}\cdot \Kcal_{t-1}^\SRP(m), \]
which are then combined via summation to yield a final capital of \eqref{eq:SRP}.

It is now routine to check that the process given by \eqref{eq:SRP} is a nonnegative martingale when evaluated at $\mu$ since
\begin{align*}
  \EE \left ( \Kcal_t^\SRP(\mu) \mid X_1^{t-1} \right ) &= \sum_{d=1}^{D_t} \Kcal_{t-1}^\SRP(\mu) \cdot \theta_t^{D_t} \cdot \left ( 1 + \lambda_t(\mu) \left ( \underbrace{\EE(X_t \mid X_1^{t-1}) - \mu}_{=0} \right ) \right ) \\
                                                        &= \Kcal_{t-1}^\SRP(\mu)\underbrace{\sum_{d=1}^{D_t} \theta_t^{D_t}}_{=1}
                                                        = \Kcal_{t-1}^\SRP(\mu).
\end{align*}
Note that SRP is the most general and customizable betting strategy presented in this paper, since it can be composed of any of the previously discussed strategies, and includes each of them as a special case.


\section{Simulations}
\label{section:simulations}
This section contains a comprehensive set of simulations comparing our new confidence sets presented  against previous works. We present simulations for building both time-uniform CSs and fixed-time CIs with or without replacement. Each of these are presented under four distributional ``themes'': (1) discrete, high-variance; (2) discrete, low-variance; (3) real-valued, evenly spread; and (4) real-valued, concentrated.

\clearpage
\subsection{Time-uniform confidence sequences (with replacement)}

\begin{figure}[h!]
 \centering
    \includegraphics[width=\textwidth]{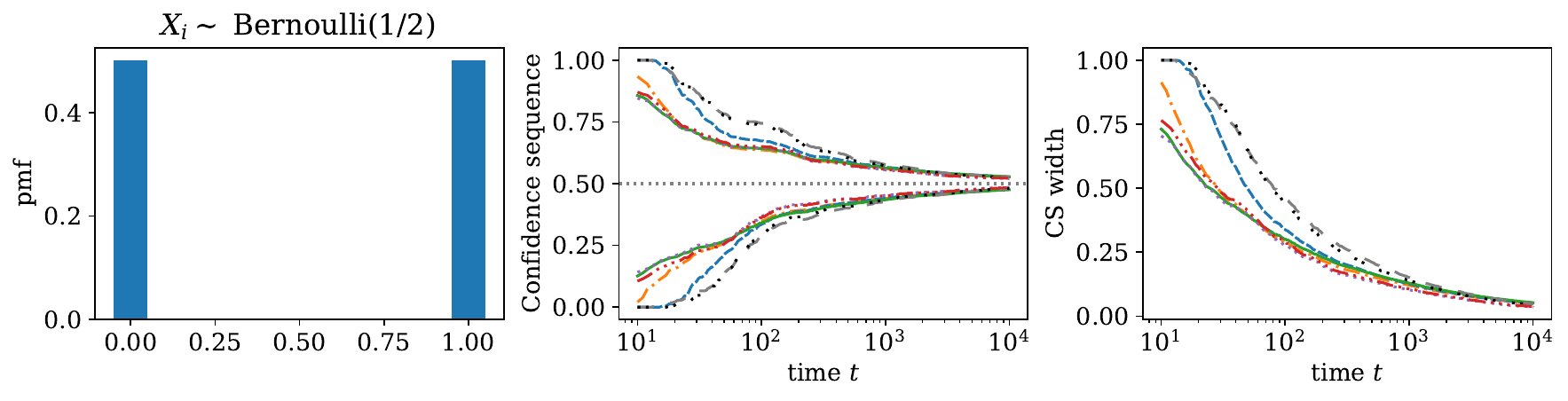}
    \centering
    \includegraphics[width=\textwidth]{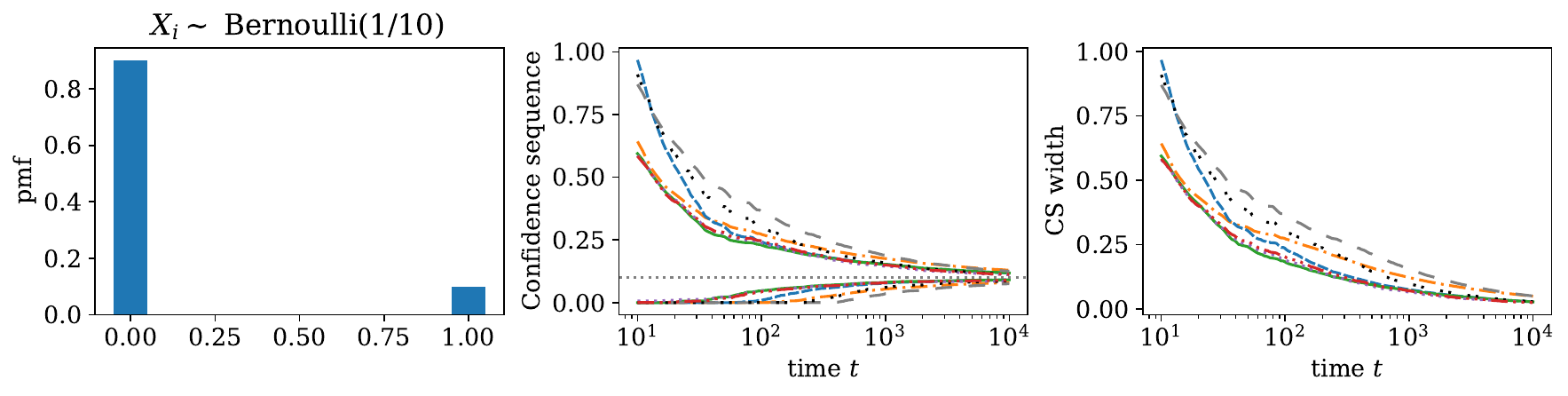}
    \includegraphics[width=\textwidth]{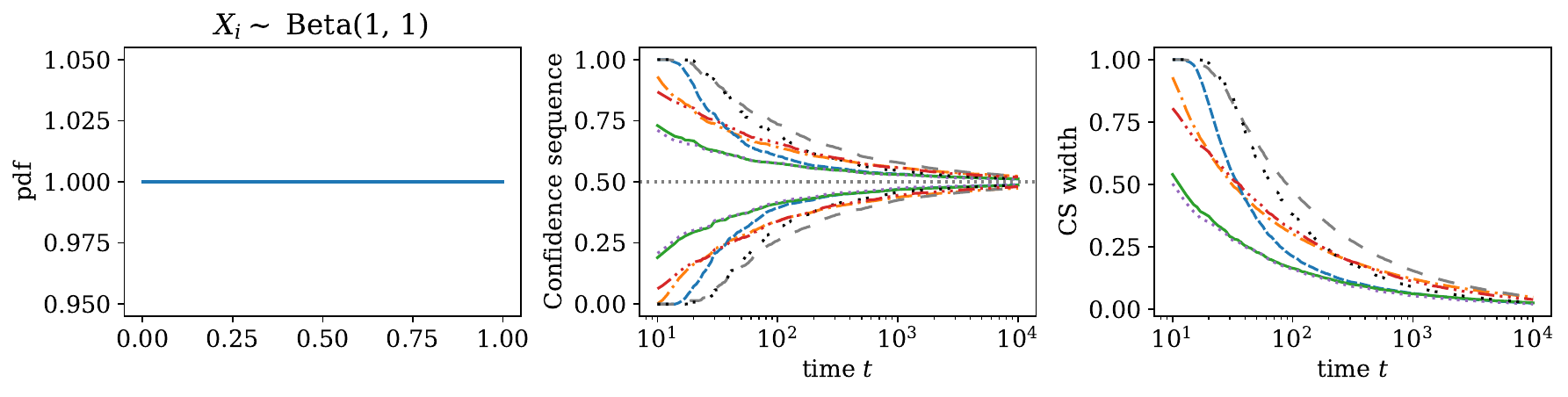}
    \centering
    \includegraphics[width=\textwidth]{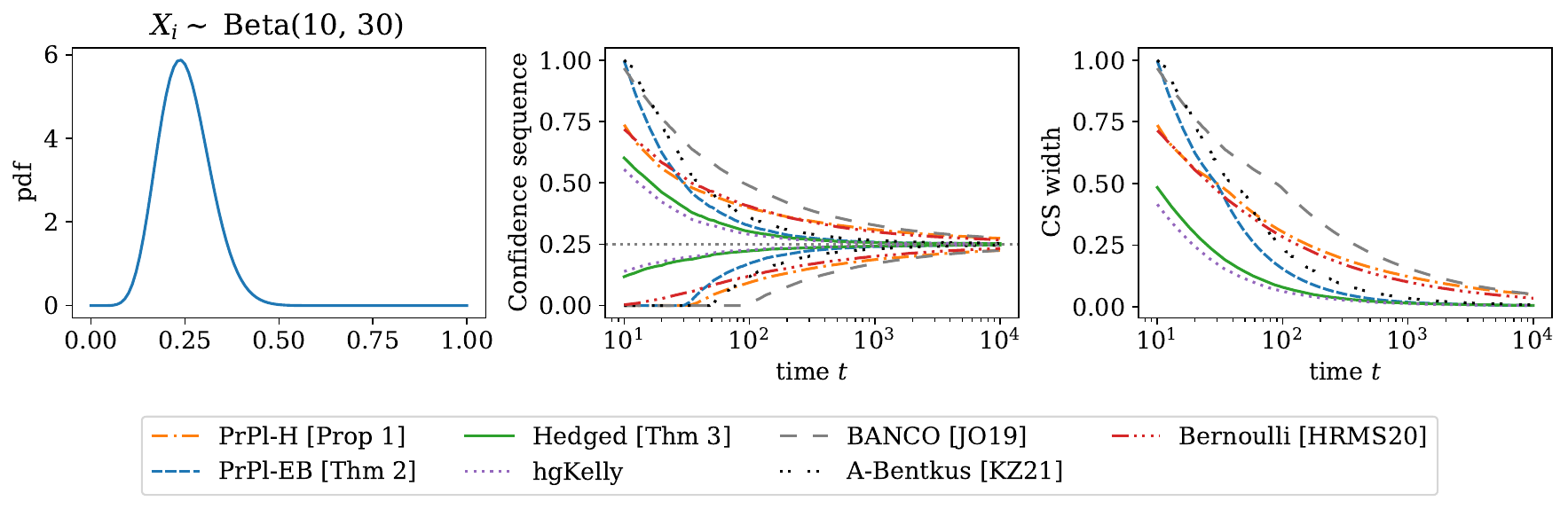}
 \caption{Comparing Hedged, hgKelly, PrPl-EB, and PrPl-H CSs alongside other time-uniform confidence sequences in the literature; further details in Section~\ref{section:simDetails-TUWR}. Clearly, the betting approach is dominant in all settings.}
 \label{fig:TUWR}
\end{figure}


\clearpage
\subsection{Fixed-time confidence intervals (with replacement)}

\begin{figure}[!htbp]
 \centering
    \includegraphics[width=\textwidth]{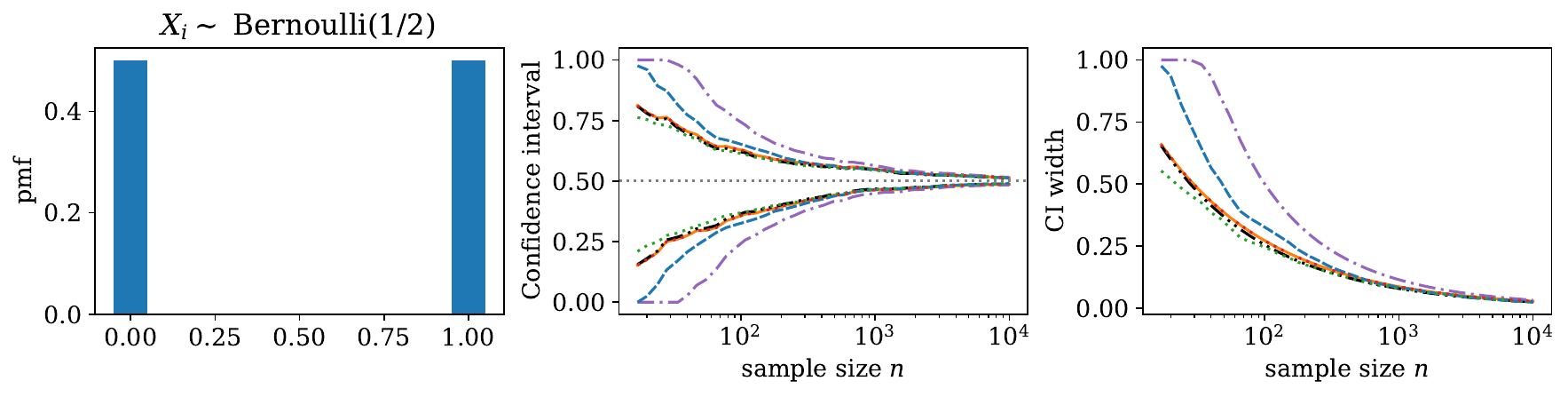}
 \hfill
    \includegraphics[width=\textwidth]{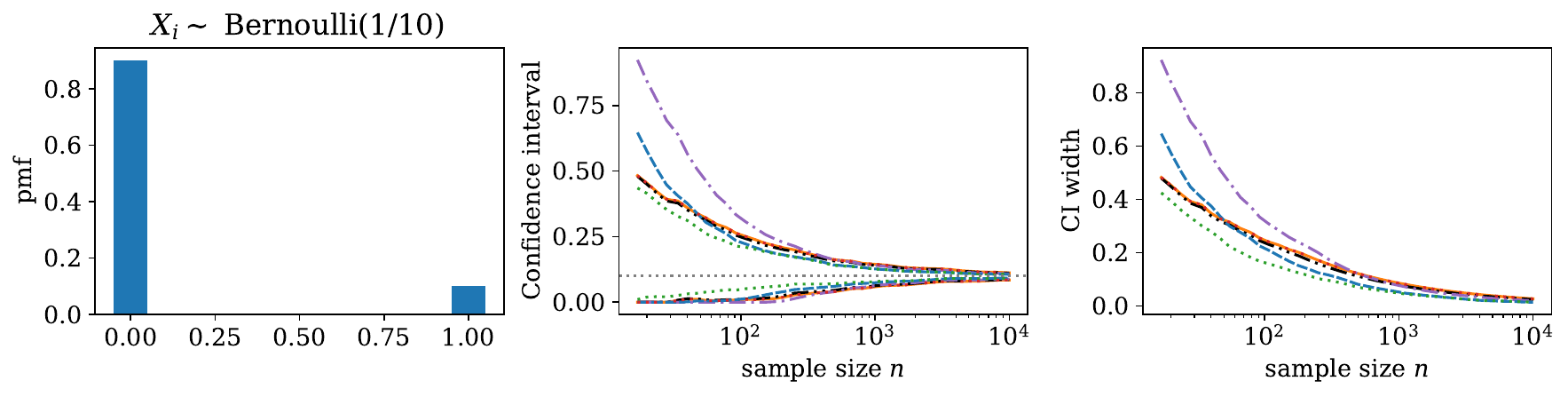}
    \includegraphics[width=\textwidth]{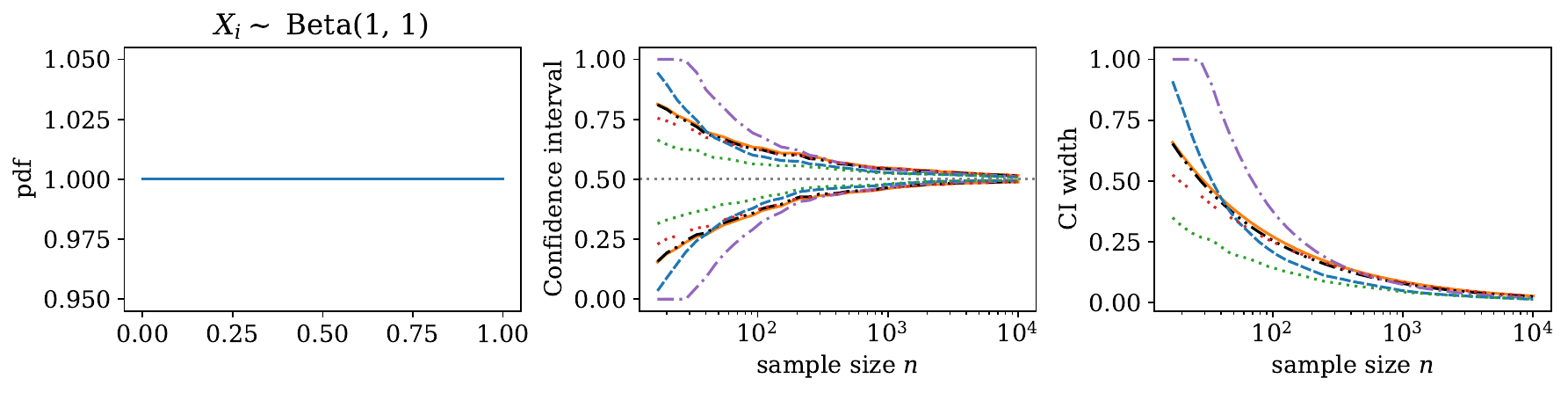}
 \hfill
    \includegraphics[width=\textwidth]{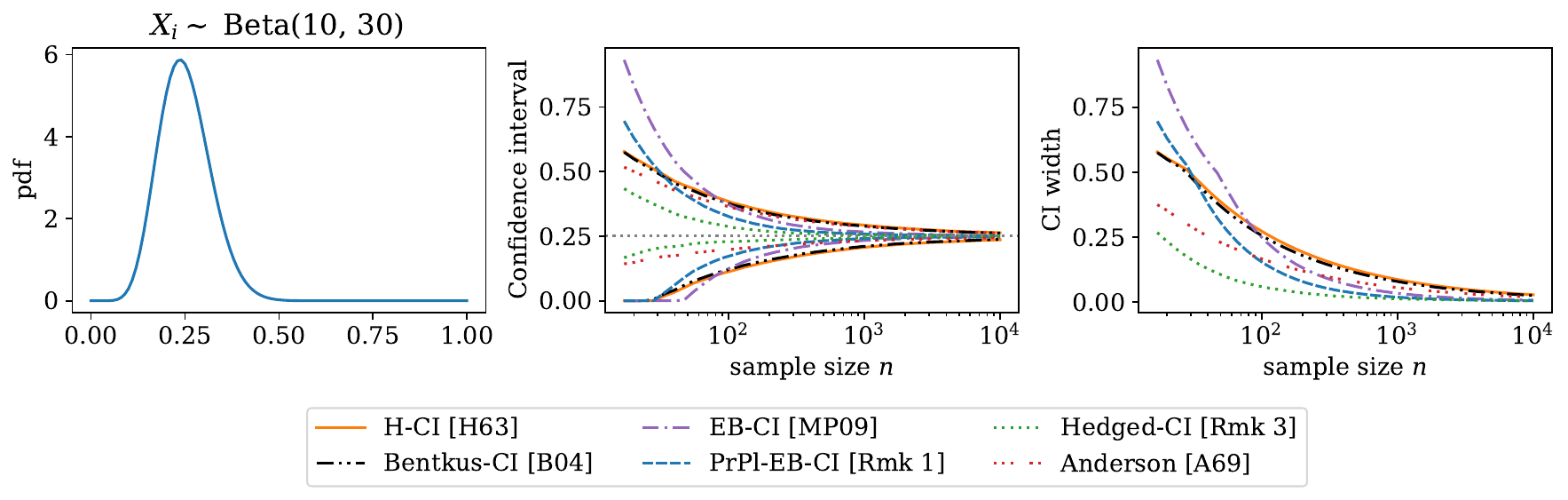}
 \caption{Hedged capital, Anderson, Bentkus, Maurer-Pontil empirical Bernstein, and predictable plug-in empirical Bernstein CIs under four distributional scenarios. Further details can be found in Section~\ref{section:simDetails-FTWR}. Clearly, the betting approach is dominant in all settings.}
 \label{fig:FTWR}
\end{figure}


\clearpage
\subsection{Time-uniform confidence sequences (without replacement)}

\begin{figure}[!htbp]

 \centering
    \includegraphics[width=\textwidth]{fig/WoR_Time-uniform/Bernoulli_0.5__WoR_time-uniform.pdf}
 \hfill
    \includegraphics[width=\textwidth]{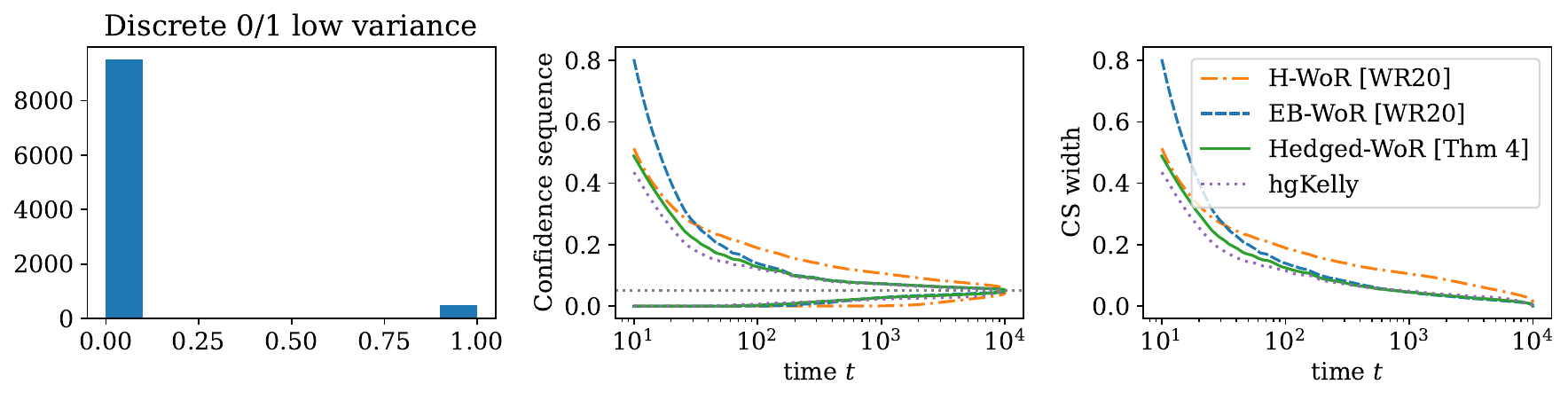}
    \centering
    \includegraphics[width=\textwidth]{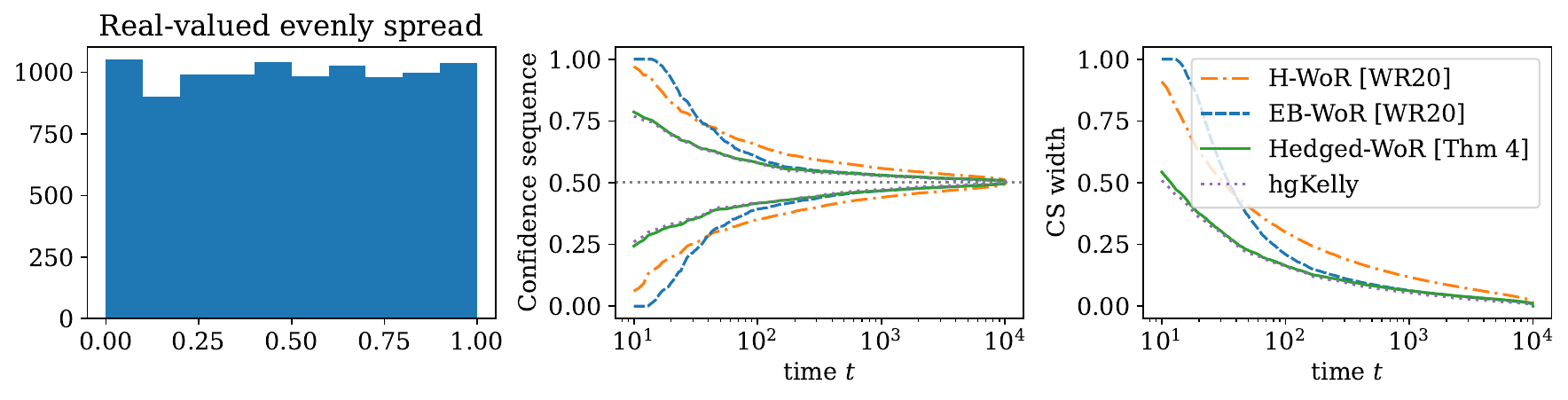}
 \hfill
    \includegraphics[width=\textwidth]{fig/WoR_Time-uniform/Beta_10,_30__WoR_time-uniform.pdf}
 \caption{Hedged capital, Hoeffding, and empirical Bernstein CSs for the mean of a finite set of bounded numbers when sampling WoR. Further details can be found in Section~\ref{section:simDetails-TUWoR}. Clearly, the betting approach is dominant in all settings.}
 \label{fig:TUWoR}
\end{figure}


\clearpage
\subsection{Fixed-time confidence intervals (without replacement)}

\begin{figure}[!htbp]
 \centering
    \includegraphics[width=\textwidth]{fig/WoR_Fixed-time/Bernoulli_0.5__WoR_fixed-time.pdf}
    \includegraphics[width=\textwidth]{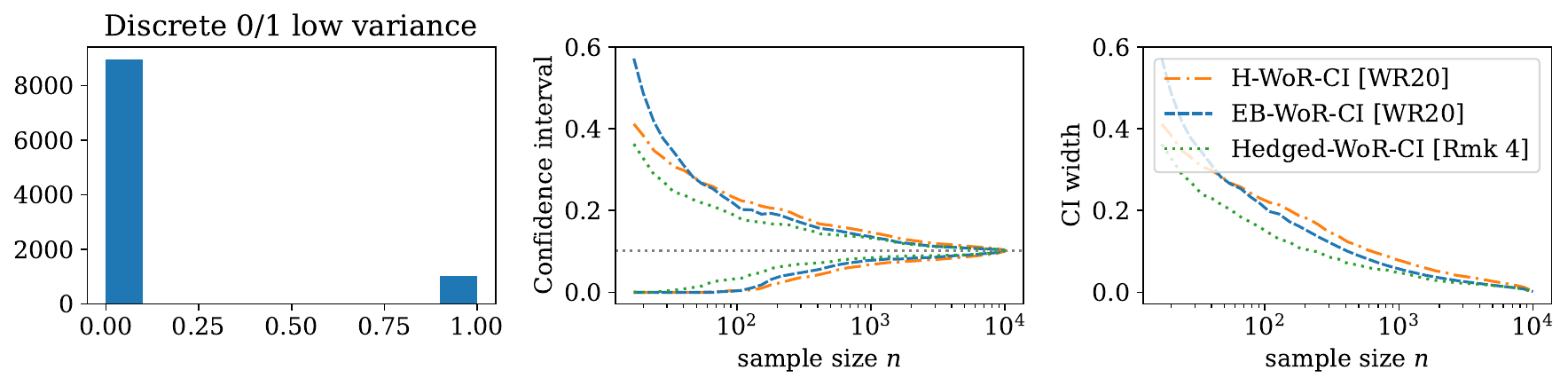}
    \includegraphics[width=\textwidth]{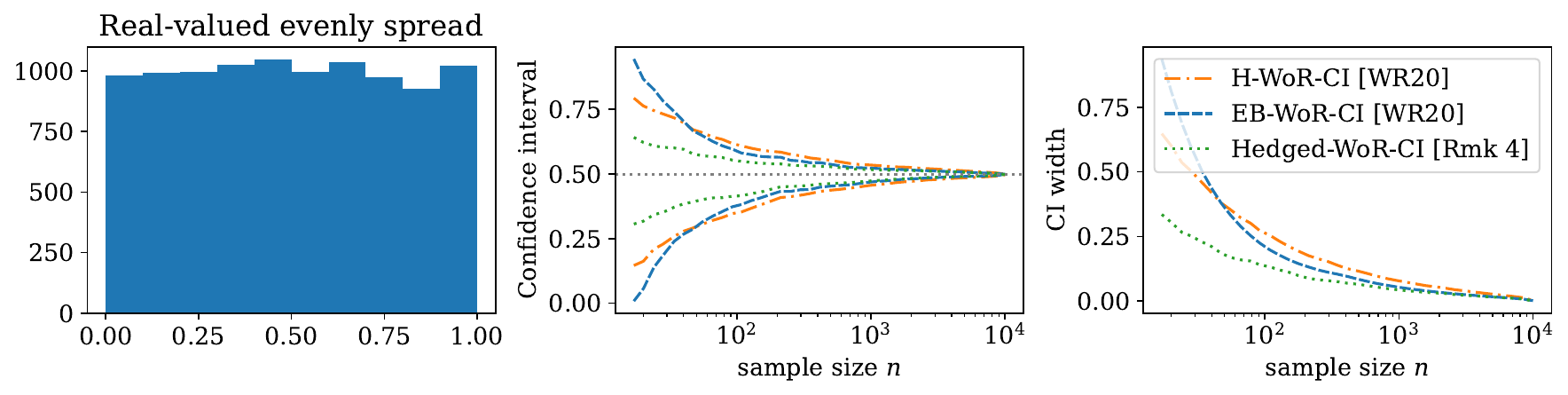}
    \includegraphics[width=\textwidth]{fig/WoR_Fixed-time/Beta_10,_30__WoR_fixed-time.pdf}
  \caption{Fixed-time hedged capital, Hoeffding-type, and empirical Bernstein-type CIs for the mean of a finite set of bounded numbers when sampling WoR. Further details can be found in Section~\ref{section:simDetails-FTWoR}. Clearly, the two betting approaches (Hedged and ConBo) are dominant in all settings.}
 \label{fig:FTWoR}
\end{figure}

\newpage
\section{Simulation details} 

In each simulation containing confidence sequences or intervals and their widths, we took an average over 5 random draws from the relevant distribution. For example, in the ``Time-uniform confidence sequences'' plot of Figure~\ref{fig:comparisonAll}, the CSs (PrPl-H, PrPl-EB, and Hedged) were averaged over 5 random draws from a Beta$(10, 30)$ distribution. Computation times for various strategies are given in Table~\ref{tab:simulationTimes}.

\begin{table}
\caption{\label{tab:simulationTimes}Typical computation time for constructing a CS from time $1$ to $10^3$ for the mean of Bernoulli$(1/2)$-distributed random variables. The three betting CSs were computed for 1000 evenly-spaced values of $m$ in $[0, 1]$, while a coarser grid would have sped up computation. All CSs were calculated on a laptop powered by a quad-core 2GHz 10th generation Intel Core i5. Parallelization was carried out using the Python library, \texttt{multiprocess} \citep{mckerns2012building}.}
\centering
    \fbox{%
    \begin{tabular}{*{4}{c}}
    \Centering Betting scheme & \Centering Interval (a.s.) & \Centering Computation time  (seconds) \\
    \hline
    ConBo+LBOW & \checkmark & 0.08 \\
    Hedged+$(\lambda_t^{\PMpm})_{t=1}^\infty$ & \checkmark & 0.25 \\
    hgKelly ($G=20$) & \checkmark & 1.38 \\
    aGRAPA & & 0.35 \\
    LBOW & & 0.25 \\
    ONS-$m$ & & 12.45 \\
    Kelly & & 197.38 \\
    \end{tabular}}
\end{table}

\label{sec:app-simulations}


\subsection{Time-uniform confidence sequences (with replacement)}
\label{section:simDetails-TUWR}
Each of the CSs considered in the time-uniform (with replacement) case are presented as explicit theorems and propositions throughout the paper. Specifically,
\begin{itemize}
    \item \textbf{PrPl-H}: Predictable plug-in Hoeffding (Proposition~\ref{proposition:predmixHoeffding});
    \item \textbf{PrPl-EB}: Predictable plug-in empirical Bernstein (Theorem~\ref{theorem:EBCS});
    \item \textbf{Hedged}: Hedged capital process (Theorem~\ref{theorem:hedgedCS}); and
    \item \textbf{hgKelly}: Hedged grid-Kelly (Proposition~\ref{proposition:hedgedKellyInterval}).
\end{itemize}

\paragraph{Bernoulli [HRMS20]}Section~\ref{section:simulations} compared these against the conjugate mixture sub-Bernoulli confidence sequence by \citet{howard_uniform_2019}, recalled below.

\citet[Equation (3.4)]{hoeffding_probability_1963}, presented the sub-Bernoulli upper-bound on the moment generating function of bounded random variables for any $\lambda > 0$:
\[ \EE_P\left ( \exp\left \{ \lambda (X_i - \mu) \right \} \right ) \leq 1 - \mu + \mu \exp\{\lambda\}, \]
which can be used to construct an $e$-value by noting that 
\[\EE_P \left ( \exp \left \{ \lambda (X_i - \mu) - \log(1 - \mu + \mu e^\lambda) \right \} \mid \Fcal_{i-1}\right ) \leq 1.\]
Then, \citet{howard_uniform_2019} showed that the cumulative product process
\begin{equation}
\label{eq:subBernoulliSuperMG}
\prod_{i=1}^t\left ( \exp \left \{ \lambda (X_i - \mu) - \log(1 - \mu + \mu e^\lambda) \right \} \right ) 
\end{equation}
forms a test supermartingale, as does a mixture of \eqref{eq:subBernoulliSuperMG} for any probability distribution $F(\lambda)$ on $\RR^+$:
\begin{equation}
\label{eq:subBernoulliMixtureSuperMG}
\int_{\lambda \in \RR^+} \prod_{i=1}^t\left ( \exp \left \{ \lambda X_i - \log(1 - \mu + \mu e^\lambda) \right \} \right ) dF(\lambda).
\end{equation}
In particular, \citet{howard_uniform_2019} take $F(\lambda)$ to be a beta distribution so that the integral \eqref{eq:subBernoulliMixtureSuperMG} can be computed in closed-form. Using \eqref{eq:subBernoulliMixtureSuperMG}  in Step (b) in Theorem~\ref{theorem:4step} yields the ``Bernoulli [HRMS20]'' confidence sequence.

There are yet other improvements of Hoeffding's inequality, for example one that goes by the name of Kearns-Saul \citep{kearns_large_1998} but was incidentally noted in Hoeffding's original paper itself. This inequality, and other variants, are looser than the sub-Bernoulli bound and so we exclude them here; see \citet{howard_exponential_2018} for more details. Most importantly, none of these adapt to the true underlying variance of the random variables, unlike most of our new techniques.

\paragraph{A-Bentkus [KZ21]} 

We also compared our bounds against the ``adaptive Bentkus confidence sequence'' (A-Bentkus) due to \citet[Section~3.5]{kuchibhotla2020near}. These combine a maximal version of \citeauthor{bentkus2006domination}'s concentration inequality \cite[Theorem~1]{kuchibhotla2020near} with the ``stitching'' technique \cite{zhao_adaptive_2016, mnih_empirical_2008, howard_uniform_2019} --- a method to obtain infinite-horizon concentration inequalities by taking a union bound over exponentially-spaced \emph{finite} time horizons. 

\subsection{Fixed-time confidence intervals (with replacement)}
\label{section:simDetails-FTWR}
For the fixed-time CIs included from this paper, we have 
\begin{itemize}
    \item \textbf{PrPl-EB-CI}: Predictable plug-in empirical Bernstein CI (Remark~\ref{remark:EBCI}); and
    \item \textbf{Hedged-CI}: Hedged capital process CI (Remark~\ref{remark:hedgedCI}).
\end{itemize}
These were compared against CIs due to \citet{hoeffding_probability_1963}, \citet{maurer_empirical_2009}, \citet{anderson1969confidence}, and \citet{bentkus2004hoeffding} which we now recall.

\paragraph{H-CI [H63]}
These intervals refer to the CIs based on Hoeffding's classical concentration inequalities \citep{hoeffding_probability_1963}. Specifically, for a sample size $n \geq 1$, ``H-CI [H63]'' refers to the CI, 
\[ \frac{1}{n} \sum_{i=1}^n X_i \pm \sqrt{\frac{\log(2/\alpha)}{2n}}. \]

\paragraph{Anderson [A69]}
These intervals refer to the confidence intervals due to \citet{anderson1969confidence} which take a unique approach by considering the entire sample cumulative distribution function, rather than just the mean and variance. Consequently, however, Anderson's CIs require iid observations, rather than the more general setup we consider. We nevertheless find that even in the iid setting, our approach outperforms Anderson's.

Suppose $X_1, \dots, X_n \overset{iid}{\sim} P$ are $[0, 1]$-bounded with mean $\EE_P (X_1) = \mu$. Let $X_{(1)}, \dots, X_{(n)}$ denote the order statistics of $X_1^n$ with the convention that $X_{(0)} := 0$ and $X_{(n+1)} := 1$. Following the notation of \citet{learned2019new}, Anderson's CI is given by 
\[ \left [ \sum_{i=1}^n u_i^\mathrm{DKW}\left (-X_{(n-(i+1))} + X_{(n-i)} \right ),\ 1 - \sum_{i=1}^n u_i^\mathrm{DKW}\left (  X_{(i+1)} - X_{(i)} \right ) \right ], \]
where $u_i^\mathrm{DKW} = \left ( i / n - \sqrt{\log (2/\alpha) / 2n} \right ) \lor 0$. \citet[Theorem 2]{learned2019new} show that Anderson's CI is always tighter than Hoeffding's. The authors also introduce a bound which is strictly tighter than Anderson's which they conjecture has valid $(1-\alpha)$-coverage, but we do not compare to this bound here.

\paragraph{EB-CI [MP09]}
The empirical Bernstein CI of \citet{maurer_empirical_2009} is given by
\[ \frac{1}{n} \sum_{i=1}^n X_i \pm \sqrt{\frac{2\widehat \sigma^2 \log(4/\alpha)}{n}} + \frac{7 \log(4/\alpha)}{3(n-1)},\]
and $\widehat \sigma^2$ is the sample variance.

\paragraph{Bentkus-CI [B04]}
Bentkus' confidence interval requires an a-priori upper bound on $\Var(X_i)$ for each $i$. As alluded to in the introduction, we do not consider concentration bounds which require knowledge of the variance. However, since we assume $X_i \in [0, 1]$, we have the trivial upper bound, $\Var(X_i) \leq \tfrac{1}{4}$, which we implicitly use throughout our computation of Bentkus' confidence interval.

Define the independent, mean-zero random variables $(G_i)_{i=1}^n$ as 
\[ G_i := 
\begin{cases}
-\frac{1}{4} & \text{w.p. } \frac{4}{5} \\
1 & \text{w.p. } \frac{1}{5}
\end{cases} \quad , \]
an important technical device which has appeared in seminal works by \citet[Equation (2.14)]{hoeffding_probability_1963} and \citet[Equation (10)]{bennett_probability_1962}.
Then the ``Bentkus-CI'' is 
\[ \frac{1}{n} \sum_{i=1}^n X_i \pm \frac{W_\alpha^\star}{n}, \]
where $W_\alpha^\star \in [0, n]$ is given by the value of $W_\alpha$ such that
\[ \inf_{y \in [0, n]\ :\ y \leq W_\alpha} \frac{\EE \left[ \sum_{i=1}^n (G_i - y)_+^2 \right ]}{(W_\alpha - y)_+^2} = \alpha. \] 
Efficient algorithms have been developed to solve the above \citep[Section~9]{bentkus2006domination}, \citep{kuchibhotla2020near}.

\paragraph{PTL-$\ell_2$ [PTL21]}
The work by \citet{phan2021towards} proposes an interesting but computationally intensive approach to constructing confidence intervals for means of iid bounded random variables. Specifically, we will focus on their tightest bound (according to \cite[Figure 4]{phan2021towards}) which makes use of the $\ell_2$ norm in its derivation (and which we thus refer to as PTL-$\ell_2$).

For example, computing PTL-$\ell_2$ confidence intervals\footnote{We used code  by \citet{phan2021towards} with their default tuning parameters, available at \href{https://github.com/myphan9/small_sample_mean_bounds}{github.com/myphan9/small\_sample\_mean\_bounds}.} from a sample $X_1, \dots, X_{300} \sim \mathrm{Unif}[0, 1]$ of $n=300$ uniformly distributed random variables took upwards of 11 minutes while our betting confidence interval (Remark~\ref{remark:hedgedCI}) took less than 0.5 seconds. For this reason, we conduct a small-scale simulation of sample sizes 5-200 (see Figure~\ref{fig:PTL}). We find that PTL-$\ell_2$ performs extremely well for the low-variance continuous distribution Beta(10, 30) but poorly for sample sizes closer to 200 for Bernoulli data. Nevertheless, PTL-$\ell_2$ requires i.i.d. data (while we only require boundedness and conditional mean $\mu$) and PTL-$\ell_2$ does not have time-uniform or without-replacement analogues. 

\begin{figure}[!htbp]

 \centering
    \includegraphics[width=\textwidth]{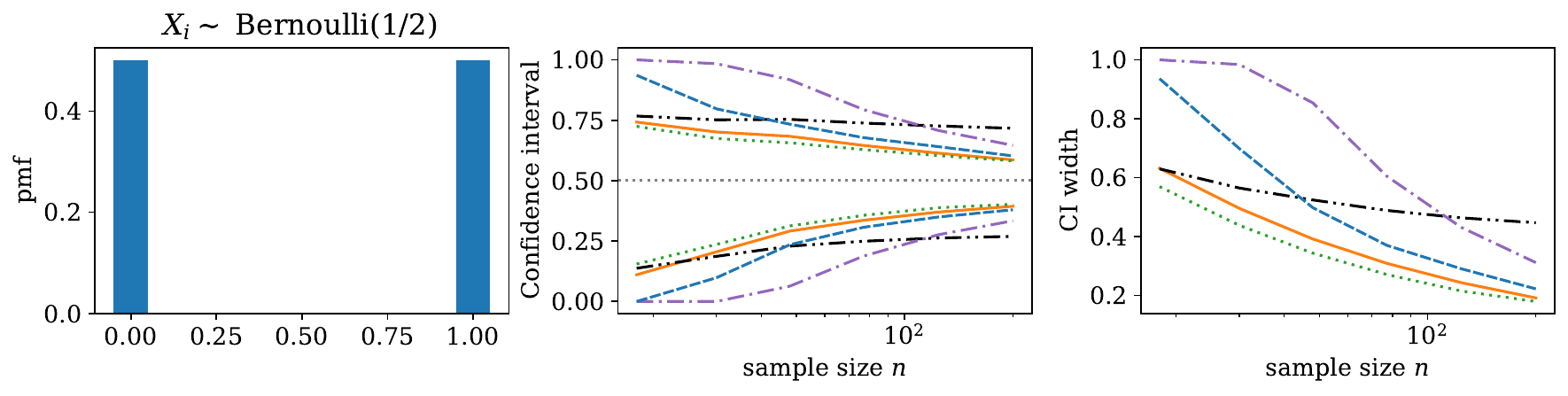}
 \hfill
    \includegraphics[width=\textwidth]{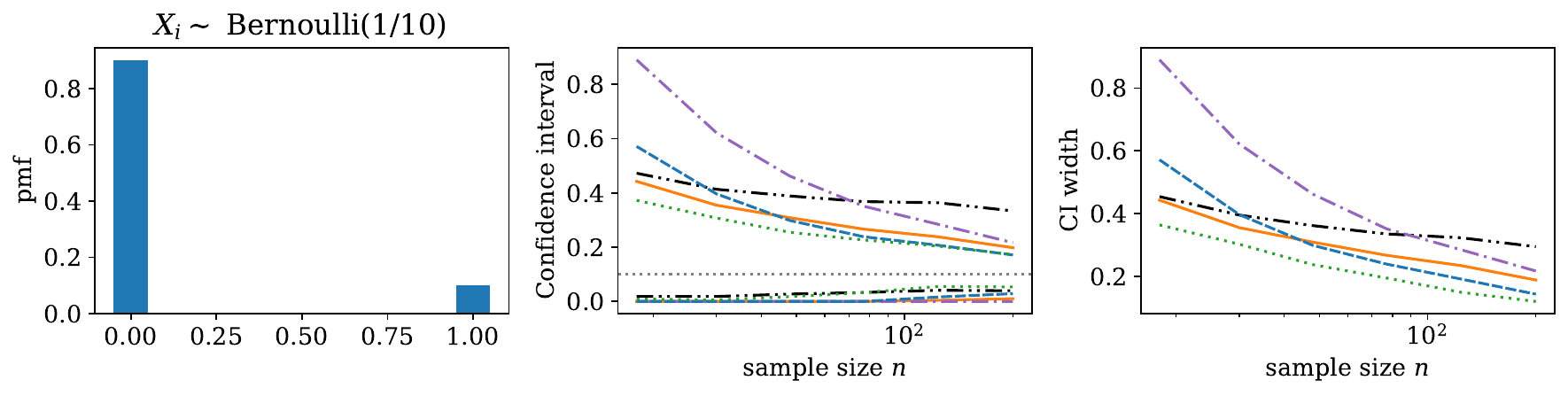}
    \centering
    \includegraphics[width=\textwidth]{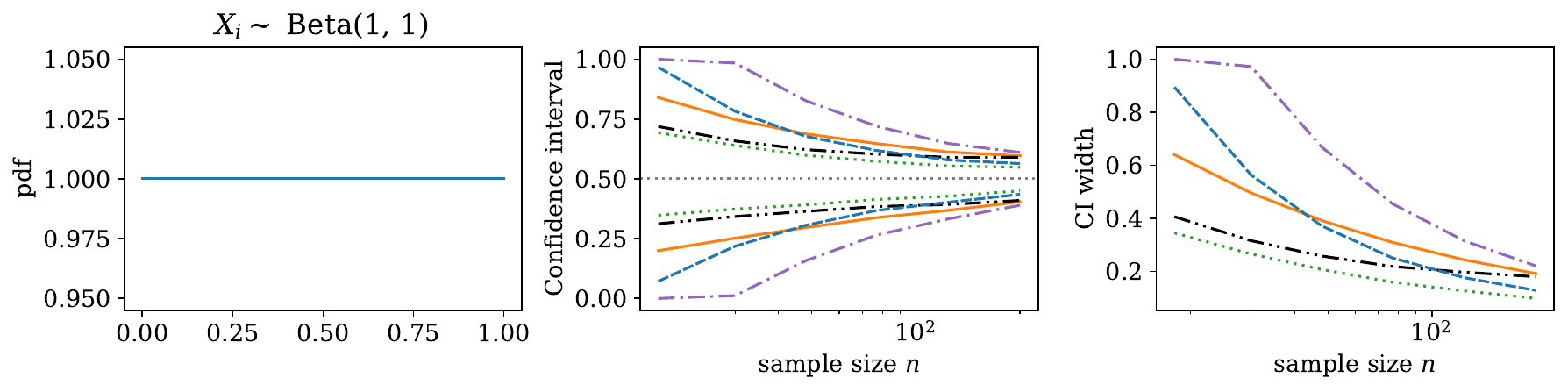}
 \hfill
    \includegraphics[width=\textwidth]{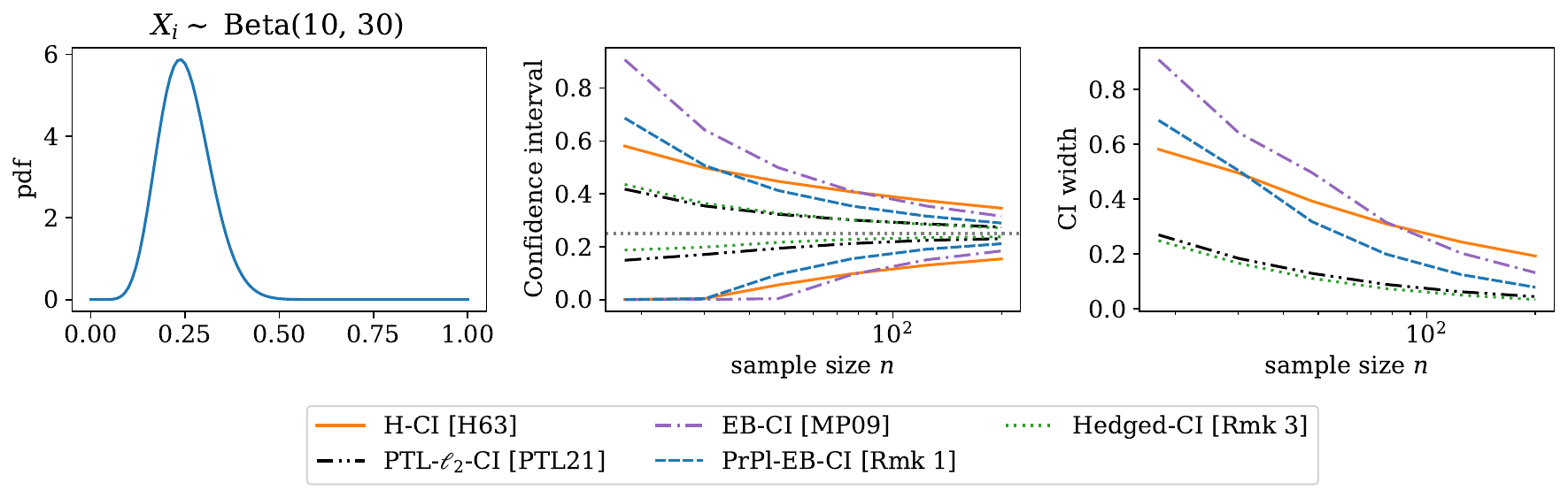}
 \caption{Various with-replacement fixed-time confidence intervals, including that of \citet{phan2021towards} (PTL-$\ell_2$-CI). While PTL-$\ell_2$-CI performs very well in the Beta(10, 30) regime, it appears to suffer for Bernoulli(1/2) with larger $n$. In any case, PTL-$\ell_2$-CI relies on iid data, while the other four methods do not.} 
 \label{fig:PTL}
\end{figure}

\subsection{Time-uniform confidence sequences (without replacement)}
\label{section:simDetails-TUWoR}
The WoR CSs which were introduced in this paper include
\begin{itemize}
    \item \textbf{Hedged-WoR}: Without replacement hedged capital process (Theorem~\ref{theorem:hedgedCSWoR}); and
    \item \textbf{hgKelly-WoR}: Without replacement analogue of hgKelly (Proposition~\ref{proposition:hedgedKellyInterval}).
\end{itemize}
The CSs labeled ``H-WoR [WR20]'' and ``EB-WoR [WR20]'' are the without-replacement Hoeffding- and empirical Bernstein-type CSs due to \citet{waudby2020confidence} which we recall now. 
\paragraph{H-WoR [WR20]}
Define the weighted WoR mean estimator and the Hoeffding-type $\lambda$-sequence,
\begin{equation*}
	\label{eqn:weightedWithoutReplaceMeanEst}
    \widehat \mu_t^\text{WoR}(\lambda_1^t) := \frac{\sum_{i=1}^t \lambda_i (X_i + \frac{1}{N-i+1} \sum_{j=1}^{i-1} X_j)}{\sum_{i=1}^t \lambda_i (1 + \frac{i-1}{N-i+1})}, ~~~\text{and}~~~ \lambda_t := \sqrt{\frac{8 \log(2/\alpha)}{t \log(t+1)}} \land 1,
\end{equation*}
respectively. Then ``H-CS [WR20]'' refers to the WoR Hoeffding-type CS,
\[ \widehat \mu_t^\text{WoR}(\lambda_1^t) \pm \frac{\sum_{i=1}^t \psi_{H}(\lambda_i) + \log(2/\alpha)}{\sum_{i=1}^t \lambda_i \left (1 + \frac{i-1}{N-i+1}\right )}.\]

\paragraph{EB-WoR [WR20]}
Analogously to the Hoeffding-type CSs, ``EB-CS [WR20]'' corresponds to the empirical Bernstein-type CSs for sampling WoR due to \citet{waudby2020confidence}. These CSs take the form
\[\widehat \mu_t^\WoR(\lambda_1^t) \pm \frac{\sum_{i=1}^t 4(X_i - \widehat \mu_{i-1})^2\psi_{E}(\lambda_i) + \log(2/\alpha)}{\sum_{i=1}^t \lambda_i \left (1 + \frac{i-1}{N-i+1}\right )},\]
where in this case, we have
\begin{equation}
    \lambda_t := \sqrt{\frac{2 \log(2/\alpha)}{\widehat \sigma_{t-1}^2 t \log(t+1)}} \land \frac12, \quad \widehat \sigma_t^2 := \frac{1/4 + \sum_{i=1}^t(X_i - \widehat \mu_i)^2}{t + 1}, ~~~\text{and}~~~ \widehat \mu_t := \frac{1}{t}\sum_{i=1}^t X_i.
\end{equation} 

\subsection{Fixed-time confidence intervals (without replacement)}
\label{section:simDetails-FTWoR}
The only fixed-time CI introduced in this paper is \textbf{Hedged-WoR-CI}: the without-replacement hedged capital process CI described in Section~\ref{sec:WoR}. The other two are both due to \citet{waudby2020confidence} which we describe now.

\paragraph{H-WoR-CI [WR20]}
This corresponds to the CI described in Corollary~3.1 of \citet{waudby2020confidence}. This has the form
\[\widehat \mu_n^\text{WoR} \pm \frac{\sqrt{\frac{1}{2} \log(2/\alpha)}}{\sqrt{n} + \frac{1}{\sqrt n} \sum_{i=1}^n \frac{i-1}{N-i+1}}.\]

\paragraph{EB-WoR-CI [WR20]}
Similarly, this CI corresponds to that described in Corollary~3.2 of \citet{waudby2020confidence}. Specifically, ``EB-WoR-CI [WR20]'' is defined as
\[ \widehat\mu_n^\WoR(\lambda_1^n) \pm \frac{\sum_{i=1}^n 4(X_i - \widehat \mu_{i-1})^2 \psi_E(\lambda_i) + \log(2/\alpha)}{\sum_{i=1}^n \lambda_i\left ( 1 + \frac{i-1}{N-i+1} \right )},\]
where
\begin{equation}
    \lambda_t := \sqrt{\frac{2 \log(2/\alpha)}{n\widehat \sigma_{t-1}^2}}\land \frac12 , \quad \widehat \sigma_t^2 := \frac{1/4 + \sum_{i=1}^t(X_i - \widehat \mu_i)^2}{t + 1}, ~\text{and}~ \widehat \mu_t := \frac{\tfrac{1}{2} + \sum_{i=1}^t X_i}{t + 1},
\end{equation} 
and $\widehat \mu_n^\WoR$ is defined as 
\[ \widehat \mu_t^\text{WoR}(\lambda_1^t) := \frac{\sum_{i=1}^t \lambda_i (X_i + \frac{1}{N-i+1} \sum_{j=1}^{i-1} X_j)}{\sum_{i=1}^t \lambda_i (1 + \frac{i-1}{N-i+1})}. \]

\subsection{Betting ``confidence distributions'': confidence sets at several resolutions}
\label{subsec:visual}
Figures~\ref{fig:CS-all-alpha} and~\ref{fig:conf-dist} demonstrate two tools to visualize CSs at various $\alpha$ and $t$.

\begin{figure}[h!]
    \centering
    \includegraphics[width=0.6\textwidth]{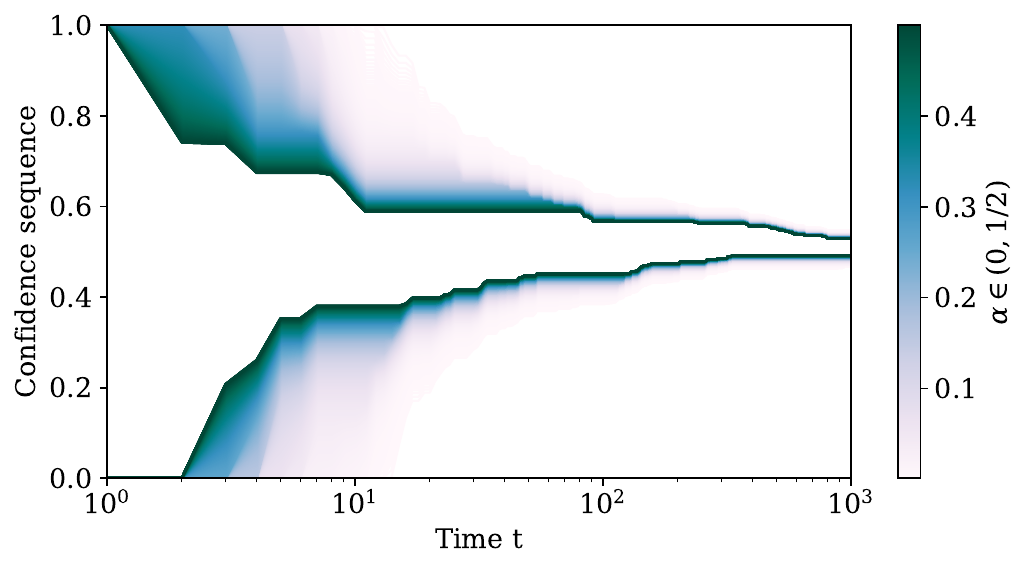}
    \caption{This plot shows the aGRAPA CS for all $\alpha\in[0,1/2]$ under $\text{Unif}[0,1]$ data.}
    \label{fig:CS-all-alpha}
\end{figure}

\begin{figure}[h!]
    \centering
    \includegraphics[width=\textwidth]{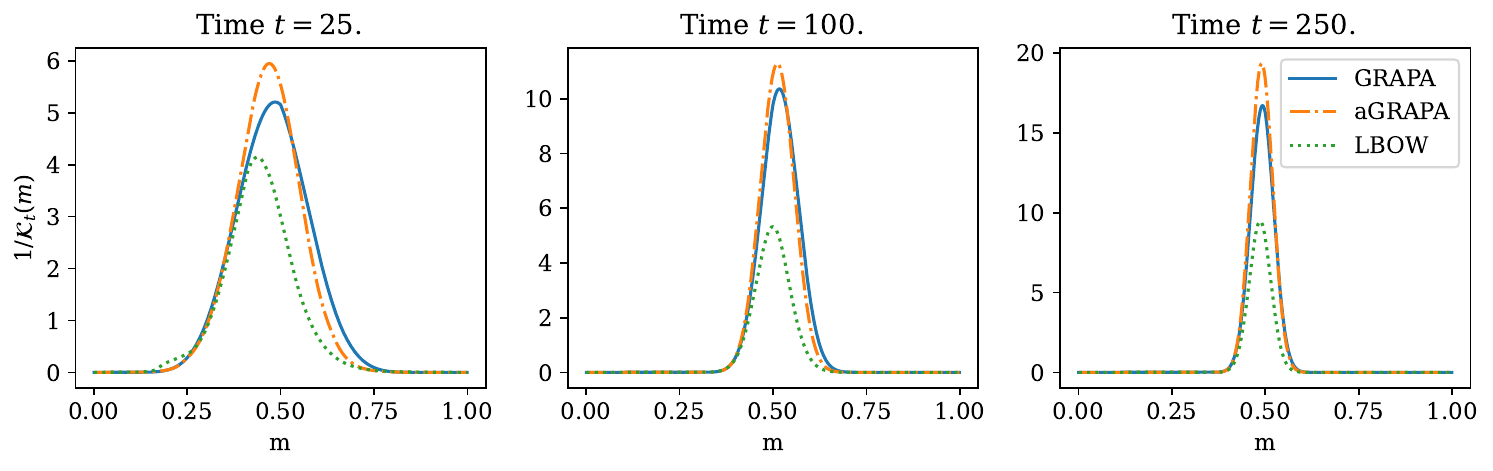}
    \caption{Here we plot the inverse wealth $1/\Kcal_t(m)$ in game $m$  against $m\in[0,1]$, at $t=25,100,250$ for three different betting strategies. Note the different $y$-axis scales. Despite not being normalized to yield a ``confidence distribution'', this is a useful visual tool. For example, the mode in each plot signifies the $m$ against which we have minimum wealth, which is a reasonable point estimator for $\mu$. Further, the superlevel set for any $\alpha \in [0,1]$ yields exactly the $(1-\alpha)$-CS for $\mu$ (for that corresponding time and strategy) since it yields all $m$ with wealth less than $1/\alpha$. Last, for any $m\in[0,1]$, the height (truncated at one) is anytime-valid $p$-value for the null hypothesis that the mean equals $m$.}
    \label{fig:conf-dist}
\end{figure}

\newpage

\section{Additional theoretical results}
\label{sec:app-theory}

\subsection{Betting confidence sets are tighter than Hoeffding}
\label{section:bettingBetterThanHoeffding}
In this section, we demonstrate that the betting approach can dominate Hoeffding for sufficiently large sample sizes. First, we show that for any $x, m \in (0, 1)$ and any $\lambda \in \mathbb R$, then $\gamma \equiv \gamma^m(\lambda)$ can be set as
\begin{equation*}
  \gamma^m(\lambda) := \exp \left \{ -m\lambda - \lambda^2/8 \right \}(\exp(\lambda) - 1) ,
\end{equation*}
so that
\[ H^m(x) := \underbrace{\exp \left \{ \lambda (x - m) - \lambda^2 / 8 \right \}}_{\text{Hoeffding term}} \leq \underbrace{1 + \gamma(x - m)}_{\text{Capital process term}} =: \Kcal^m(x) \]
for any $x,m \in [0, 1]$. In particular, the Hoeffding-type and capital process supermartingales are built from precisely the above terms, respectively, and so if $H^m(x) \leq \Kcal^m(x)$ for any $x \in [0, 1]$, then their respective supermartingales will satisfy the same inequality almost surely.
\begin{proposition}[Capital process dominates Hoeffding process]\label{proposition:capital-process-vs-hoeffding}
Suppose $x, m \in [0, 1]$ and $\lambda \in \mathbb R$. Then there exists $\gamma^m(\lambda) \in \mathbb R$ such that
\[ H^m(x) := \exp\left (\lambda(x-m) - \lambda^2/8 \right) \leq 1 + \gamma^m(\lambda)(x-m) =: \Kcal^m(x). \]
\end{proposition}
Note that Proposition~\ref{proposition:capital-process-vs-hoeffding} alone does not confirm that the Hoeffding-based CIs will be dominated by capital process-based CIs since $\gamma$ must be within $[-1/(1-m), 1/m]$ for $\Kcal^m(x)$ to be nonnegative. However, it is easy to verify that for all $\lambda \in [-0.45, 0.45]$, we have that $\gamma \in [-1, 1]$ and thus $\Kcal^m(x) \geq 0$. When constructing a Hoeffding-type $(1-\alpha)$-confidence interval, for example, one would set $\lambda_n^H := \sqrt{8\log(2/\alpha)/n}$, making $\lambda_n^H \in [-0.45, 0.45]$ whenever $n \geq 40\log(2/\alpha)$, in which case a capital process-based CI will dominate a Hoeffding-based CI almost surely.  
\begin{proof}[\myproofname{Proposition~\ref{proposition:capital-process-vs-hoeffding}}] 

We prove the result for $\lambda \geq 0 $ and remark that this implies the result for the case when $\lambda \leq 0$ by considering $(1-x)$ and $(1-m)$ instead of $x$ and $m$, respectively.

The proof proceeds in 3 steps. First, we consider the line segment $L^m(x)$ connecting $H^m(0)$ and $H^m(1)$ and note that by convexity of $H^m(x)$, we have that $H^m(x) \leq L^m(x)$ for all $x \in [0,1]$. We then find the slope of this line segment and set $\gamma$ to this value so that the line $\Kcal^m(x):= 1 + \gamma(x-m)$ has the same slope as $L^m(x)$. Finally, we demonstrate that $L^m(0) \leq \Kcal^m(0)$, and conclude that $H^m(x) \leq L^m(x) \leq \Kcal^m(x)$ for all $x \in [0, 1]$.

\paragraph{Step 1.}

Note that $H^m(x)$ is a convex function in $x \in [0,1]$, and thus \[\forall x \in [0, 1],\ H^m(x) \leq H^m(0) + \left [H^m(1) - H^m(0) \right]x =: L^m(x).\]
\paragraph{Step 2.} Observe that the slope of $L^m(x)$ is $H^m(1)-H^m(0)$. Setting $\gamma := H^m(1)-H^m(0)$ we have that $\Kcal^m(x)$ and $L^m(x)$ are parallel.
\paragraph{Step 3.} It remains to show that $ \Kcal^m(0) \geq L^m(0) \equiv H^m(0)$ for every $m \in [0, 1]$. Consider the following equivalent statements:
\begin{align*}
    &\Kcal^m(0) \geq H^m(0)\\
    \iff &1-m\left [ H^m(1) - H^m(0) \right] \geq H^m(0) \\
    \iff &1-m\exp \left(\lambda - \lambda m - \lambda^2 / 8\right) \geq (1-m) \exp\left ( -\lambda m - \lambda^2/8 \right) \\
    \iff &1 \geq \exp\left ( -\lambda m - \lambda^2/8 \right ) \left [ 1 - m + m \exp(\lambda)\right ]\\
    \iff &\exp \left(\lambda m + \lambda^2 / 8\right) \geq  \left [ 1 - m + m \exp(\lambda)\right ]\\
    \iff &a(\lambda) := \exp \left(\lambda m + \lambda^2 / 8\right) -  \left [ 1 - m + m \exp(\lambda)\right ] \geq 0.
\end{align*}
Now, note that $a$ is smooth and $a(0) = 0$ and so it suffices to show that its derivative $a'(\lambda) \geq 0$ for all $\lambda \geq 0$. To this end, consider the following equivalent statements.
\begin{align*}
    &a'(\lambda) \equiv \left ( m + \frac{\lambda}{4} \right ) \exp\left ( \lambda m + \lambda^2/8\right) - m\exp(\lambda)\geq 0 \\
    \iff &\left ( m + \frac{\lambda}{4} \right ) \exp\left ( \lambda m + \lambda^2/8\right) \geq m \exp(\lambda)\\
    \iff &\ln \left(1 + \frac{\lambda}{4m} \right) + \lambda m + \lambda^2/8 \geq \lambda \\
    \iff &b(\lambda) := \ln \left(1 + \frac{\lambda}{4m} \right) + \lambda m + \lambda^2/8 -\lambda \geq 0 ,
\end{align*}
and hence it suffices to show that $b(\lambda) \geq 0$. Similar to $a(\lambda)$, we have that $b(0) = 0$ and so it suffices to show that its derivative, $b'(\lambda) \geq 0$ for all $\lambda \geq 0$. Indeed,
\begin{align*}
    &b'(\lambda) \equiv \frac{1}{4m + \lambda} + m + \frac{\lambda}{4} - 1 \geq 0\\
    \iff &c(\lambda) := 1+m(4m+\lambda) +\frac{\lambda}{4} (4m + \lambda) - 4m - \lambda \geq 0
\end{align*}
Since $c(\lambda)$ is a convex quadratic, it is straightforward to check that
\[ \argmin_{\lambda \in \mathbb R} c(\lambda) = 2 - 4m, \]
and that $c(2-4m) = 0$.
In conclusion, if we set $\gamma \equiv \gamma^m(\lambda)$ as 
\[ \gamma^m(\lambda) := H^m(1) - H^m(0) = \exp \left \{ -m\lambda - \lambda^2/8 \right \}(\exp(\lambda) - 1), \]
then $H^m(x) \leq \Kcal^m(x):= 1 + \gamma^m(\lambda) (x- m)$ for every $m \in [0, 1]$.
This completes the proof. \myqed
\end{proof}

\subsection{Optimal convergence of betting confidence sets}
\label{section:optimalRates}
In Section~\ref{section:howToBet}, it was mentioned that for nonnegative martingales, Ville's inequality is nearly an equality and hence martingale-based CSs are nearly tight in a time-uniform sense. However, it is natural to wonder what other theoretical guarantees betting CSs/CIs can have in addition to their empirical performance. In the time-uniform setting, CSs for the mean cannot attain widths which scale faster than $\asymp \sqrt{\log\log t / t}$, due to the law of the iterated logarithm. Similarly, fixed-time CIs cannot scale faster than $\asymp 1/\sqrt{n}$. In this section, we show that it is possible to choose betting strategies such that the resulting CSs and CIs scale at the optimal rates of $O(\sqrt{\log\log t / t})$ and $O(1/\sqrt{n})$, respectively.

\subsubsection{An iterated logarithm betting confidence sequence}

We will establish the law of the iterated logarithm (LIL) convergence rate by carefully constructing a capital process martingale whose resulting CS is --- for sufficiently large $t$ ---  tighter than a larger CS which itself attains the required LIL rate. 

Before stating the result in Proposition~\ref{proposition:lil-betting-cs}, let $\zeta(s) := \sum_{k=1}^\infty \frac{1}{k^s}$ be the Riemann zeta function and for each $k \in \{1, 2, \dots\}$, define 
\begin{align*}
    \lambda_k &:= \sqrt{\frac{8\log\left ( k^s \zeta(s) \right )}{\eta^{k+1/2}}}, ~~~\text{and} \\
    \gamma_k(m) &= \exp \left \{-m\lambda_k - \lambda_k^2 / 8 \right \}(\exp(\lambda_k) - 1) \land 1,
\end{align*}
where $\eta > 1$ is some user-chosen constant.
Let $k_t$ denote the (unique) integer such that $\log_\eta t \leq k_t \leq \log_\eta t + 1$.
Define the process
\begin{align*}
\Kcal_t^\Lcal &:= \frac{1}{2}\Kcal_t^{\Lcal+}(m) + \frac{1}{2}\Kcal_t^{\Lcal-}(m) \\
\text{ where } \quad     \Kcal_t^{\Lcal+}(m) &:= \frac{1}{k_{t}^{s} \zeta(s)} \prod_{i=1}^t (1 + \gamma_{k_t}(X_i - m)) ~~~\text{and}\\
    \Kcal_t^{\Lcal-}(m) &:= \frac{1}{k_t^{s} \zeta(s)} \prod_{i=1}^t (1 - \gamma_{k_t}(X_i - m)).
\end{align*}
Note that $\Kcal_t^{\Lcal+}(m)$ and $\Kcal_t^{\Lcal-}(m)$ are both upper-bounded by the infinite mixtures
\begin{align}
    \Kcal_t^{\Lcal+} (m) &\leq \sum_{k=1}^\infty \frac{1}{k^{s} \zeta(s)} \prod_{i=1}^t (1 + \gamma_k(X_i - m)) ~~~\text{and}\label{eq:infinite-mixture-LIL-positive}\\
    \Kcal_t^{\Lcal-}(m) &\leq \sum_{k=1}^\infty \frac{1}{k^{s} \zeta(s)} \prod_{i=1}^t (1 - \gamma_k(X_i - m)), \label{eq:infinite-mixture-LIL-negative}
\end{align}
which themselves form nonnegative martingales when $m = \mu$ by Fubini's theorem.
Consequently,
\[ C_t^\Lcal := \left \{ m \in [0, 1] : \Kcal_t^\Lcal(m) < \frac{1}{\alpha} \right \} \]
forms a $(1-\alpha)$-CS for $\mu$. The following proposition establishes the LIL rate of $C_t^\Lcal$.

\begin{proposition}\label{proposition:lil-betting-cs}
    The CS $(C_t^\Lcal)_{t=1}^\infty$ has a width of $O(\sqrt{ \log\log t / t})$, meaning
    \[ \nu(C_t^\Lcal) = O \left ( \sqrt{\frac{\log \log t }{t}} \right ), \]
    where $\nu$ is the Lebesgue measure.
\end{proposition}

\begin{proof}
    The proof proceeds in three steps. In Step 1, we construct a distinct but related CS (which we will denote by $(C_t^\times)_{t=1}^\infty$) via the stitching technique \citep{howard_uniform_2019}. In Step 2, we demonstrate that this stitched CS achieves the desired rate by deriving an analytically tractible superset whose width scales as $O(\sqrt{\log \log t / t})$. Finally, in Step 3, we will show that the stitched CS $C_t^\times$ is a superset of $C_t^\Lcal$ for all $t$ sufficiently large, thus implying the final result. 

    \paragraph{Step 1. Constructing the stitched CS $C_t^\times$:}
    
    In the language of betting, the idea behind stitching is to first divide one's capital up into infinitely many portions $w_1, w_2, \dots$ such that $\sum_{k=1}^\infty w_k = 1$, and then place a constant bet $\lambda_k$ using a capital of $w_k$ on a designated epoch of time, which will be chosen to be geometrically spaced. In what follows, the portions $w_k$ will be given by $w_k = \frac{1}{\zeta(s) k^s}$, and we will divide time $\{1, 2, 3, \dots \}$ up into epochs demarcated by the endpoints $\eta^{k-1}$ and $\eta^{k}$ for each $k \in \{1,2,3, \dots\}$ and for some user-specified $\eta > 1$ (e.g. $\eta = 1.1$). The constant bets $\lambda_k$ will be chosen so that they are effective between $\eta^{k-1}$ and $\eta^k$ and lead to $O(\sqrt{\log \log t / t})$ widths after being combined across epochs.

    The construction of the stitched boundary essentially follows (a simplified version of) the proof of Theorem 1 in \citet[Section A.1]{howard_uniform_2019}, but we present the derivation here for completeness. Consider the Hoeffding-type process for a fixed $\lambda \in \RR$:
\begin{equation}
    \label{eq:hoeffdingStitch}
    M_t^{\lambda}(m) := \exp \left \{ \lambda S_t(m) - t \lambda^2 / 8 \right \} ,
\end{equation}
where $S_t(m) := \sum_{i=1}^t (X_i - m)$. As discussed in Section~\ref{section:warmup}, $M_t(\mu)$ forms a test supermartingale, and hence by Ville's inequality we have
\[ P\left ( \exists t \geq 1 : S_t(\mu) \geq \underbrace{\frac{r + t\lambda^2 / 8}{\lambda}}_{g_{\lambda, r}(t)} \right ) \leq e^{-r}. \]
We have typically used $r = \log(1/\alpha)$ throughout the paper, but the above alternative notation will help in the following discussion. Using the notation of \citet[Section A.1]{howard_uniform_2019}, define the boundary above as $g_{\lambda, r}(t) := (r + t \lambda^2/8)/\lambda$, and let 
\begin{align*}
    \lambda_k &:= \sqrt{\frac{8r_k}{\eta^{k-1/2}}}, \\
    \text{where} ~~~ r_k &:= \log \left ( \frac{k^s \zeta(s)}{\alpha/2} \right).
\end{align*}
Some algebra will reveal that plugging the above choices of $\lambda_k$ and $r_k$ into $g_{\lambda, r}(t)$ yields
\[ g_{\lambda_k, r_k}(t) := \sqrt{\frac{r_k t}{8}}\left ( \sqrt{\frac{\eta^{k-1/2}}{t}} + \sqrt{\frac{t}{\eta^{k-1/2}}}\right ), \]
resulting in the following concentration inequality for each $k$:
\[ P\left ( \exists t \geq 1 : S_t(\mu) \geq g_{\lambda_k, r_k}(t) \right ) \leq \exp\{-r_k\}. \]
Let $k_t$ denote the (unique) epoch number such that $\eta^{k_t-1} \leq t \leq \eta^{k_t}$ (i.e. such that $\log_\eta t \leq k_t \leq \log_\eta t + 1$). Now, we take a union bound over $k = 1, 2, 3, \dots$ resulting in the following boundary,
\[ P\left ( \exists t \geq 1 : S_t(\mu) \geq g_{\lambda_{k_t}, r_{k_t}}(t) \right ) \leq \sum_{k=1}^\infty \exp\{ -r_k \} = \frac{\alpha/2}{\zeta(s)} \underbrace{\sum_{k=1}^\infty \frac{1}{k^s}}_{\zeta(s)} = \alpha/2.\]
Repeating all of the previous steps for $-S(\mu)$ and taking a union bound, we arrive at the $(1-\alpha)$ stitched CS $(C_t^\times)_{t=1}^\infty$ given by
\[ C_t^\times := \left ( \frac{1}{t}\sum_{i=1}^t X_i \pm \frac{g_{\lambda_{k_t}, r_{k_t}}(t)}{t} \right ), \]
with the guarantee that $P(\exists t \geq 1 : \mu \notin C_t^\times ) \leq \alpha$.

\paragraph{Step 2. Demonstrating that $C_t^\times$ achieves the desired LIL width:}

Now, we will simply upper-bound $g_{\lambda_{k_t}, r_{k_t}}(t)$ by an analytical boundary depending explicitly on $t$ (rather than implicitly through $k_t$) to see that it achieves the desired LIL width. First, notice that $\sqrt{\eta^{k_t-1/2} / t} + \sqrt{t / \eta^{k_t-1/2}}$ is uniquely minimized when $t = \eta^{k_t-1/2}$ and hence its maximum on the interval $(\eta^{k_t-1}, \eta^{k_t})$ must be at the endpoints. Therefore, $\sqrt{\eta^{k_t-1/2} / t} + \sqrt{t / \eta^{k_t-1/2}} \leq \eta^{1/4} + \eta^{-1/4}$ and thus for each $k$, we have
\[ g_{\lambda_{k_t}, r_{k_t}}(t) \leq \sqrt{\frac{r_{k_t}t }{8}} \left ( \eta^{1/4} + \eta^{-1/4} \right ) ~~~\text{for all } \eta^{k_t-1} \leq t \leq \eta^{k_t}. \]
Furthermore, for all $\eta^{k_t-1} \leq t \leq \eta^{k_t}$, we have that $k_t \leq \log_\eta t + 1$. Applying this inequality to the above, we obtain the final bound which does not depend on $k$,
\[ g_{\lambda_{k_t}, r_{k_t}}(t) \leq \sqrt{\frac{t \log\left ( 2\left ( \log_\eta t + 1 \right )^s \zeta(s) / \alpha \right )}{8}} \left ( \eta^{1/4} + \eta^{-1/4} \right ) ~~~\text{for all } k. \]
In conclusion, we have that 
\[ C_t^\times \subseteq \left ( \frac{1}{t} \sum_{i=1}^t X_i \pm \sqrt{\frac{\log\left ( 2\left ( \log_\eta t + 1 \right )^s \zeta(s) / \alpha \right )}{8t}} \left ( \eta^{1/4} + \eta^{-1/4} \right ) \right ), \]
and thus $C_t^\times = O\left ( \sqrt{\log \log t / t} \right ), $ as desired.

\paragraph{Step 3. Showing that $C_t^\Lcal \subseteq C_t^\times$ for all $t$ large enough:}

This step in the proof essentially follows immediately from the discussion in Section~\ref{section:bettingBetterThanHoeffding}. We justified that for $\lambda \geq 0$, setting $\gamma$ as 
\[ \gamma = \exp \left \{-m\lambda - \lambda^2 / 8 \right \}(\exp(\lambda) - 1) \land 1, \]
yields $1 + \gamma (x - m) \geq \exp\left \{ \lambda(x - m) - \lambda^2 / 8 \right \}$ for all $x, m \in [0, 1]$ if $\lambda$ is sufficiently small (i.e. so that $\gamma$ is not relying on truncation at 1). 
Since $\lambda_k$ is decreasing in $t$, it follows that for $t$ sufficiently large,
\[ \prod_{i=1}^t (1 + \gamma_{k_t}(X_i - m)) \geq \exp\left \{ \lambda_{k_t} S_t(m) - \lambda_{k_t}^2/8 \right \} ~~~\text{almost surely.} \]
Therefore, for $t$ sufficiently large,
\begin{align*}
    \Kcal_t^{\Lcal+}(m) &:= \frac{1}{k_t^{s} \zeta(s)} \prod_{i=1}^t (1 + \gamma_{k_t}(X_i - m)) \\
                         &\geq \frac{1}{k_t^{s} \zeta(s)} \exp\left \{ \lambda_{k_t} S_t(m) - \lambda_{k_t}^2/8 \right \} =: H_t^{\infty+}(m)
\end{align*}
and similarly for $K_t^{\Lcal-}(m)$,
\begin{align*}
    \Kcal_t^{\Lcal-}(m) &\geq \frac{1}{k_t^{s} \zeta(s)} \exp\left \{ -\lambda_{k_t} S_t(m) - \lambda_{k_t}^2/8 \right \} =: H_t^{\infty-}(m) .
\end{align*}
Therefore, for sufficiently large $t$, we have
\begin{align*}
    C_t^\Lcal &:= \left \{ m \in [0, 1] : \Kcal_t^\Lcal(m) < \frac{1}{\alpha} \right\}\\
    &\subseteq \underbrace{\left \{ m \in \RR : \max \left \{ \frac{1}{2}H_t^{\infty+}(m) ,\ \frac{1}{2} H_t^{\infty -}(m) \right \} < \frac{1}{\alpha} \right \}}_{(\star)}
\end{align*}
and it is straightforward to verify that $(\star)$ is precisely $C_t^\times$.

In summary, we constructed a CS $C_t^\times$ using the stitching technique in Step 1, and then showed that $\nu(C_t^\times) = O(\sqrt{\log \log t / t})$ in Step 2. Finally in Step 3, we showed that our discrete mixture betting CS $C_t^\Lcal$ is a subset of $C_t^\times$ for $t$ sufficiently large, and hence by subadditivity of measures,
\[ \nu(C_t^\Lcal) = O\left ( \sqrt{\frac{\log \log t }{t}} \right ), \]
which completes the proof.\myqed
\end{proof}
\begin{remark}
 Notice that $\Kcal_t^{\Lcal+}$ and $\Kcal_t^{\Lcal-}$ can be made strictly more powerful if they are replaced by adding additional terms, as long as the final sums are upper-bounded by \eqref{eq:infinite-mixture-LIL-positive} and \eqref{eq:infinite-mixture-LIL-negative}, respectively. In particular, any finite sum analogue of \eqref{eq:infinite-mixture-LIL-positive} and \eqref{eq:infinite-mixture-LIL-negative} would have sufficed, as long as $\Kcal_t^{\Lcal+}$ and $\Kcal_t^{\Lcal-}$ form a term in each sum, respectively. We presented $\Kcal_t^{\Lcal+}$ and $\Kcal_t^{\Lcal-}$ in their current forms for the sake of notational (and computational) simplicity.
\end{remark}


\subsubsection{\texorpdfstring{The $\sqrt{n}$}{Root-n}-convergence of betting CIs}
\label{section:root-n-convergence-hedgedCI}

\begin{proposition}
\label{proposition:hedgedWidth}
Suppose $X_1^n \sim P$ are independent observations from a distribution $P \in \Pcal^\mu$ with mean $\mu \in [0, 1]$. Let $\lambda_n \in (0, 1)$ such that $\lambda_n \asymp 1/\sqrt{n}$. Then the confidence interval, 
\[ C_n := \left \{ m \in [0, 1] : \Kcal^\pm_n < \frac{1}{\alpha} \right \} ~~~\text{has an asymptotic width of $O(1/\sqrt{n})$.} \]
\end{proposition}
\begin{proof}
Writing out the capital process with positive bets, we have by Lemma~\ref{lemma:extendedFan} that for any $m \in [0, 1]$,
\begin{align*}
    \Kcal_n^+(m) &:= \prod_{i=1}^n ( 1 + \lambda_n (X_i - m)) \\
    &\geq \exp \left ( \lambda_n\sum_{i=1}^n (X_i - m) - \psi_E(\lambda_n)\sum_{i=1}^n 4(X_i - m)^2  \right ) \\
    &\geq \exp \left (  \lambda_n\sum_{i=1}^n(X_i - m) - 4n \psi_E(\lambda_n) \right ) =: B_t^+(m),
\end{align*}
and similarly for negative bets, 
\begin{align*}
    \Kcal_n^-(m) &:= \prod_{i=1}^n (1 - \lambda_n (X_i - m)) \\
    &\geq \exp \left (  -\lambda_n\sum_{i=1}^t (X_i - m) -4n \psi_E(\lambda_n) \right ) =: B_t^-(m).
\end{align*}
For any $\theta \in (0, 1)$, consider the set,
\[ \Scal_n := \left \{ m : B_t^+(m) < \frac{1}{\theta\alpha} \right \} \bigcap \left \{m : B_t^-(m) < \frac{1}{(1-\theta)\alpha} \right \} \]
Now notice that the $1/\alpha$-level set of $\Kcal_n^\pm(m) := \max \left \{\theta\Kcal_n^+(m), (1-\theta)\Kcal_n^-(m)\right \}$ is a subset of $\Scal_n$:
\[ C_n = \left \{m : \Kcal_n^+(m) < \frac{1}{\theta\alpha} \right \} \bigcap \left \{ m: \Kcal_n^-(m) < \frac{1}{(1-\theta)\alpha} \right \} \subseteq \Scal_n. \]
On the other hand, it is straightforward to derive a closed-form expression for $\Scal_n$: 
\[ \left (\frac{\sum_{i=1}^n X_i}{n} - \frac{\log \left ( \frac{1}{\theta\alpha} \right ) + 4n \psi_E(\lambda_n) }{n\lambda_n}, \frac{\sum_{i=1}^n X_i}{n} + \frac{\log \left (\frac{1}{(1-\theta)\alpha}\right) + 4n \psi_E(\lambda_n) }{n\lambda_n}\right ), \]
which in the typical case of $\theta = 1/2$ has the cleaner expression,
\[ \frac{\sum_{i=1}^n X_i}{n} \pm \frac{\log (2/\alpha) + 4n \psi_E(\lambda_n) }{n\lambda_n}. \]
As discussed in Section~\ref{section:howToBet}, we have by two applications of L'H\^opital's rule that 
\( \frac{\psi_E(\lambda_n)}{\psi_H(\lambda_n)} \xrightarrow{n \rightarrow \infty} 1, \)
where $\psi_H(\lambda_n) := \lambda_n^2 / 8 \asymp 1/n $ and thus the width $W_n$ of $\Scal_n$ scales as 
\[ W_n := 2\cdot \frac{\log (1/\alpha) + 4n \psi_E(\lambda_n) }{n\lambda_n} \asymp \frac{\log (1/\alpha) }{\sqrt n} + \frac{4 n / n}{\sqrt{n}} \asymp \frac{1}{\sqrt{n}}. \]
Since $C_n \subseteq \Scal_n$, we have that $C_n$ has a width of $O(1/\sqrt{n})$, which completes the proof. \myqed
\end{proof}
Despite these results, the hedged capital CI presented and recommended in Section~\ref{section:hedgedCapitalProcess} does not satisfy the assumptions of the above proof. In particular, we recommended using the variance-adaptive predictable plug-in,
\begin{equation}
\small
\label{eq:hedgedCI-PM}
\lambda_t^\VAEB := \sqrt{\frac{2 \log (2/\alpha)}{n \widehat \sigma_{t-1}^2}}, ~~~  \widehat \sigma_t^2 := \frac{1/4 + \sum_{i=1}^t ( X_i - \widehat \mu_i)^2}{t+1}, ~~~\text{and}~~~  \widehat \mu_t := \frac{ 1/2 + \sum_{i=1}^t X_i }{t + 1},
\end{equation}
using a truncation which depends on $m$,
\begin{equation}
\label{eq:HedgedCI-m-trunc}
\lambda_t^+(m) := \lambda_t^\pm \land \frac{c}{m}, ~~~ \lambda_t^-(m) := - \left (\lambda_t^\pm \land \frac{c}{1-m} \right ), 
\end{equation}
and finally defining the hedged capital process for each $t \in \{1, \dots, n\}$:
\[ \Kcal_t^\pm(m) := \max \left \{ \theta \prod_{i=1}^t( 1 + \lambda_i^+(m)\cdot (X_i - m)), (1-\theta) \prod_{i=1}^t( 1 - \lambda_i^-(m)\cdot (X_i - m)) \right \}. \]
Furthermore, the resulting CI is defined as an intersection, 
\begin{equation}
\label{eq:hedgedCI-intersection}
\BI_n := \bigcap_{t=1}^n \left \{ m \in [0, 1] : \Kcal_t^\pm(m) < \frac{1}{\alpha} \right \}.
\end{equation}
All of these tweaks (i.e. making bets predictable, truncating beyond $(0, 1)$, and taking an intersection) do not in any way invalidate the type-I error, but we find (through simulations) that they tighten the CIs, especially in low-variance, asymmetric settings (see Figure~\ref{fig:hedgedCI-tweaks}).
\begin{figure}[!ht]
 \centering
    \includegraphics[width=\textwidth]{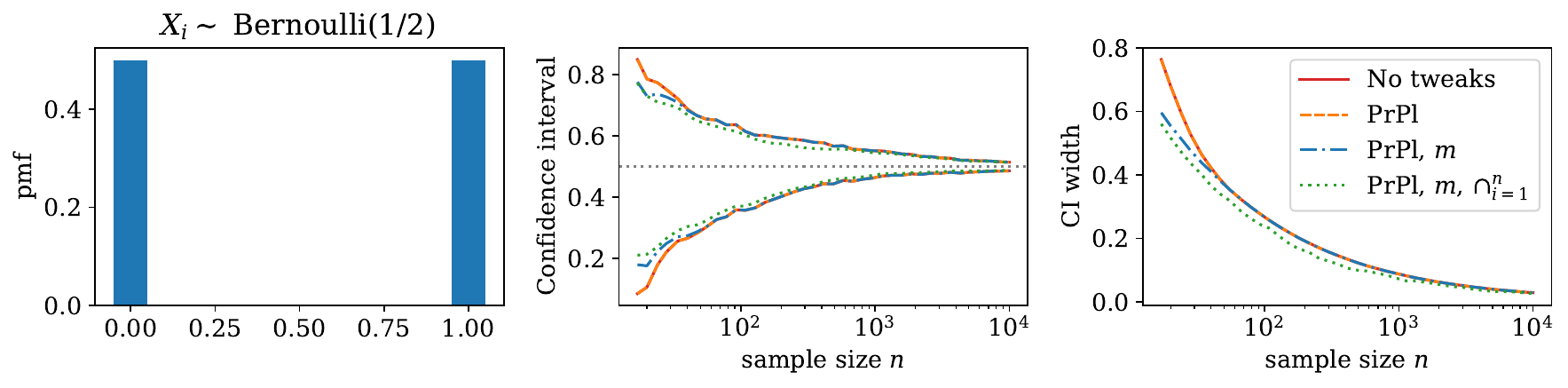}
 \hfill
    \centering
    \includegraphics[width=\textwidth]{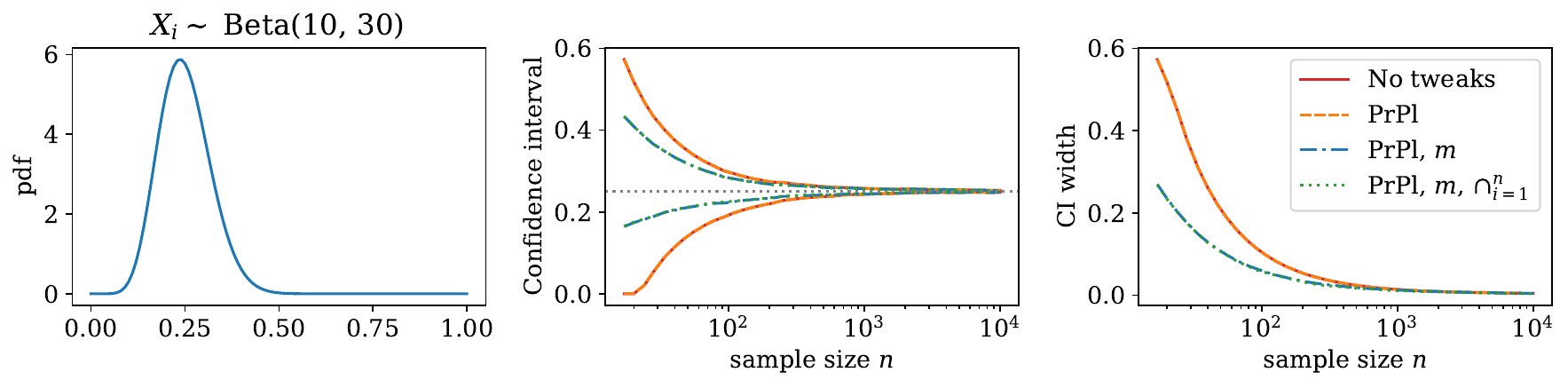}
 \caption{Hedged capital CIs with various added tweaks. The CIs labeled ``No tweaks'' refer to those which satisfy the conditions of Proposition~\ref{proposition:hedgedWidth}. The other three plots differ in which ``tweaks'' have been added. Those with ``PrPl'' in the legend use the predictable plug-in approach defined in \eqref{eq:hedgedCI-PM}; those with $m$ in the legend have been truncated using $m$ as outlined in \eqref{eq:HedgedCI-m-trunc}; finally, the plots with $\cap_{i=1}^n$ in their legends had their running intersections taken as in \eqref{eq:hedgedCI-intersection}.} 
 \label{fig:hedgedCI-tweaks}
\end{figure}

\subsection{On the width of empirical Bernstein confidence intervals}
\label{section:VAEB_scale}
Recall the predictable plug-in empirical Bernstein confidence interval:
\[ C_n^\VAEB := \left ( \frac{\sum_{i=1}^n \lambda_i X_i }{\sum_{i=1}^n \lambda_i} \pm \frac{\log (2/\alpha) + \sum_{i=1}^n v_i \psi_E (\lambda_i)}{\sum_{i=1}^n \lambda_i} \right ), \]
where 
\[ \lambda_t := \sqrt{\frac{2 \log (2/\alpha)}{n \widehat \sigma_{t-1}^2}},~ \widehat \sigma^2_{t} := \frac{\tfrac{1}{4} + \sum_{i=1}^t (X_i - \widehat \mu_i)^2 }{t + 1}, \text{and}~ \widehat \mu_t := \frac{\tfrac12 + \sum_{i=1}^t X_i }{t + 1}. \]
Below, we analyze the asymptotic behavior of the width of $C_n^\VAEB$ in the i.i.d. setting. In Proposition~\ref{proposition:vaeb-width}, we will show that
if the data are drawn i.i.d.\ from a distribution $Q \in \Qcal^\mu$ having variance $\sigma^2$, then the half-width $W_n$ of $C_n^\VAEB$ scales as
\begin{equation}
    \sqrt{n} W_n ~\equiv~ \sqrt{n} \left ( \frac{\log(2/\alpha) + \sum_{i=1}^n v_i \psi_E(\lambda_i)}{\sum_{i=1}^n \lambda_i}  \right ) \xrightarrow{a.s.} \sigma \sqrt{2 \log (2/\alpha)},
\end{equation}
and hence the width is asymptotically proportional to the standard deviation. 

First, let us prove a few lemmas about \textit{nonrandom} sequences of numbers, which will be helpful in what follows. These are simple facts for which we could not find a proof to reference, so we prove them below for completeness.

\begin{lemma}
\label{lemma:avgOfConvergent}
    Suppose $(a_n)_{n=1}^\infty$ is a sequence of real numbers such that $a_n \rightarrow a$. Then their cumulative average also converges to $a$, meaning that
    \(\frac{1}{n} \sum_{i=1}^n a_i \rightarrow a. \)
\end{lemma}
\begin{proof}
    Let $\epsilon > 0$ and choose $N \equiv N_\epsilon \in \NN$ such that whenever $n \geq N$, we have 
    \begin{equation}
    \label{eq:a_n-a}
        | a_n - a | < \epsilon.
    \end{equation} Moreover, choose 
    \begin{equation}
        \label{eq:M_N1}
        M \equiv M_{N} > \frac{\sum_{i=1}^N |a_i - a|}{\epsilon}
    \end{equation} and note that 
    \begin{equation}
        \label{eq:M_N2}
        \frac{n - N - 1}{n} < 1.
    \end{equation}
    Let $n \geq \max\left \{ N, M \right \}$. Then we have by the triangle inequality,
    \begin{align*}
        \left | \frac{1}{n} \sum_{i=1}^n (a_i -a) \right | &\leq  \frac{1}{n}\sum_{i=1}^N | a_i - a | + \frac{1}{n} \sum_{i=N+1}^n |a_i - a| \\
        &\leq \frac{1}{n}\sum_{i=1}^N | a_i - a | + \frac{1}{n} (n - N - 1)\epsilon & \text{by \eqref{eq:a_n-a}}\\
        &\leq 2\epsilon & \text{by \eqref{eq:M_N1} and \eqref{eq:M_N2}},
    \end{align*}
    which can be made arbitrarily small. This completes the proof of Lemma~\ref{lemma:avgOfConvergent}.\myqed
\end{proof}

\begin{lemma}
\label{lemma:productConvergentBounded}
    Let $(a_n)_{n=1}^\infty$ and $(b_n)_{n=1}^\infty$ be sequences of numbers such that 
    \begin{align}
        a_n &\rightarrow 0 ~ \text{ and}\label{eq:a_convergent} \\
        |b_n| &\leq C ~ \text{ for some $C \geq 0$ and for all $n \geq 1$.} \label{eq:b_bounded_by_C}
    \end{align}
    Then $a_n b_n \rightarrow 0$. Further, if $(A_n)$ is a sequence of random variables such that $A_n \to 0$ almost surely, then $A_nb_n \to 0$ almost surely. 
\end{lemma}
The proof is trivial, since $|A_nb_n| \leq C|A_n|$ which converges to zero almost surely.\qed

Now, we prove that a modified variance estimator is consistent.
\begin{lemma}
    \label{lemma:modifiedVarianceConsistent}
    Let $X_1, \dots, X_n \stackrel{i.i.d.}{\sim} Q \in \Qcal^\mu$ with $\Var(X_i) = \sigma^2$. Then the modified variance estimator
    \[ \widehat \sigma_n^2 := \frac{1}{n}\sum_{i=1}^n (X_i - \widehat\mu_{i-1})^2 \]
    converges to $\sigma^2$,  $Q$-almost surely.
\end{lemma}
\begin{proof}
By direct substitution, 
    \begin{align*}
        \widehat \sigma_n^2 &:= \frac{1}{n}\sum_{i=1}^n (X_i - \widehat\mu_{i-1})^2 ~= \frac{1}{n}\sum_{i=1}^n (X_i - \mu + \mu - \widehat \mu_{i-1} )^2 \\
        &= \underbrace{\frac{1}{n} \sum_{i=1}^n (X_i - \mu)^2}_{\xrightarrow{a.s.} \sigma^2} - \underbrace{\frac{2}{n}\sum_{i=1}^n (X_i - \widehat \mu_{i-1})(\widehat \mu_{i-1} - \mu)}_{(\star)} + \underbrace{\frac{1}{n} \sum_{i=1}^n(\mu - \widehat \mu_{i-1} )^2}_{(\star \star)}.
    \end{align*}
    Now, note that $\widehat \mu_{i-1} - \mu \xrightarrow{a.s.} 0$ and $|X_i - \widehat \mu_{i-1}| \leq 1$ for each $i$. Therefore, by Lemma~\ref{lemma:productConvergentBounded}, $(X_i - \widehat \mu_{i-1})(\widehat \mu_{i-1} - \mu) \xrightarrow{a.s.} 0$, and by Lemma~\ref{lemma:avgOfConvergent}, $(\star) \xrightarrow{a.s.} 0$. Furthermore, we have that $(\mu - \widehat \mu_{i-1})^2 \xrightarrow {a.s.} 0$ and so by another application of Lemma~\ref{lemma:avgOfConvergent}, we have $(\star \star) \xrightarrow{a.s.} 0$. This completes the proof of Lemma~\ref{lemma:modifiedVarianceConsistent}. \myqed
\end{proof}

\bigskip

Next, let us analyze the second term in the numerator in the margin of $C_n^{\VAEB}$,
\begin{equation}
\label{eq:VAEB_margin}
    \frac{\log(2/\alpha) + \sum_{i=1}^n v_i \psi_E(\lambda_i)}{\sum_{i=1}^n \lambda_i}.
\end{equation}
\begin{lemma}
\label{lemma:VAEB_margin_numerator}
    Under the same assumptions as Lemma~\ref{lemma:modifiedVarianceConsistent},
    \[\sum_{i=1}^n v_i \psi_E(\lambda_i) \xrightarrow{a.s.} \log(2/\alpha).\]
\end{lemma}

\begin{proof}
Recall that 
\(
    \frac{\psi_E(\lambda)}{\psi_H(\lambda)} \xrightarrow{\lambda \rightarrow 0} 1, \text{ and } 
    \widehat \sigma^2_t \xrightarrow{t \rightarrow \infty} \sigma^2.
\)
By definition of $\lambda_i$, we have that $\lambda_i \xrightarrow{a.s.} 0$ and thus we may also write
\begin{align}
    \frac{\psi_E(\lambda_i)}{\psi_H(\lambda_i)} &= 1 + R_i \label{eq:psi_E-psi_H-convergenceRate} ~~\text{and}~~\\
    \sqrt{\frac{\sigma^2}{\widehat \sigma_t^2}} &= 1 + R'_i \label{eq:varianceRatioConvergenceRate}
\end{align}
for some $R_i, R_i' \xrightarrow{a.s.} 0$. Thus, we rewrite the left hand side of the claim as
    \begin{align*}
        \sum_{i=1}^n v_i \psi_E(\lambda_i) &= \sum_{i=1}^n v_i \psi_H(\lambda_i) \frac{\psi_E(\lambda_i)}{\psi_H(\lambda_i)} 
        = \sum_{i=1}^n v_i (\lambda_i^2/8) (1 + R_i)   \\
        &= \sum_{i=1}^n v_i \cdot \frac{2 \log(2/\alpha)}{8 \widehat n\sigma_{i-1}^2} \cdot (1 + R_i) \\
        &= \sum_{i=1}^n v_i \cdot \frac{2 \log(2/\alpha)}{8 n\sigma^2} \cdot(1 + R_i') \cdot (1 + R_i) \\
        &= \sum_{i=1}^n 4 (X_i - \widehat \mu_{i-1})^2 \cdot \frac{2 \log(2/\alpha)}{8n \sigma^2} \cdot(1 + R_i + R_i' + R_iR_i').
\end{align*}
Defining $R''_i = R_i + R_i' + R_iR_i'$ for brevity, and noting that $R''_i \to 0$ almost surely, the above expression becomes
\begin{align*}
       \sum_{i=1}^n v_i \psi_E(\lambda_i)  &= \sum_{i=1}^n (X_i - \widehat \mu_{i-1})^2 \cdot \frac{\log(2/\alpha)}{n\sigma^2} \cdot(1 + R_i'') \\
        &= \frac{\log (2 / \alpha)}{\sigma^2} \left[ \frac{1}{n} \sum_{i=1}^n (X_i - \widehat \mu_{i-1})^2 \cdot (1 + R_i'') \right ] \\
        &= \frac{\log (2 / \alpha)}{\sigma^2} \left[ \underbrace{\frac{1}{n} \sum_{i=1}^n (X_i - \widehat \mu_{i-1})^2}_{\xrightarrow{a.s.} \sigma^2 \text{ by Lemma~\ref{lemma:modifiedVarianceConsistent}}} + \underbrace{\frac{1}{n} \sum_{i=1}^n (X_i - \widehat \mu_{i-1})^2 R_i''}_{\xrightarrow{a.s.} 0 \text{ by Lemma~\ref{lemma:productConvergentBounded}}} \right ]
        \xrightarrow{a.s.} \log (2 / \alpha),
\end{align*}
which completes the proof of Lemma~\ref{lemma:VAEB_margin_numerator}.\myqed
\end{proof}
Now, consider the denominator in \eqref{eq:VAEB_margin}.
\begin{lemma}
Continuing with the same notation,
\label{lemma:VAEB_margin_denominator}
    \[ \frac{1}{\sqrt{n}}\sum_{i=1}^n \lambda_i \xrightarrow{a.s.} \sqrt{\frac{2 \log(2/\alpha)}{\sigma^2}}. \]
\end{lemma}
\begin{proof}
Let $R_i'$ be as in \eqref{eq:varianceRatioConvergenceRate}. Then,
    \begin{align*}
        \frac{1}{\sqrt{n}} \sum_{i=1}^n \lambda_i &= \frac{1}{\sqrt{n}} \sum_{i=1}^n\sqrt{\frac{2 \log (2/\alpha) }{n \widehat \sigma^2_{i-1}}} \\
        &= \frac{1}{\sqrt{n}} \sum_{i=1}^n\sqrt{\frac{2 \log (2/\alpha) }{ n \sigma^2}} \cdot \left( 1 + R_i' \right ) \\
        &= \sqrt{\frac{2 \log (2/\alpha) }{ \sigma^2 }} \cdot \underbrace{\frac{1}{n} \sum_{i=1}^n \left (1 + R_i' \right )}_{ \xrightarrow{a.s.} 1 \text{ by Lemma~\ref{lemma:avgOfConvergent}} } 
         \xrightarrow{a.s.} \sqrt{\frac{2 \log (2/\alpha) }{ \sigma^2 }},
    \end{align*}
    completing the proof of Lemma~\ref{lemma:VAEB_margin_denominator}.\myqed
\end{proof}

We are now able to combine Lemmas~\ref{lemma:VAEB_margin_numerator} and \ref{lemma:VAEB_margin_denominator} to prove the main result.

\begin{proposition}\label{proposition:vaeb-width}
Denoting the half-width of $C_n^{\VAEB}$ as $W_n$, and assuming the data are drawn iid from a distribution $Q \in \Qcal^\mu$ with variance $\sigma^2$, we have
\begin{equation}
    \sqrt{n} W_n ~\equiv~ \sqrt{n} \left ( \frac{\log(2/\alpha) + \sum_{i=1}^n v_i \psi_E(\lambda_i)}{\sum_{i=1}^n \lambda_i}  \right ) \xrightarrow{a.s.} \sigma \sqrt{2 \log (2/\alpha)}.
\end{equation}
Thus, the width is asymptotically proportional to the standard deviation.
\end{proposition}
\begin{proof}
By direct rearrangement of the left hand side, we see that
\begin{align*}
    \sqrt{n} \left ( \frac{\log(2/\alpha) + \sum_{i=1}^n v_i \psi_E(\lambda_i)}{\sum_{i=1}^n \lambda_i}  \right ) &=  \frac{\log(2/\alpha) + \sum_{i=1}^n v_i \psi_E(\lambda_i)}{\frac{1}{\sqrt{n}}\sum_{i=1}^n \lambda_i} \\
    &\xrightarrow{a.s.} \frac{\log(2/\alpha)+ \log(2/\alpha)}{\sigma^{-1} \sqrt{2\log (2/\alpha)}} = \sigma \sqrt{2\log(2/\alpha)},
\end{align*}
which completes the proof of Proposition~\ref{proposition:vaeb-width}.\myqed
\end{proof}

\subsection{aGRAPA sublevel sets need not be intervals: a worst-case example}
\label{section:aSOSsublevelSets}
In the proof of Theorem~\ref{theorem:hedgedCS}, we demonstrated that the hedged capital process with predictable plug-in bets yielded convex confidence sets, making their construction more practical. However, this proof was made simple by taking advantage of the fact that the sequences before truncation $(\dot \lambda_t^+)_{t=1}^\infty$ and $(\dot \lambda_t^-)_{t=1}^\infty$ did not depend on $m \in [0, 1]$. This raises the natural question, of whether there are betting-based confidence sets which are nonconvex when these sequences depend on $m$. Here, we provide a (somewhat pathological) example of the aGRAPA process with nonconvex sublevel sets.

Consider the aGRAPA bets,
\begin{equation} \small \lambda_t^\aSOS := \frac{\widehat \mu_{t-1} - m}{\widehat \sigma_{t-1}^2 + (\widehat \mu_{t-1} - m)^2} ~\text{where}~ \widehat \mu_t := \frac{1/2 + \sum_{i=1}^t X_i}{t+1},~\widehat \sigma_t^2 := \frac{1/20 + \sum_{i=1}^t (X_i - \widehat \mu_i)^2}{t+1}. \end{equation}
Furthermore, suppose that the observed variables are $X_1 = X_2 = 0$. Then it can be verified that 
\begin{align*}
\Kcal_2^\aSOS(m) &= \left (1 + \lambda_1^\aSOS(X_1 - m) \right ) \left ( 1 + \lambda_2^\aSOS(X_2 - m) \right ) \\
&= \left (1 + \frac{1/2 - m}{1/20 + (1/2-m)^2 }(-m) \right ) \left ( 1 + \frac{1/4 - m}{0.05625
 + (1/4 - m)^2} (-m) \right ),
\end{align*}
which does not yield convex sublevel sets. For example, $\Kcal_2^\aSOS(0.08) < 0.85$ and $\Kcal_2^\aSOS(0.4) < 0.85$ but $\Kcal_2^\aSOS(0.03) > 0.85$. In particular, the sublevel set, 
\[  \left \{m \in [0, 1] : \Kcal_2^\aSOS(m) < 0.85 \right \} \]
is not convex. In our experience, however, situations like the above do not arise frequently. In fact, we needed to actively search for these examples and use a rather small ``prior'' variance of $1/20$ which we would not use in practice. Furthermore, the sublevel set given above is at the $0.85$ level while confidence sets are compared against $1/\alpha$ which is always larger than 1 and typically larger than 10. We believe that it may be possible to restrict $(\lambda_t^\aSOS)_{t=1}^\infty$ and/or the confidence level, $\alpha \in (0, 1)$ in some way so that the resulting confidence sets are convex. One reason to suspect that this may be possible is because of the intimate relationship between $\lambda^\aSOS_t$, $\lambda_t^\SOS$, and the optimal hindsight bets, $\lambda^\HS$. Specifically, we show in Section~\ref{section:EL} that the optimal hindsight capital $\Kcal^\HS_t$ is exactly the empirical likelihood ratio \citep{owen2001empirical} which is known to generate convex confidence sets for the mean \citep{hall1990methodology}. We leave this question as a direction for future work. 

\subsection{Betting confidence sequences for non-iid data}
\label{section:noniid}
The CSs presented in this paper are valid under the assumption that each observation is bounded in $[0, 1]$ with conditional mean $\mu$. That is, we require that $X_1, X_2, \dots$ are $[0, 1]$-valued with 
\(\EE(X_t \mid \Fcal_{t-1} ) = \mu \text{ for each $t$}, \)
which includes familiar regimes such as independent and identically-distributed (iid) data from some common distribution $P$ with mean $\mu$. Despite the generality of our results, we made matters simpler by focusing the simulations in Section~\ref{section:simulations} on the iid setting. For the sake of completeness, we present a simulation to examine the behavior of our CSs in the presence of some non-iid data.
\begin{figure}[!htbp]
 \centering
 \textbf{250 observations from Beta(10, 10), followed by all Bernoulli(1/2)}
    \includegraphics[width=0.75\textwidth]{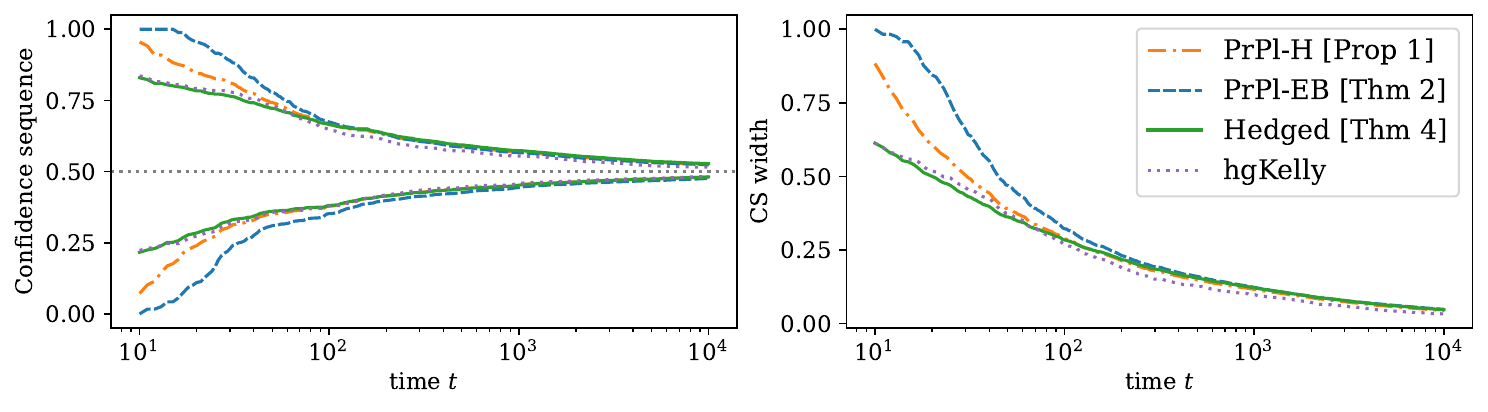}
 \hfill\\
 \textbf{2500 observations from Beta(10, 10), followed by all Bernoulli(1/2)}
    \centering
    \includegraphics[width=0.75\textwidth]{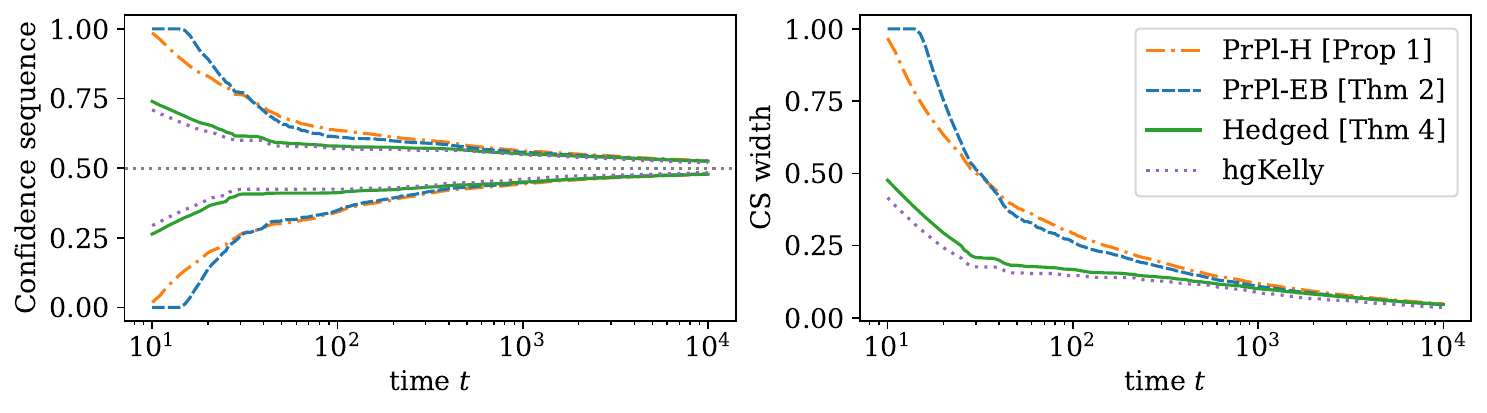}
 \caption{CSs for the true mean $\mu=1/2$ for non-iid data. In top pair of plots, the first 250 observations were independently drawn from a Beta(10, 10) while the subsequent observations are drawn from a Bernoulli(1/2). The bottom pair of plots is similar, but with 2500 initial draws from a Beta(10, 10) instead of 250. In both cases, the betting-based CSs (Hedged and ConBo) tend to outperform those based on supermartingales.}
 \label{fig:non-iid}
\end{figure}

In this setup, we draw the first several hundred or thousand observations independently from a Beta(10, 10) --- a distribution whose mean is $1/2$ but whose variance is small ($\approx 0.012$) --- while the remaining observations are independently drawn from a Bernoulli($1/2$) whose mean is also $1/2$ but with a maximal variance of $1/4$. We chose to start the data off with low-variance observations in an attempt to ``trick'' our betting strategies into adapting to the wrong variance. Empirically, we find that the hedged capital (Theorem~\ref{theorem:hedgedCS}) and ConBo (Corollary~\ref{corollary:ConBo}) CSs start off strong, adapting to the small variance of a Beta(10, 10). After several Bernoulli(1/2) observations, the CSs remain tight, but seem to shrink less rapidly. Nevertheless, we find that the hedged capital and ConBo CSs greatly outperform the Hoeffding (Proposition~\ref{proposition:predmixHoeffding}) and empirical Bernstein (Theorem~\ref{theorem:EBCS}) predictable plug-in CSs (see Figure~\ref{fig:non-iid}). Regardless of empirical performance, all methods considered produce \textit{valid} CSs for $\mu$. 

\subsection{Owen's empirical likelihood ratio and Mykland's dual likelihood ratio}
\label{section:EL}
Let $x_1, \dots, x_t \in [0, 1]$ and recall the optimal hindsight capital process $\Kcal^\HS_t(m)$,
\[ \Kcal^\text{HS}_t(m) := \prod_{i=1}^t (1 + \lambda^\text{HS}(x_i - m)) ~~~\text{where $\lambda^\text{HS}$ solves}~~~\sum_{i=1}^t \frac{x_i - m}{1+\lambda^\text{HS}(x_i-m)} = 0. \]
Now, let $\Qcal^m \equiv \Qcal^m(x_1^t)$ be the collection of discrete probability measures with support $\{x_1, \dots, x_t\}$ and mean $m$. Let $\Qcal \equiv \Qcal(x_1^t) := \bigcup_{m \in [0, 1]}\Qcal^m$ and define the empirical likelihood ratio \citep{owen2001empirical},
\[ \EL_t(m) := \frac{\sup_{Q \in \Qcal} \prod_{i=1}^t Q(x_i)}{\sup_{Q \in \Qcal^m} \prod_{i=1}^t Q(x_i) }. \]
\citet{owen2001empirical} showed that the numerator equals $(1/t)^t$ and the denominator equals
\[ \prod_{i=1}^t (1 + \lambda^\EL(x_i - m))^{-1} ~~~\text{where $\lambda^\EL$ solves}~~~ \sum_{i=1}^t \frac{x_i - m}{1+\lambda^\EL(x_i-m)} = 0. \]
Notice that the above product is exactly the reciprocal of $\Kcal_t^\text{HS}$ and that $\lambda^\EL = \lambda^\text{HS}$. Therefore for each $m \in [0, 1]$,
\[ \EL_t(m) = (1/t)^t \Kcal_t^\text{HS}(m). \]
Furthermore, given the connection between the empirical and dual likelihood ratios for independent data \citep{mykland1995dual}, the hindsight capital process is also proportional to the dual likelihood ratio in this case.

\newpage
\section{An extended history of betting and its applications} \label{section:history}

(This is an expanded version of Section~\ref{section:history-short} and Figure~\ref{fig:history}.)

The use of betting-related ideas in probability, statistics, optimization, finance and machine learning has evolved in many different parallel threads, emanating from different influential early works and thus having different roots and evolutions. Since these threads have had little interaction for many decades now, we consider it worthwhile to mention them in some detail. Two notes of caution:
\begin{itemize}
    \item We anticipate missing some authors and works in our broad strokes below, but a thorough coverage would be better suited to a longer survey paper on the topic. For example, we entirely skip the field of mathematical finance, since betting is literally a foundation of the entire field (and theoretical and applied progress on martingales, betting strategies, and related topics has been phenomenal).
    \item Many of the authors listed below have used the language of betting in their works explicitly, but others have not --- and may even prefer (or have preferred) \emph{not} to do so. Thus, our references should be treated with a pinch of salt, as some connections that we draw to betting may be more apparent in hindsight (to us) than foresight (to the authors).
\end{itemize}
If we had to pick the most critical early authors without whom our work would have been impossible, it would be Ville, Wald, Kelly and Robbins; later influences on us have been via Lai, Cover, Shafer, Vovk, Grunwald and the second author's own earlier works~\citep{howard_exponential_2018,howard_uniform_2019}. These authors stand out below.

\paragraph{Probability.} Ville's 1939 PhD thesis~\citep{ville_etude_1939} contained an important and rather remarkable result of its time that connected measure-theoretic probability with betting, and indeed brought the very notion of a martingale into probability theory. In brief, Ville proved that for every event of measure zero, there exists a betting strategy for which a gambler's wealth process (a nonnegative martingale) grows to infinity if that event occurs. For example, the strong law of large numbers (SLLN) and the law of the iterated logarithm (LIL) are two classic measure-theoretic statements that occur on all sequences of observations, except for a null set according to some underlying probability measure (where the two null sets for the two laws are different). Ville proved that it is possible to bet on the next outcome such that if the LIL were false for that particular sequence of observations, then the gambler's wealth would grow in an unbounded fashion.

Doob's monumental papers and book~\cite{doob1953stochastic} in the following decades stripped martingales of their betting roots and presented them as some of the most powerful tools of measure-theoretic probability theory, with applications to many other branches of mathematics. (However, betting could be viewed as instances of ``Doob's martingale transform''.) 
These betting roots were revived in the 1960s with the renewed interest in algorithmic definitions of randomness, due to Kolmogorov, \citet{martin1966definition} and many others. 

More recently, \citet{shafer_probability_2005, shafer2019game} have produced two seminal books that aim bring betting and martingales to the front and center of probability and finance, aiming to derive much (if not all) of probability theory from purely game-theoretic principles based on betting strategies. The product martingale wealth process that appears in our work also appears in theirs (indeed, it is a fundamental process), but Shafer and Vovk did not explore the topics in our paper (confidence sequences, explicit computationally efficient betting strategies, sampling without replacement, thorough numerical simulations, and so on). Indeed, their book has a thorough treatment of probability and finance, but with respect to statistical inference, there is little explicit methodology for practice. Perhaps they were aware of such a statistical utility, but they did not explicitly recognize or demonstrate the excellent power of betting in practice (when properly developed)  for problems such as ours.

\paragraph{Statistical inference.} Using the power of hindsight, we now know that Wald's influential work on the sequential probability ratio test was implicitly based on martingale techniques \cite{wald_sequential_1945}. Wald derived many fundamental results that he required from scratch without having the general language that was being set up by Doob in parallel to his work. In the case of testing a simple null $H_0: \theta = \theta^*$ against a composite alternative $H_0: \theta \neq \theta^*$, \citet[Eq (10:10)]{wald_sequential_1945} suggests forming the likelihood ratio process $\prod_{i=1}^n f_{\theta_{i-1}}(X_i) / \prod_{i=1}^n f_{\theta^*}(X_i)$, where $\theta_{i-1}$ is a mapping from $X_1,\dots,X_{i-1}$ to $\Theta$; in other words, $\theta_{i-1}$ is predictable. In the language of our paper, this is a predictable plug-in, and the first appearance of betting-like ideas in the statistical literature. However, beyond this passing equation in a parametric setup, the idea appears to have lain dormant.

Robbins (along with students and colleagues Siegmund, Darling, and Lai) quickly realized the power of Wald's and Ville's ideas as well as martingales more generally, and pursued a rather broad agenda around sequential testing and estimation, including the introduction and extensive study of confidence sequences and the method of mixtures \citep{darling_iterated_1967,darling_confidence_1967, darling_inequalities_1967, robbins_iterated_1968, robbins_probability_1969, robbins_boundary_1970, robbins_class_1972, robbins_expected_1974, lai_confidence_1976}. Robbins and Siegmund also analyzed Wald's ``betting'' test, and proved in some generality that its behavior is similar to a mixture likelihood ratio test \citep[Section 6]{robbins_expected_1974}. Most of Wald's and Robbins' work was parametric, but Robbins did explicitly study the sub-Gaussian setting in some detail \citep{robbins_statistical_1970}. Building on a vast literature of Chernoff-style concentration inequalities that exploded after Robbins' time, \citet{howard_exponential_2018, howard_uniform_2019} recently extended mixture methods of Robbins to derive confidence sequences under a large class of nonparametric settings using exponential supermartingales. \citet{howard_exponential_2018, howard_uniform_2019} recognized Wald's betting idea, but did not develop it nonparametrically beyond a brief mention in the paper as a direction for future work. The current work takes this natural next step in some thorough detail.

\paragraph{Information and coding theory.} Soon after the seminal work of \citet{shannon1948mathematical}, another researcher at AT\&T Bell Labs, John Larry Kelly Jr. wrote a paper titled ``A New Interpretation of Information Rate'' which explicitly connected betting with the new field of information theory, complementing the work of Shannon \citep{kelly1956new}. In short, he proved that it is possible to bet on the symbols in a communication channel at odds consistent with their probabilities in order to have a gambler's wealth grow exponentially, with the exponent equaling the rate of transmission over the channel. More explicitly, given a sequence of Bernoulli random variables with probability $p > 1/2$, Kelly proved that betting a $(2p-1)$ fraction of your current wealth on the next outcome being 1 is the unique strategy that maximizes the expected log wealth of the gambler. 

When the probability $p$ changes at each step in an unknown manner, the ``universal coding'' work of \citet{krichevsky1981performance} showed that a mixture method involving the Jeffreys prior and maximum likelihood can achieve nearly the optimal wealth in hindsight, with the expected log wealth of their strategy only being worse than the optimal oracle log-wealth by a factor that is logarithmic in the number of rounds; these observations work for any discrete alphabet, not just a binary. Cover's interest in these techniques spans several decades~\citep{cover1974universal,cover1984algorithm,cover1987log,bell1980competitive,bell1988game}, culminating  in his famous universal portfolio algorithm \citep{cover1991universal}, that today forms a standard textbook topic in information theory.

There are other parts of information/coding theory that could be seen as related in some ways to betting through the use of (what are now called) e-variables: these include the topics of prequential model selection and minimum description length; see works by \citet{rissanen1984universal,rissanen1998stochastic}, \citet{dawid1984present,dawid1997prequential}, \citet{grunwald2007minimum,grunwald_safe_2019}, \citet{li1999estimation} and references therein.

\paragraph{Online learning and sequential prediction under log loss.} In the 1990s, the problems studied by Krichevsky, Trofimov, and Cover continued to be extended --- often dropping the information theoretic context --- under the title of sequential prediction under the logarithmic loss. In the active subfield of online learning, the previous results were effectively ``regret bounds'' against potentially adversarial sequences of observations, with a chapter devoted to the problem in the book on prediction, learning and games by \citet{cesa2006prediction}. More recently, Orabona and colleagues such as Pal and Jun have found powerful implications of these ideas in deriving parameter-free algorithms for online convex optimization \citep{orabona2016coin, orabona2017training, jun2017improved, jun2019parameter}. 

\citet{rakhlin2017equivalence} found that deterministic regret inequalities can be used to derive concentration inequalities for martingales, connecting the two rich fields. Later, \citet{jun2019parameter} also derive concentration inequalities using their betting-based regret bounds, with explicit bounds derived in the sub-Gaussian and bounded settings. However, because regret bounds could be tight in rate but are typically loose in constants, the resulting concentration inequalities are not tight in practice. Thus, we view this line of work as important and complementary to our explorations, which are different in their motivation, derivation and practicality.\\

\emph{Typically, none of these lines of literature have cited the others.} For example, the important paper of \citet{rakhlin2017equivalence} does not mention the work of Ville, Wald or Robbins, or even of Vovk and Shafer. Similarly, despite the books of Shafer and Vovk having a wonderful coverage of the history of probability and martingales stemming back hundreds of years, even their recent 2019 book~\citep{shafer2019game} does not cite the coding theory and online learning literature very much, including the works of Orabona and coauthors \citep{orabona2016coin, orabona2017training, jun2017improved, cutkosky2018black, jun2019parameter}, \citet{krichevsky1981performance}, or \citet{rakhlin2017equivalence}. Recent work of Orabona and colleagues also in turn has no mention of the books of \citet{shafer_probability_2005, shafer2019game}, or works of Ville, Wald, Robbins, Howard, their coauthors and other recent authors. The work of \citet{howard_exponential_2018, howard_uniform_2019} does cite the Wald and Robbins literatures, as well as the books of Shafer and Vovk and pioneering work of Ville, but does not form connections to information/coding theory nor to online learning. The excellent book of \citet{cesa2006prediction} does not cite Ville, the seminal martingale works of Robbins, or the 2001 book by Shafer and Vovk. \footnote{Authors like like~\citet{rissanen1984universal,rissanen1998stochastic} and ~\citet{dawid1984present,dawid1997prequential} are not cited in most of these works, perhaps because the connections of their works to betting are indirect.} 

The reason for the lack of intersection of these parallel threads is likely manifold, and definitely far from malicious:
(a) these works were and continue to be published in different literatures, 
(b) these works had different goals in mind, meaning that they were addressing different problems and often using different techniques, 
(c) our understanding of these literatures and their relationships is constantly evolving and  far from complete; it is likely that no author has a command over all these parallel literatures, and indeed this should not be expected.

In the preface of their 2006 book, Cesa-Bianchi and Lugosi write
\begin{quote}
Prediction of individual sequences, the main theme of this book, has been studied in various fields, such as statistical decision theory, information theory, game theory, machine learning, and mathematical finance. Early appearances of the problem go back as far as the 1950s, with the pioneering work of Blackwell, Hannan, and others. Even though the focus of investigation varied across these fields, some of the main principles have been discovered independently. Evolution of ideas remained parallel for quite some time. As each community developed its own vocabulary, communication became difficult. By the mid-1990s, however, it became clear that researchers of the different fields had a lot to teach each other. When we decided to write this book, in 2001, one of our main purposes was to investigate these connections and help ideas circulate more fluently. In retrospect, we now realize that the interplay among these many fields is far richer than we suspected. ... Today, several hundreds of pages later, we still feel there remains a lot to discover. This book just shows the first steps of some largely unexplored paths. We invite the reader to join us in finding out where these paths lead and where they connect.
\end{quote}
Thus it is clear that Cesa-Bianchi and Lugosi already foresaw that there were many connections between the fields that have been unstated, underappreciated, undiscovered and underutilized. The connections we briefly point out above between these literatures, both historical and modern, are themselves new in their own right (not existing in any of the aforementioned books or papers) and may be considered a small contribution of this paper. A  more thorough investigation of these connections may be the topic of a future survey paper, or indeed, a book on these topics.

\end{document}